\documentclass[10pt]{article}
\usepackage{amsmath}
\usepackage{amscd}
\usepackage{amssymb}
\usepackage{graphicx}
\usepackage{amsthm}
\newtheorem{thm}{Theorem}[section]
\newtheorem{prop}[thm]{Proposition}
\newtheorem{lem}[thm]{Lemma}
\newtheorem{cor}[thm]{Corollary}  \theoremstyle{definition}
\newtheorem{df}[thm]{Definition}   \theoremstyle{definition}

\newtheorem{rem}[thm]{Remark}                \theoremstyle{plain}
 \theoremstyle{definition}
\newtheorem{ex}[thm]{Example}   
\def\CC{\Bbb{C}}
\def\RR{\Bbb{R}}  
\def\CCI{\hat{\CC}}        \def\NN{\Bbb{N}} 

\def\B1{{\rm\kern.32em\vrule    width.12em       height1.4ex
depth-.05ex\kern-.28em 1}}
\def\G{\Gamma}
\def\GN{\Gamma ^{\NN }}

\def\g{\gamma }
\def\l{\lambda }

\def\Hol{\text{H\"{o}l}}
\def\emHol{\text{{\em H\"{o}l}}}
\def\Min{\text{Min}}
\def\emMin{\text{\em Min}}

\begin{document}
\title{Random complex dynamics and devil's coliseums  
\footnote{Date: March 13, 2015. Published in Nonlinearity {\bf 28} (2015) 1135-1161.  
This research was partially supported by JSPS KAKENHI 24540211. 
2010 Mathematics Subject Classification. 
37F10, 30D05. Keywords: Rational semigroups, polynomial semigroups, random complex dynamics, 
random iteration, Markov process,    
Julia sets, fractal geometry, (backward) iterated function systems, interaction cohomology,  complex singular functions, 
devil's coliseum, randomness-induced phenomena, cooperation principle.}}
\author{Hiroki Sumi\
\\   Department of Mathematics, Graduate School of Science,\\ 
Osaka University\\   
1-1, Machikaneyama,\ Toyonaka,\  Osaka,\ 560-0043,\ 
Japan\\ E-mail:    sumi@math.sci.osaka-u.ac.jp\\ 
http://www.math.sci.osaka-u.ac.jp/\textasciitilde sumi/}
\date{}
\maketitle
%\vspace{-10mm}
\begin{abstract}
We investigate the random dynamics of polynomial maps on the Riemann sphere $\CCI $ 
and the dynamics of semigroups of polynomial maps on $\CCI .$ 
In particular, the dynamics of a semigroup $G$ of polynomials whose planar postcritical set 
is bounded and the associated random dynamics are studied. In general, the Julia set of such a $G$ 
may be disconnected. We show that if $G$ is such a semigroup, then regarding the associated random dynamics, 
the chaos of the averaged system disappears in the $C^{0}$ sense, and 
the function $T_{\infty }$ of probability of tending to $\infty \in \CCI $ 
is \Hol der continuous on $\CCI $ and varies only on the Julia set of $G$. Moreover, 
the function $T_{\infty }$ has a kind of monotonicity. 
It turns out that $T_{\infty }$ is a complex analogue of the devil's staircase, and 
we call $T_{\infty }$ a ``devil's coliseum.'' We investigate the details of 
$T_{\infty }$ when $G$ is generated by two polynomials. In this case, 
$T_{\infty }$ varies precisely on the Julia set of $G$, which is a thin fractal set. 
Moreover, under this condition, we investigate the pointwise H\"{o}lder exponents of $T_{\infty }$. 
%by using some geometric observations, ergodic theory, potential theory and function theory.  
%In particular, we show that for almost every point $z$ in the Julia set of $G$ with respect to 
%an invariant measure, $T_{\infty }$ is not differentiable at $z.$ 
%We find many new phenomena of random complex dynamics which cannot hold in the usual 
%iteration dynamics of a single polynomial, and we systematically investigate them. 

\end{abstract}
%\vspace{-5mm} 
\section{Introduction}
\label{Introduction}
 Some results of this paper have been announced in \cite{S11, Ssugexp} without proofs. 
%which is a proceedings paper of a conference. 

In this paper, we simultaneously investigate the random dynamics of polynomial maps on the Riemann sphere $\CCI $ and the dynamics 
of polynomial semigroups (i.e., semigroups of non-constant polynomial maps 
where the semigroup operation is functional composition) on $\CCI .$ 
%We see that the both fields are related to each other very deeply.
%In fact, we develop both theories simultaneously.  

%One motivation for research in complex dynamical systems is to describe some
% mathematical models on ethology. For 
% example, the behavior of the population 
% of a certain species can be described by the 
% dynamical system associated with iteration of a polynomial 
% $f(z)= az(1-z)$ 
% such that $f$ preserves the unit interval and 
% the postcritical set in the plane is bounded 
% (cf. \cite{D}). However, when there is a change in the natural environment,  
%some species have 
% several strategies to survive in nature. 
%From this point of view, 
% it is very natural and important not only to consider the dynamics 
%of iteration, where the same survival strategy (i.e., function) is repeatedly applied, but also 
%to consider random 
% dynamics, where a new strategy might be applied at each time step.  
The first study of random complex dynamics was given by J. E. Fornaess and  N. Sibony (\cite{FS}). 
For the motivations to study random complex dynamics, see \cite{Splms10, Scp}. 
For research on random complex dynamics of quadratic polynomials, 
see \cite{BBR,  GQL}.  
For research on random dynamics of polynomials (of general degrees) 
with bounded planar postcritical set, see the author's works \cite{SdpbpI, SdpbpII, SdpbpIII, SU4, Ssugexp}. 
 In \cite{Splms10, Scp}, the author of this paper discussed more general 
random dynamics of rational maps with a systematic approach. 
 
%In order to investigate the random dynamics of a family of rational maps, 
%it is important to consider the dynamics of associated rational semigroups 
%(i.e., semigroups of non-constant rational maps). 
The first study of dynamics of polynomial semigroups was 
conducted by
A. Hinkkanen and G. J. Martin (\cite{HM}),
who were interested in the role of the
dynamics of polynomial semigroups 
%(i.e., semigroups of non-constant polynomial maps) 
while studying
various one-complex-dimensional
moduli spaces for discrete groups,
and
by F. Ren's group (\cite{GR}), 
 who studied 
such semigroups from the perspective of random dynamical systems.
Since the Julia set $J(G)$ (the set of non-normality) of a finitely generated polynomial semigroup 
$G$ generated by $\{ h_{1},\ldots ,h_{m}\} $ has 
``backward self-similarity,'' i.e.,  
$J(G)=\bigcup _{j=1}^{m}h_{j}^{-1}(J(G))$ (see \cite[Lemma 1.1.4]{S1}),  
the study of the dynamics of rational semigroups can be regarded as the study of  
``backward iterated function systems,'' and also as a generalization of the study of 
self-similar sets in fractal geometry.  
For recent work on the dynamics of polynomial semigroups, 
see 
%the author's papers 
\cite{S1}--\cite{Scp},  
%and 
\cite{SS, SU2, SU4, SU5}. 

 In order to consider the random dynamics of a family of polynomials on $\CCI $, for each $z\in \CCI $, 
 let $T_{\infty }(z)$ be the probability of tending to $\infty \in \CCI $ 
starting with the initial value $z\in \CCI .$ 
Note that in the usual iteration dynamics of a single polynomial $f$ with $\deg (f)\geq 2$, 
the function $T_{\infty }$ is equal to the constant $1$ in the basin of infinity, 
and $T_{\infty }$ is equal to the constant $0$ in the filled-in Julia set of $f$. 
Thus $T_{\infty }$ is not continuous at any point in the Julia set of $f.$ 
However, we see the following main results of this paper.\\ 
%in this paper, we see that 
{\bf Main Results (rough statements).}
\vspace{-2mm} 
\begin{itemize}
\item[(I)] 
If the planar postcritical set (see section~\ref{Main}) of the associated polynomial semigroup $G$ of 
a random dynamical system of complex polynomials is bounded 
and the Julia set of $G$ is disconnected, then 
the ``Julia set'' and the chaos of the averaged system disappears in the ``$C^{0}$'' sense,    
the function 
$T_{\infty }: \CCI \rightarrow [0,1]$ is \Hol der continuous on $\CCI $ 
(i.e., there exist constants $C>0$ and $0<\alpha <1$ such that 
$|T_{\infty }(z_{1})-T_{\infty }(z_{2})|\leq Cd(z_{1},z_{2})^{\alpha }$ for each 
$z_{1},z_{2}\in \CCI $),  and 
$T_{\infty }$ has a kind of monotonicity (e.g., if $J_{1}$ and $J_{2}$ are two 
connected components of the Julia set of $G$ and $J_{1}$ is included in 
a bounded connected component of $\CC \setminus J_{2}$, then 
$\max _{z\in J_{1}}T_{\infty }(z)\leq \min _{z\in J_{2}}T_{\infty }(z)$).  
(For the precise statement, see Theorem~\ref{randomthm1}.) 
\item[(II)]
Under certain conditions $T_{\infty }$ has 
some singular properties (for instance, it varies only on a thin fractal set, the so-called  
Julia set of the associated polynomial semigroup $G$, 
and for almost every point $z_{0}$ in the Julia set of 
$G$ with respect to a nice``invariant measure'', the function $T_{\infty }$ is not differentiable 
at $z_{0}$), and this function is a complex analogue of 
the devil's staircase (Cantor function) or Lebesgue's singular functions 
(see Theorems~\ref{t:hnondiffp}, \ref{t:2gengdis},  
%\ref{t:3genmain}, 
Example~\ref{ex:dc1}). 
%figures in \cite{Splms10}). 
%Figures~\ref{fig:dcjulia}, \ref{fig:dcgraphgrey2}, and \ref{fig:dcgraphudgrey2}). 
%Moreover, the ``Julia set'' of the averaged system is empty. 
(For the definition of the devil's staircase and Lebesgue's singular functions, see \cite{YHK}.) 
%For the figure of the graph of $T_{\infty }$, see figures in \cite{Splms10}. 
Also, the pointwise \Hol der exponents of $T_{\infty }$ are investigated. 
Note that we do not assume hyperbolicity in Theorem~\ref{t:hnondiffp}, 
while in the previous result \cite[Theorem 3.82]{Splms10}, it is assumed that 
$G$ is hyperbolic. 
\end{itemize}
Graphs of $T_{\infty }$ are illustrated in \cite{Splms10}. 
Thus even though the chaos of the averaged system disappears, the system has new kind of complexity. 
These are new phenomena which 
cannot hold in the usual iteration dynamics of a single polynomial. 
%\begin{figure}[htbp]
%\caption{(From left to right) The graphs of the devil's staircase and Lebesgue's singular function.}
%\ \ \ \ \ \ \ \ \ \ \ \ \ \ \ \ \ \ \ \ \ \ \ \ \ \ \ \ \ \ \ 
%\ \ \ \ \ \ \ \ \ \ \ \ \ \ \ \ \ \ \ \  
%\begin{minipage}{.2\linewidth}     
%\includegraphics[width=1.4cm,height=1.4cm, origin =c, angle =-90]{dsgraph1-1.eps}
%\label{fig:dsgraph1-1}
%\end{minipage}
%\begin{minipage}{.2\linewidth}
%\includegraphics[width=1.4cm,height=1.4cm, origin =c, angle =-90]{lebfcnmaple025.eps}\label{fig:lebfcn}
%\end{minipage}
%\begin{minipage}{.2\linewidth}
%\includegraphics[width=2.1cm,height=2.1cm, origin =c, angle =-90]{takagifcnmaple1.eps}
%\end{minipage}
%\end{figure}
Such phenomena in random dynamical systems are called ``{\bf randomness-induced phenomena}''. 
To explain the detail of the above result, we first remark that 
these well-known singular functions (the devil's staircase and Lebesgue's singular functions) defined on $[0,1]$ can be redefined by using random dynamical systems 
on $\RR $ as follows (see \cite{Splms10, Ssugexp}).  
Let $f_{1}(x):= 3x, f_{2}(x):=3(x-1)+1\ (x\in \RR )$ and 
we consider the random dynamical system 
%(random walk) 
on $\RR $ such that at every step we choose $f_{1}$ with probability $1/2 $ and $f_{2}$ with probability $1/2.$ 
We set $\hat{\RR }:= \RR \cup \{ \pm \infty \} .$ We denote by $T_{+\infty }(x)$ the 
probability of tending to $+\infty \in \hat{\RR }$ starting with the initial value $x\in \RR .$   
Then, we can see that the function $T_{+\infty }|_{[0,1]}:[0,1]\rightarrow [0,1]$ is equal to the devil's staircase. 
 Similarly, let $g_{1}(x):= 2x, g_{2}(x):= 2(x-1)+1\ (x\in \RR )$ and let 
 $0<a<1$ be a constant. We consider the random dynamical system on $\RR $ such that 
 at every step we choose the map $g_{1}$ with probability $a$ and the map $g_{2}$ with 
probability $1-a.$ Let $T_{+\infty ,a}(x)$ be the probability of tending to 
$+\infty $ starting with the initial value $x\in \RR .$ 
Then, we can see that the function $T_{+\infty ,a}|_{[0,1]}: [0,1]\rightarrow [0,1]$ is 
equal to Lebesgue's singular function $L_{a}$ with respect to the parameter $a$ provided $a\neq 1/2.$  
From the above point of view, the function $T_{\infty }:\CCI \rightarrow [0,1]$ is a 
complex analogue of the devil's staircase and Lebesgue's singular functions. 
We call $T_{\infty }$ a ``{\bf devil's coliseum}''(\cite{Splms10, Scp}) .  

We also explain why we focus on polynomial semigroups with bounded planar postcritical set. 
A polynomial semigroup $G$ is said to be postcritically bounded if the planar postcritical set of $G$ 
is bounded. 
It is well-known that if $g\in {\cal P}$, where 
${\cal P}$ denotes the set of polynomials of degree two or more, then $J(g)$ is connected if and only if 
the semigroup $\{ g^{n}\mid n\in \NN \} $ is postcritically bounded.  
However, we remark that there are many examples of elements of 
postcritically bounded polynomial semigroup $G$ with $G\subset {\cal P}$  
%where ${\cal P}$ denotes the set of polynomials of degree two or more, 
such that the Julia set of $G$ is disconnected  
(see section~\ref{Examples}, \cite{SdpbpI,SdpbpIII}). 
In fact, it is easy to construct such examples by using (\ref{eq:pg}) in section~\ref{Main}, 
and 
many systematic studies on the dynamics of 
postcritically bounded polynomial semigroups $G$ with $G\subset {\cal P}$ are given 
in \cite{SdpbpI,SdpbpII,SdpbpIII, SS, S15}.  
Thus we are very interested in the new phenomena regarding  
the dynamics of postcritically bounded polynomial semigroups. 
%It is very natural to ask {\bf ``what happens for a $G\in {\cal G}_{dis}$ and the associated random dynamics?''  
%``How can we classify the elements $G$ in ${\cal G}_{dis}$ in terms of the dynamics of $G$ 
%and the associated random dynamics?'' }   
 
%For a polynomial semigroup $G$ with $\infty \in F(G),$ we denote by 
%$F_{\infty }(G)$ the connected component of $F(G)$ containing $\infty .$  
%Note that if $G$ is generated by a compact subset of ${\cal P}$, then 
%$\infty \in F(G).$ 
%For a polynomial $g\in {\cal P}$, we set $F_{\infty }(g):=F_{\infty }(\langle g\rangle ).$ 
One of the purposes of this paper 
is to combine the study of the dynamics of 
postcritically bounded polynomial semigroups with disconnected Julia set and the study of random dynamics of polynomials. 
To prove Main Result (I) (Theorem~\ref{randomthm1}),   
we need the following result from \cite[Theorem 3.15]{Splms10} and 
\cite[Theorem 1.9]{Scp}: 
For a random dynamical system of complex polynomials generated by finitely many 
elements in ${\cal P}$, 
if the kernel Julia set $J_{\ker }(G):=\cap _{g\in G_{\tau }}g^{-1}(J(G))$ 
of the associated polynomial semigroup $G$ is empty, 
then $T_{\infty }$ is \Hol der continuous on $\CCI $ and there exists a finite dimensional subspace 
$U$ of the space $C(\CCI )$ of continuous functions on $\CCI $ 
with $M(U)=U$, where $M$ denotes the transition operator of 
the system (see section~\ref{Main}),  
and a bounded operator $\pi :C(\CCI )\rightarrow U$ such that 
$M^{n}(\varphi -\pi (\varphi ))\rightarrow 0$ in $C(\CCI )$ as $n\rightarrow \infty $ for each 
$\varphi \in C(\CCI ).$ Therefore, to prove Main Result (I) (Theorem~\ref{randomthm1}),  
it is an important key to prove that if 
$G$ with $G\subset {\cal P}$ is postcritically bounded and the Julia set of $G$ is disconnected, 
 then $J_{\ker }(G)=\emptyset $, 
which is proved in Lemma~\ref{l:pbjkerem} of this paper.  
In order to prove the monotonicity of $T_{\infty }$ and statements \ref{randomthm1tca} and 
\ref{randomthm1q} of Theorem~\ref{randomthm1}, we combine the idea from \cite{Splms10} and 
 new careful observations on the dynamics of postcritically bounded 
 polynomial semigroup $G$ with disconnected Julia set.  

Main Result (II)  (Theorem~\ref{t:hnondiffp}) means that even though the chaos of the averaged system 
disappears in the $C^{0} $ sense as in Main Result (I),   
it can remain in the $C^{\alpha }$ sense with some $\alpha \in (0,1)$,  
where $C^{\alpha }$ denotes the space of 
$\alpha $-H\"{o}lder continuous functions. 
%In \cite{Scp} it is shown that $T_{\infty }$ is H\"{o}lder continuous on $\CCI $ with some exponent.   
From these, we can say that we have a gradation between chaos and order. 
In the proof (section~\ref{Proofs}) of Main Result (II) (Theorem~\ref{t:hnondiffp}),    
we use Birkhoff's ergodic theorem, potential theory, the Koebe distortion theorem, 
and some observations 
(\cite{SdpbpI}) 
about the space of all connected components of $J(G)$ and the Julia set of the associated real affine semigroup.    
%from \cite{SdpbpI}. 

We also prove several results on $2$- or $3$-generator 
postcritically bounded polynomial 
semigroups with disconnected Julia set and associated 
random dynamics (see Theorem~\ref{t:2gengdis}, \ref{t:3genmain}, Corollary~\ref{c:3gengdis}, Remark~\ref{r:3gendet}). 
In order to prove the results 
 on $2$- or $3$-generator postcritically bounded polynomial semigroups $G$ with disconnected Julia set  and associated random dynamics, 
we need the idea of the nerves of backward images of $J(G)$ under elements of $G$ and their 
inverse limit from \cite{S15}, which are related to certain kind of cohomology groups introduced by the author.

In section~\ref{Main}, we give the details of the main results. 
In section~\ref{Background}, we explain the known results and tools to prove the main results. 
In section~\ref{Proofs}, we prove the main results. In section~\ref{Examples}, we give some examples. \\ 
\noindent {\bf Acknowledgment:} The author thanks Rich Stankewitz for valuable comments. 
The author also thanks the referee for checking the manuscript carefully and giving the author 
some valuable comments.  
%\vspace{-5mm} 
\section{Main results}
\label{Main}

In this section, we give the details of the main results. 
%We use notations and definitions in section~\ref{Introduction}. 

A 
{\bf polynomial semigroup } is a 
semigroup generated by a family of non-constant 
polynomial maps on the Riemann sphere $\CCI $ with the semigroup operation 
being functional composition(\cite{HM,GR}). 
%\begin{df} 
We set 
%Rat : $=\{ h:\CCI \rightarrow \CCI \mid 
%h \mbox { is a non-constant rational map}\} $
${\cal P}:= \{ g:\CCI \rightarrow \CCI \mid 
g \mbox{ is a polynomial}, \deg (g)\geq 2\} $
endowed with the distance $\kappa $ which is defined 
by $\kappa (f,g):=\sup _{z\in \CCI }d(f(z),g(z))$, where $d$ denotes the 
spherical distance on $\CCI .$   
%We set 
%Poly :$=\{ h:\CCI \rightarrow \CCI 
%\mid h \mbox{ is a non-constant polynomial }\} $ endowed with 
%the relative topology from Rat.   
%Moreover, we set 
%$\Ratp:=\{ h\in \mbox{Rat}\mid \deg (h)\geq 2\} $ endowed with the 
%relative topology from Rat. 
%Furthermore, we set 
%Poly$_{\deg \geq 2}
%:= \{ g\in \mbox{Poly}\mid \deg (g)\geq 2\} $ 
%${\cal P}:= \{ g:\CCI \rightarrow \CCI \mid 
%g \mbox{ is a polynomial}, \deg (g)\geq 2\} $
%endowed with the relative topology from 
%Rat. 
Note that $g_{n}\rightarrow g$ in ${\cal P}$ if and only if 
(i) $\deg (g_{n})=\deg (g)$ for each large $n$, and 
(ii) the coefficients of $g_{n}$ converge appropriately 
to the coefficients of $g$ (\cite{Be}). 
%Note also that ${\cal P}$ and $\Ratp$ are closed subspaces of $\Rat .$ 
Also, setting ${\cal P}_{n}:=\{ g\in {\cal P} \mid \deg(g)=n\} $ for each 
$n\geq 2$, we have that 
${\cal P}_{n}$ is a connected, open and closed subset of ${\cal P}$, 
${\cal P}_{n}$ is a connected component of ${\cal P}$, and 
${\cal P}_{n}\cong (\CC \setminus \{ 0\} )\times \CC ^{n}$ (\cite{Be}).   
For a polynomial semigroup $G$, we denote by $F(G)$ the Fatou set of $G$, which is defined to be  
the maximal open subset of $\CCI $ where $G$ is equicontinuous with respect to 
the spherical distance on $\CCI $ (for the definition of equicontinuity, see \cite[Definition 3.11]{Be}). 
We call $J(G):=\CCI \setminus F(G)$ the Julia set of $G.$ 
For fundamental properties on the Fatou sets and Julia sets, see \cite{HM, S3}. 
The Julia set is backward invariant under each element $h\in G$, but 
might not be forward invariant. This is a difficulty of the theory of rational semigroups. 
Nevertheless, we ``utilize'' this to investigate the associated random complex dynamics.     
%\end{df}
For a non-empty subset $\Lambda $ of ${\cal P}$, 
we denote by $\langle \Lambda \rangle $ the polynomial semigroup 
generated by $\Lambda .$ Thus $\langle \Lambda \rangle =\{ h_{1}\circ \cdots \circ h_{m}\mid 
m\in \NN , h_{1},\ldots ,h_{m}\in \Lambda \} .$ For finitely many polynomial maps $g_{1},\ldots ,g_{m}$, 
we denote by  $\langle g_{1},\ldots ,g_{m}\rangle $ the polynomial semigroup generated by 
$\{ g_{1},\ldots ,g_{m}\} .$ For a polynomial map $g$, we set $F(g):= F(\langle g\rangle )$ and  
$J(g):= J(\langle g\rangle )$.
For a polynomial semigroup $G$, we set $G^{\ast }:= G\cup \{ \mbox{Id}\} $, 
where Id denotes the identity map. For a polynomial semigroup $G$ and a subset $A$ of $\CCI $, 
we set $G(A):= \bigcup _{g\in G}g(A)$ and $G^{-1}(A):= \bigcup _{g\in G}g^{-1}(A).$ 
  
For a polynomial semigroup $G$,    
we set $\hat{K}(G):=\{ z\in \CC \mid G(\{ z\} ) \mbox{ is bounded in } \CC \} .$ 
This is called the {\bf smallest filled-in Julia set} of $G.$  
For a polynomial $g\in {\cal P}$, we set 
$K(g):= \hat{K}(\langle g\rangle ).$ 
For a polynomial semigroup $G$, we set 
$P(G):= \overline{\bigcup _{g\in G}\{ z\in \CCI \mid z \mbox{ is a critical value of } g:\CCI \rightarrow \CCI \} } .$ 
This is called the {\bf postcritical set} of $G.$ 
Note that if $G=\langle \Lambda \rangle $, 
then 
\begin{equation}
\label{eq:pg}
P(G)=\overline{G^{\ast }(\bigcup _{h\in \Lambda }
\{ z\in \CCI \mid z \mbox{ is a critical value of }h:\CCI \rightarrow \CCI \} )}.
\end{equation}
Thus for each $g\in G$, $g(P(G))\subset P(G).$ 
For a polynomial semigroup $G$, we set 
$P^{\ast }(G):= P(G)\setminus \{ \infty \} .$ This is called the {\bf planar postcritical set} of $G.$ 
A polynomial semigroup $G$ is said to be {\bf postcritically bounded} if 
$P^{\ast }(G)$ is bounded in $\CC .$ 
We denote by ${\cal G}$ the set of all postcritically bounded polynomial semigroups $G$ 
with $G\subset {\cal P}.$ Moreover, we set ${\cal G}_{dis}:= \{ G\in {\cal G}\mid 
J(G) \mbox{ is disconnected}\} .$ 
It is well-known that if $g\in {\cal P}$, then $J(g)$ is connected if and only if 
$\langle g\rangle \in {\cal G}.$  
However, we remark that there are many examples of elements of ${\cal G}_{dis}$ 
(see section~\ref{Examples}, \cite{SdpbpI,SdpbpIII}). In fact, it is easy to construct such examples by using (\ref{eq:pg}), and 
many systematic studies on the dynamics of semigroups $G$ in ${\cal G}$ or ${\cal G}_{dis}$ are given 
in \cite{SdpbpI,SdpbpII,SdpbpIII,SS, S15}.  
Thus we are very interested in the new phenomena on ${\cal G}_{dis}$. 
It is very natural to ask {\bf ``what happens for a $G\in {\cal G}_{dis}$ and the associated random dynamics?''  
``How can we classify the elements $G$ in ${\cal G}_{dis}$ in terms of the dynamics of $G$ 
and the associated random dynamics?'' }   
 
For a polynomial semigroup $G$ with $\infty \in F(G),$ we denote by 
$F_{\infty }(G)$ the connected component of $F(G)$ containing $\infty .$  
Note that if $G$ is generated by a compact subset of ${\cal P}$, then 
$\infty \in F(G).$ 
For a polynomial $g\in {\cal P}$, we set $F_{\infty }(g):=F_{\infty }(\langle g\rangle ).$

For a non-empty subset $A$ of $\CCI $ and a point $z\in \CCI $, 
we set $d(z,A):= \inf _{a\in A}d(z,a)$, where $d$ is the spherical distance. 
For a non-empty subset $A$ of $\CCI $ and a positive number $r$, we set 
$B(A,r):= \{ z\in \CCI \mid d(z,A)<r\} .$ 
For a non-empty subset $A$ of $\CC $, we set 
$d_{e}(z,A):= \inf _{a\in A}|z-a|.$ 
%For a non-empty subset $A$ of $\CC $ and a 
For a non-empty subset $A$ of $\CC $ and 
a positive number $r$, we set 
$D(A,r):= \{ z\in \CCI \mid d_{e}(z,A)<r\} .$ 

For a metric space $X$, let ${\frak M}_{1}(X)$ be the space of all 
Borel probability measures on $X$ endowed with the topology 
induced by the weak convergence (thus $\mu _{n}\rightarrow \mu $ in ${\frak M}_{1}(X)$ if and only if 
$\int \varphi d\mu _{n}\rightarrow \int \varphi d\mu $ for each bounded continuous function $\varphi :X\rightarrow \RR $). 
Note that if $X$ is a compact metric space, then ${\frak M}_{1}(X)$ is compact and metrizable. 
For each $\tau \in {\frak M}_{1}(X)$, we denote by supp$\, \tau $ the topological support of $\tau .$  
Let ${\frak M}_{1,c}(X)$ be the space of all Borel probability measures $\tau $ on $X$ such that supp$\,\tau $ is 
compact.     

%\begin{enumerate}
%\item 
Let 
$\tau \in {\frak M}_{1}({\cal P} ).$  
In the following, we consider the independent and identically-distributed random dynamical system on $\CCI $ 
such that at every step we choose a polynomial map according to the probability distribution $\tau $ 
(\cite{Splms10}).  
This determines a time-discrete Markov process with time-homogeneous transition probabilities 
on the phase space 
$\CCI $ such that for each $x\in \CCI $ and each Borel subset $A$ of $\CCI $, 
the transition probability $p(x,A)$ from $x$ to $A$ is equal to $\tau (\{ g\in {\cal P}\mid g(x)\in A\} ).$  
%\item 
We set $\G _{\tau }:= \mbox{supp}\,\tau$ and $X_{\tau }:=($supp $\tau )^{\NN }$. 
%where  
%supp $\tau $ denotes the topological support of $\tau $ in ${\cal P}.$   
%\item 
We set $\tilde{\tau }:=\otimes _{j=1 }^{\infty }\tau $.  
This is the unique Borel probability measure on ${\cal P} ^{\NN }$ such that,  
for each $n\in \NN $, if $A_{1},A_{2},\ldots ,A_{n}$ are Borel subsets of ${\cal P}$, then 
$\tilde{\tau }(A_{1}\times A_{2}\times \cdots \times A_{n}\times {\cal P} \times {\cal P} \cdots )=
\prod _{j=1}^{n}\tau (A_{j}).$ Note that supp$\, \tilde{\tau }=X_{\tau }.$   
%\item 
Let $G_{\tau }$ be the polynomial semigroup generated by the polynomials contained in  supp $\tau $.   
%\item 
We set $C(\CCI ):=\{ \varphi :\CCI \rightarrow \CC \mid 
\varphi $ is continuous$\} $ endowed with the supremum norm.  
We define an operator 
$M_{\tau }:C(\CCI )\rightarrow 
C(\CCI )$ by  
$M_{\tau }(\varphi )(z):=
\int _{{\cal P}}\varphi (g(z))\ d\tau (g)$.
This $M_{\tau }$ is called the transition operator of the random dynamical system associated with 
$\tau .$  
Moreover, we denote by $M_{\tau }^{\ast }:{\cal M}_{1}(\CCI )
\rightarrow {\cal M}_{1}(\CCI )$ the dual of $M_{\tau }$ 
(thus $\int _{\CCI }\varphi (z) d(M_{\tau }^{\ast }(\mu ))(z)=\int_{\CCI } M_{\tau }(\varphi )(z)d\mu (z)$ 
for each $\mu \in {\frak M}_{1}(\CCI ), \varphi \in C(\CCI )$).   
%\item 
Note that for each $z\in \CCI $, 
$M_{\tau }^{\ast }(\delta _{z})=\int _{{\cal P} }\delta _{g(z)} d\tau (g).$ 
Hence $M_{\tau }^{\ast }$ can be regarded as the {\bf averaged map} of elements of supp$\,\tau $ with respect to $\tau .$ 
We denote by $F_{meas }(\tau )$ 
the set of all 
$\mu \in {\cal M}_{1}(\CCI )$ satisfying the following:   
There exists a neighborhood $B$ of $\mu $ in ${\cal M}_{1}(\CCI )$
such that 
 $\{ (M_{\tau }^{\ast })^{n}:B\rightarrow 
{\cal M}_{1}(\CCI )\} _{n\in \NN }$ is equicontinuous on $B$.   
Moreover, we set $J_{meas}(\tau ):= {\cal M}_{1}(\CCI )\setminus 
F_{meas}(\tau )$.   
%\item 
We remark that if $h\in {\cal P} $ and   
$\tau =\delta _{h}$ (the Dirac measure at $h$), then  
$J_{meas }(\tau )\neq \emptyset $. In fact,  
by embedding $\CCI $ into ${\cal M}_{1}(\CCI )$ 
under the map 
$z\mapsto \delta _{z}$, 
we have 
$J(h)\subset J_{meas}(\tau )$. 
However, we will see later that for any $\tau \in {\frak M}_{1,c}({\cal P})$ with $G_{\tau }\in {\cal G}_{dis},$ 
$J_{meas}(\tau )=\emptyset $ (Theorem~\ref{randomthm1}).   

Let $G$ be a rational semigroup. 
We say that a non-empty compact subset $K$ of $\CCI $ is a 
minimal set for $(G, \CCI )$ if $K$ is minimal in 
the space 
$\{ L\mid L\mbox{ is a non-empty compact subset of }\CCI , \forall g\in G, g(L)\subset L\} $ 
with respect to inclusion. 
We set $\Min (G,\CCI ):= \{ K\mid K\mbox { is a minimal set for } (G,\CCI )\} .$ 
Note that by Zorn's lemma, $\Min(G,\CCI )\neq \emptyset .$  
%\item 
For any $\gamma =(\gamma _{1},\gamma _{2},\ldots ,)\in {\cal P} ^{\NN }$ and 
any $n,m\in \NN $ with $n>m$, we set 
$\gamma _{n,m}:=\gamma _{n}\circ \cdots \circ \gamma _{m}.$ 
Let $\tau \in {\frak M}_{1}({\cal P})$ and let $A$ be a non-empty subset of 
$\CCI .$ 
For any $z\in \CCI $, we set  
$T_{A ,\tau }(z):=
\tilde{\tau }(\{\gamma =(\g _{1},\g _{2},\ldots )\in 
{\cal P} ^{\NN } \mid d( \g _{n,1}(z), A)\rightarrow 
0 ,\mbox{as }n\rightarrow \infty \} ).$  
This is nothing else but the {\bf probability of tending to $A $ 
starting with the initial value $z\in \CCI $}   
regarding the random dynamics on $\CCI $ such that at every step we choose a polynomial 
according to $\tau .$  
Moreover, for a point $a\in \CCI $, wet set 
$T_{a,\tau }(z):= T_{\{ a\} ,\tau }(z).$ 
Note that if $G\subset {\cal P}$, then $\{ \infty \} $ is a minimal set for $(G,\CCI ).$ 
Note also that  by \cite[Lemma 5.27]{Splms10}, 
if $\tau \in {\frak M}_{1}({\cal P})$ and if $\infty \in F(G_{\tau })$, then for each connected component $U$ of $F(G_{\tau })$, 
the function $T_{\infty ,\tau}|_{U}$ is  constant (the constant value depends on $U$). 
The main purpose of this paper is to show that 
 if $\tau \in {\cal M}_{1,c}({\cal P})$ satisfies that 
 $G_{\tau }\in {\cal G}_{dis}$, then under certain conditions 
 the function $T_{\infty ,\tau }$ can be regarded as a complex analogue of 
 the devil's staircase. 
%\end{enumerate}
%\begin{rem}
%\label{r:Tconst}
%  For the proof, see   
%\end{rem}
%\begin{rem}
%\label{r:jmne}
%\end{rem}
The following, which was introduced by the author in \cite{Splms10}, 
is the key to investigating the dynamics of rational semigroups and the random complex dynamics. 
\begin{df}[\cite{Splms10}]
%[\cite{Srcd, S10,S11}]
Let $G$ be a rational semigroup. We set  
$J_{\ker}(G):= \bigcap _{g\in G} g^{-1}(J(G))$ and  
this is called the {\bf kernel Julia set} of $G.$ 
\end{df}
%For any $f\in \Ratp $, $J_{\ker }(\langle f\rangle )=J(f)\neq \emptyset .$ However, 
%we will see that for any $G\in {\cal G}_{dis}$, $\emptyset =J_{\ker }(G)\neq J(G)$ (Theorem~\ref{t:pbgen}). 
\begin{df}[\cite{SdpbpI}]
For any connected sets $K_{1}$ and 
$K_{2}$ in $\CC ,\ $  we write $K_{1}\leq _{s}K_{2}$ to indicate that 
$K_{1}=K_{2}$, or $K_{1}$ is included in 
a bounded component of $\CC \setminus K_{2}.$ Furthermore, 
$K_{1}<_{s}K_{2}$ indicates $K_{1}\leq _{s}K_{2}$ 
and $K_{1}\neq K_{2}.$ 
Moreover, $K_{2}\geq _{s}K_{1}$ indicates $K_{1}\leq _{s}K_{2}$,  
and $K_{2}>_{s}K_{1}$ indicates $K_{1}<_{s}K_{2}.$  
Note that 
$ \leq _{s}$ is a partial order in 
the space of all non-empty compact connected 
sets in $\CC .$ This $\leq _{s}$ is called 
the {\bf surrounding order.} 
\end{df}

\begin{rem}
\label{r:icjto}
For a topological space $X$, we denote by $\mbox{Con}(X)$ the set of all connected components of $X.$ 
Let $G\in {\cal G}_{dis}.$ 
In \cite{SdpbpI}, it was shown that 
$J(G)\subset \CC $, $(\mbox{Con}(J(G)), \leq _{s})$ is totally ordered, 
there exists a unique maximal element $J_{\max }=J_{\max}(G)\in (\mbox{Con}(J(G)),\leq _{s})$, there exists 
a unique minimal element $J_{\min }=J_{\min }(G)\in (\mbox{Con}(J(G)),\leq _{s})$, each element of $\mbox{Con}(F(G))$ 
is either simply connected or doubly connected, 
and the connected component $F_{\infty }(G)$ of $F(G)$ with $\infty \in F_{\infty }(G)$ 
is simply connected. 
%and $F_{\infty }(G)$ is simply connected. 
Moreover, in \cite{SdpbpI}, it was shown that
${\cal A}\neq \emptyset $, where  
${\cal A}$ denotes the set of all doubly connected components of $F(G)$  
(more precisely, for each $J,J'\in \mbox{Con}(J(G))$ with 
$J<_{s}J'$, there exists an $A\in {\cal A}$ with $J<_{s}A<_{s}J'$), 
$\bigcup _{A\in {\cal A}}A\subset \CC $, and $({\cal A},\leq _{s})$ is totally 
ordered. Note that each $A\in {\cal A}$ is bounded and multiply connected, while for a single $f\in {\cal P}$, 
we have no bounded multiply connected component of $F(f).$   
\end{rem}
%Recalling Remark~\ref{r:icjto}, 
We now present the main results of this paper. 
\begin{thm}
\label{randomthm1} 
Let $\tau \in {\frak M}_{1,c}({\cal P})$. Suppose that  
$G_{\tau }\in {\cal G}_{dis}$.  
Then, all of the following \ref{randomthm1c}--\ref{randomthm1-6} hold. 
\begin{enumerate}
\item \label{randomthm1c}
{\em (}{\bf \Hol der Continuity}{\em )} The function $T_{\infty ,\tau}:\CCI \rightarrow 
[0,1]$ is \Hol der continuous on $\CCI $,  $M_{\tau }(T_{\infty ,\tau})=T_{\infty ,\tau}$ and 
$T_{\infty ,\tau }(J(G_{\tau }))=[0,1].$ 
\item \label{randomthm1-1}
For each  $U\in \mbox{{\em Con}}(F(G_{\tau }))$,  
there exists a constant $C_{U}\in [0,1]$ such that  
$T_{\infty ,\tau}|_{U}\equiv C_{U}.$
\item \label{randomthm1m} 
{\em (}{\bf Monotonicity}{\em )}  
Let ${\cal A}:=\{ U\in \mbox{{\em Con}}(F(G_{\tau }))\mid U \mbox{ is  doubly connected} \} $. 

\begin{enumerate}
\item If $A_{1},A_{2}\in {\cal A}$ and $\ A_{1}<_{s}A_{2}$, then  
$C_{A_{1}}<C_{A_{2}}.$ In particular, all elements of   
$\{C_{A}\mid A\in {\cal A}\} $ are mutually distinct.  
\item If $J_{1},J_{2}\in $ {\em Con}$(J(G_{\tau }))$ and  
$J_{1}< _{s}J_{2}$, then  
$\sup _{z\in J_{1}}T_{\infty ,\tau}(z)\leq 
\inf _{z\in J_{2}}T_{\infty ,\tau}(z)$.  
\end{enumerate}
\item \label{randomthm1tca}
For each $A\in {\cal A}$,  
$T_{\infty ,\tau}|_{\hat{K}(G_{\tau })}\equiv 0
<C_{A}<1\equiv C_{F_{\infty }(G_{\tau })}.$
%Here we denote by $F_{\infty }(G_{\tau })$
%the connected component of $F(G_{\tau })$ containing $\infty .$ 
\item \label{randomthm1q} 
Let $Q$ be an open subset of $\CCI $.  
If  
$Q\bigcap \left( \bigcup _{A\in {\cal A}}\partial A
\bigcup \partial (F_{\infty }(G_{\tau }))
\bigcup \partial (\hat{K}(G_{\tau }))\right) \neq \emptyset $, 
then  
$T_{\infty ,\tau}|_{Q}$ is not constant.  
%(From (1),$\ldots $,(5), 
%$T_{\infty ,\tau}$ looks like the ``devil's staircase''.  
%Such a function is called a {\bf devil's coliseum}. 
%See Figures~\ref{fig:dcgraphgrey2},
%\ref{fig:dcgraphudgrey2}.)  
\item \label{randomthm1-5-1} 
We have that $J_{\ker}(G_{\tau })=\emptyset $ and 
$F_{meas}(\tau )={\cal M}_{1}(\CCI )$.  
\item \label{randomthm1-min} $\sharp \emMin (G_{\tau },\CCI )=2.$ 
More precisely, $\{ \infty \} $ is a minimal set for $(G_{\tau },\CCI )$, 
and there exists a unique minimal set $L_{\tau }$ for $(G_{\tau },\CCI )$ 
such that $L_{\tau }\subset \hat{K}(G_{\tau }).$  
\item \label{randomthm1az} 
For each $z\in \CCI $, there exists a Borel subset ${\cal A}_{z}$ of ${\cal P}^{\NN }$ 
with $\tilde{\tau }({\cal A}_{z})=1$ such that 
for each $\gamma =(\gamma _{1},\gamma _{2},\ldots ,)\in {\cal A}_{z}$, 
%\begin{itemize}
%\item[{\em (a)}]
{\em (a)} 
either $\gamma _{n,1}(z)\rightarrow \infty $ or $d(\gamma _{n,1},L_{\tau })\rightarrow 0$ as $n\rightarrow \infty $, 
and 
%\item[{\em (b)}] 
{\em (b)} 
there exists a number $\delta =\delta (z,\gamma )>0$ such that 
diam$(\gamma _{n,1}(B(z,\delta )))\rightarrow 0$ as $n\rightarrow \infty .$ 
%\end{itemize}
\item \label{randomthm1-6}
%({\bf No Julia set of $\{ (M_{\tau }^{n})_{\ast }\} _{n}$})  
There exists a unique $M_{\tau }^{\ast }$-invariant Borel probability measure 
$\mu _{\tau} $ 
on 
$\hat{K}(G_{\tau })$
which satisfies the following $(\ast )$. 
\begin{itemize}
\item[$(\ast )$] For each 
$\varphi \in C(\CCI )$,  
$M_{\tau }^{n}(\varphi )(z)\rightarrow 
T_{\infty ,\tau}(z)\cdot \varphi (\infty )
+(1-T_{\infty ,\tau}(z))\cdot 
(\int _{\CCI }\varphi \ d\mu _{\tau })\ \ (n\rightarrow \infty )
$ uniformly on $\CCI $. 
\end{itemize}   
Thus 
$(M_{\tau }^{\ast})^{n }
(\nu ) \rightarrow 
(\int _{\CCI } T_{\infty ,\tau}\ d\nu )\cdot 
\delta _{\infty }+(\int  _{\CCI }(1-T_{\infty ,\tau})\ d\nu )\cdot \mu _{\tau }  
\ \ (n\rightarrow \infty )$ uniformly on ${\cal M}_{1}(\CCI )$.  
Also, {\em supp}$\,\mu _{\tau }=L_{\tau }.$ 
Moreover, the $M_{\tau }$-invariant subspace of $C(\CCI )$ 
is two-dimensional and it is spanned by the constant function and $T_{\infty ,\tau}$. 
Moreover, the set of ergodic components of $M_{\tau }^{\ast }$-invariant elements in ${\frak M}_{1}(\CCI )$  
 is equal to $\{ \delta _{\infty } , \mu _{\tau }\} $. 
\end{enumerate} 

\end{thm}
\begin{rem}
Let $f\in {\cal P}$. 
Then $J_{\ker }(\langle f\rangle )=J(f)\neq \emptyset $,  
$T_{\infty ,\delta _{f}}(\CCI )=\{ 0,1\} $ and 
$T_{\infty ,\delta _{f}}$ is not continuous at any point of $J(f).$ 
Moreover, regarding the dynamics of $f:J(f)\rightarrow J(f)$, we have chaos in the sense of Devaney (\cite{Be,D}).  
Thus Theorem~\ref{randomthm1} describes new phenomena which cannot hold in the usual iteration dynamics of a single polynomial. 
For a $\tau \in {\frak M}_{1,c}({\cal P})$ with $G_{\tau }\in {\cal G}_{dis}$, 
we sometimes call the function $T_{\infty ,\tau }$ a ``{\bf devil's coliseum}'', 
especially when $\mbox{int}(J(G_{\tau }))=\emptyset .$ This terminology and the study were introduced by the author of this paper 
in \cite{Splms10}. 
For the graph of $T_{\infty ,\tau }$  and the graphics of $J(G_{\tau })$, see figures in \cite{Splms10}. 
%Figures~\ref{fig:dcjulia}, \ref{fig:dcgraphgrey2}, \ref{fig:dcgraphudgrey2}. 
Statement~\ref{randomthm1q} means that $T_{\infty ,\tau }$ can detect many parts of $J(G_{\tau }).$ 
%\end{rem}
%\begin{rem}
%If we obtain some results of the dynamics of rational semigroups, then 
%they are applied to the random complex dynamics. 
Thus, by obtaining results about the dynamics of polynomial semigroups, 
one can correspondingly apply such results to the setting of random complex dynamics.  
Conversely, 
%by using Theorems~\ref{kerjemptythm}, \ref{randomthm1}, and Proposition~\ref{kerjemptytinfprop} etc., 
%and 
studying the level sets of 
$T_{\infty,\tau}$, we can get much information about   
$J(G)$. In other words, 
in order to investigate the dynamics of polynomial semigroups, 
it is very effective to study the associated random complex dynamics and then 
 apply the results to the original polynomial semigroups. 
In the proof (section~\ref{Proofs}) of Theorem~\ref{randomthm1}, we combine some results (geometric observations) on the dynamics of 
a $G\in {\cal G}_{dis}$  from \cite{SdpbpI}  
and some results on random complex dynamics from \cite{Splms10}. 
It is critical to know whether or not $J_{\ker }(G_{\tau })=\emptyset $.  
This condition implies that the chaos of the averaged system disappears 
in the $C^{0}$ sense 
due to the cooperation of many kinds of maps in the system even though each map has a chaotic part. 
For the details of the study of random dynamics generated by $\tau \in {\frak M}_{1,c}({\cal P} )$ 
with $J_{\ker }(G_{\tau })=\emptyset $, 
see \cite{Splms10,Scp}. 
In \cite{Splms10,Scp}, it is shown that regarding the random dynamics of complex polynomials, 
for a generic $\tau \in {\frak M}_{1,c}({\cal P})$, we have that 
$J_{\ker }(G_{\tau })=\emptyset $, the chaos of the averaged system disappears 
in the $C^{0}$ sense 
due to the automatic cooperation of many kinds of maps in the system ({\bf cooperation principle}), 
and $T_{\infty ,\tau }$ is \Hol der continuous on $\CCI $.  
We remark that many physicists have observed by numerical experiments that 
if we add uniform noise to a chaotic map on $\RR $, there are many cases in which the chaos of the averaged system disappears.  
This phenomenon in random dynamics on $\RR $ is called the ``noise-induced order'' (\cite{MT}). 
\end{rem}
We are interested in the pointwise H\"{o}lder exponents 
and (non-)differentiability of $T_{\infty ,\tau }$ at points in $J(G_{\tau }).$ 
In order to state the result, we need several definitions.  
\begin{df}
%[\cite{S7}]
\label{d:sp}
%Let $Y$ be a compact metric space and 
Let $\Gamma $ be a non-empty compact subset of  
${\cal P} .$ 
We endow $\GN \times \CCI $ with the product topology. 
Thus this is a compact metrizable space.   
We define a map $f:\GN \times \CCI \rightarrow 
\GN \times \CCI $ as follows: 
For a point $(\gamma  ,y)\in \GN \times \CCI $ where 
$\gamma =(\gamma _{1},\gamma _{2},\ldots )$, we set 
$f(\gamma  ,y):= (\sigma (\gamma  ), \gamma  _{1}(y))$, where 
$\sigma :\GN \rightarrow \GN $ is the shift map, that is, 
$\sigma (\gamma  _{1},\gamma  _{2},\ldots )=(\gamma  _{2},\gamma  _{3},\ldots ).$ 
The map $f:\GN \times \CCI \rightarrow \GN \times \CCI $ 
is called the {\bf skew product associated with the 
generator system }$\G .$ 
Moreover, we use the following notations.  
%\begin{enumerate}
%\item 
Let 
$\pi : \GN \times \CCI \rightarrow \GN $  
and $\pi _{\CCI }:\GN \times \CCI \rightarrow \CCI $ be 
the canonical projections. 
For each 
$\gamma  \in \GN $ and $n\in \NN $, we set 
$f_{\gamma  }^{n}:= f^{n}|_{\pi ^{-1}(\{ \gamma  \} )} : 
\pi ^{-1}(\{ \gamma  \} )\rightarrow 
\pi ^{-1}(\{ \sigma ^{n}(\gamma  )\} ).$  
Moreover, we set 
$f_{\gamma  ,n}:= \gamma  _{n}\circ \cdots \circ \gamma  _{1}.$ 
%\item 
We denote by $F_{\gamma }$ the set of points $z\in \CCI $ 
satisfying  that there exists a neighborhood $U$ of $z$ 
such that  
$\{ f_{\gamma ,n}:U\rightarrow \CCI \} _{n=1}^{\infty }$ 
is equicontinuous on $U.$ We set $J_{\gamma }:= \CCI \setminus 
F_{\gamma }.$ The set $F_{\gamma }$ is called the Fatou set of $\gamma $, and $J_{\gamma }$ is called the Julia set of $\gamma .$  
For each $\gamma \in \Gamma ^{\NN }$, 
we set $J^{\gamma  }:= \{ \gamma  \} \times J_{\gamma  }
\ (\subset \GN \times \CCI )$. 
%\item 
Moreover, we set 
$\tilde{J}(f):= \overline{\bigcup _{\gamma  \in \GN }J^{\gamma  }}$, 
where the closure is taken in the product space $\GN \times \CCI .$ 
Furthermore, we set $\tilde{F}(f):= (\Gamma ^{\NN }\times \CCI )\setminus \tilde{J}(f).$  
%(Note that $f^{-1}(\tilde{J}(f))=\tilde{J}(f)=f(\tilde{J}(f)).$)  
%\item 
For each $\gamma  \in \GN $, we set 
$\hat{J}^{\gamma  ,\Gamma }:= \pi ^{-1}(\{  \gamma  \} )\cap \tilde{J}(f)$, 
$\hat{F}^{\gamma ,\Gamma }:= \pi ^{-1}(\{ \gamma \} )\setminus \hat{J}^{\gamma ,\G }$,  
$\hat{J}_{\gamma ,\Gamma }:= \pi _{\CCI }(\hat{J}^{\gamma ,\Gamma })$, 
and $\hat{F}_{\g, \G }:= \CCI \setminus \hat{J}_{\g ,\G }.$   
Note that $J_{\gamma }\subset \hat{J}_{\gamma ,\Gamma }.$   
%\item 
%When $\Gamma \subset \mbox{Rat}$, 
For any point $z\in \CCI $, we denote by $T\CCI _{z}$ the complex tangent space of $\CCI $ at $z$. 
For any holomorphic map $\varphi $ defined on a domain $V$ and for any point $z\in V$, 
we denote by $D\varphi _{z}:T\CCI _{z}\rightarrow T\CCI _{\varphi (z)} $ the derivative map at $z.$ 
For each $z=(\gamma  ,y)\in \GN \times \CCI $, 
we set $Df_{z}:=(D\gamma  _{1})_{y}$.   
%More generally, for each $n\in \NN $ and 
%$z=(\gamma  ,y)\in \GN \times \CCI $, 
%we set $(f^{n})'(z):= (f_{\gamma  ,n})'(y).$ 
%\end{enumerate}
Let $U$ be a domain in $\CCI $ and let 
$g:U\rightarrow \CCI $ be a meromorphic function. 
For each $z\in U$, we denote by $\| Dg_{z}\| _{s}$ the 
norm of the derivative of $g$ at $z$ with respect to the spherical metric. 
\end{df}
\begin{rem}
Under the above notation, let $G=\langle \Gamma \rangle .$  
Then   
$\pi _{\CCI }(\tilde{J}(f))\subset J(G)$ and 
$\pi \circ f=\sigma \circ \pi $ on $\Gamma ^{\NN }\times \CCI .$ 
%Moreover, for each $\g \in \Gamma ^{\NN }$, 
%$\gamma _{1}(J_{\gamma })\subset J_{\sigma (\gamma )}$, 
%$\gamma _{1}(\hat{J}_{\gamma ,\Gamma })\subset \hat{J}_{\sigma (\gamma ),\Gamma }$, 
%and $f(\tilde{J}(f))\subset \tilde{J}(f)$ (see Lemma~\ref{genskewprodinvlem1}). 
Furthermore, 
%if $\Gamma \in \Cpt(\Rat)$, then 
for each $\gamma \in \Gamma ^{\NN }$, 
$\gamma _{1}(J_{\gamma })=J_{\sigma (\gamma )}$, $\gamma _{1}^{-1}(J_{\sigma (\gamma )})=J_{\gamma }$, 
$\gamma _{1}(\hat{J}_{\gamma ,\Gamma })=\hat{J}_{\sigma (\gamma ),\Gamma }$, 
$\gamma _{1}^{-1}(\hat{J}_{\sigma (\gamma ),\Gamma })=\hat{J}_{\gamma ,\Gamma }$,  
$f(\tilde{J}(f))=\tilde{J}(f)=f^{-1}(\tilde{J}(f))$, and $f(\tilde{F}(f))=\tilde{F}(f)=f^{-1}(\tilde{F}(f))$ 
(see \cite[Lemma 2.4]{S4}). We remark that in general, $J_{\g }\subsetneqq \hat{J}_{\g ,\G }$ (\cite[Example 1.7]{S7}). 
%$$\begin{CD}
%\Gamma ^{\NN }\times \CCI @>{f}>>\Gamma ^{\NN }\times \CCI \\ 
%@V{\pi}VV 
%@VV{\pi }V\\ 
%\Gamma ^{\NN }@>>{\sigma }>\Gamma ^{\NN }  
%\end{CD}
%$$
%\item If $\sharp J(G)\geq 3$, then $\pi _{\CCI }(\tilde{J}(f))=J(G).$ 
%\end{enumerate}
\end{rem}
\begin{df}
\label{d:dfu}
Let $m\in \NN .$ We set 
${\cal W}_{m}:= \{ (p_{1},\ldots ,p_{m})\in (0,1)^{m}\mid \sum _{j=1}^{m}p_{j}=1\} .$ 
Let $h=(h_{1},\ldots, h_{m})\in {\cal P}^{m}$ be an element such that 
$h_{1},\ldots ,h_{m}$ are 
mutually distinct. 
We set 
$\Gamma := \{ h_{1},\ldots ,h_{m}\} .$ 
Let $f:\Gamma ^{\NN }\times 
\CCI \rightarrow \Gamma ^{\NN }\times \CCI $ be the 
skew product associated with $\Gamma .$  
Let $\mu \in {\frak M}_{1}(\Gamma ^{\NN }\times \CCI )$ be an $f$-invariant Borel probability measure. 
%We set ${\cal W}_{m}:= \{ (p_{1},\ldots ,p_{m})\in \RR ^{m}\mid \sum _{j=1}^{m}p_{j}=1,\ p_{j}>0 (\forall j)\} .$ 
For each $p=(p_{1},\ldots ,p_{m})\in {\cal W}_{m}$, 
we define a function $\tilde{p}:\Gamma ^{\NN }\times \CCI \rightarrow \RR $ by 
$\tilde{p}(\gamma ,y):=p_{j}$  if $\gamma _{1}=h_{j} $ 
(where $\gamma =(\gamma _{1},\gamma _{2},\ldots )$), and  we set 
$$u(h,p,\mu ):= 
\frac{-(\int _{\Gamma ^{\NN }\times \CCI }\log \tilde{p}(z)\ d\mu (z) )}
{\int _{\Gamma ^{\NN }\times \CCI }\log \| Df_{z}\| _{s} \ d\mu (z)}$$ 
(when the integral in the denominator is positive and finite). 
%\end{df}
%\begin{df} \label{d:green} 
For each $\g \in {\cal P}^{\NN }$, 
we set $A_{\infty ,\g }:= \{ z\in \CCI \mid \g _{n,1}(z)\rightarrow \infty \ (n\rightarrow \infty )\} $ 
and $K_{\g }:= \{ z\in \CC \mid \{ \g _{n,1}(z)\} _{n\in \NN } \mbox{ is bounded in }\CC \} .$  
%Let $h=(h_{1},\ldots , h_{m})\in {\cal P}^{m}$ be an element such that 
%$h_{1},\ldots ,h_{m}$ are 
%mutually distinct. 
%We set 
%$\Gamma := \{ h_{1},\ldots ,h_{m}\} .$ 
For any $(\gamma ,y)\in \Gamma ^{\NN }\times \CC $, 
let $G_{\gamma }(y):= \lim _{n\rightarrow \infty }\frac{1}{\deg (\gamma _{n,1})}
\log ^{+}|\gamma _{n,1}(y)|$, 
where $\log ^{+}a:=\max\{\log a,0\} $ for each $a>0.$  
By the arguments in \cite{Se}, for each $\gamma \in \Gamma ^{\NN }$, 
$G_{\gamma }(y) $ exists, 
%$(\gamma ,y)\mapsto G_{\gamma }(y)$ is continuous on $\Gamma ^{\NN }\times \CC $, 
$G_{\gamma }$ is subharmonic on $\CC $, and 
$G_{\gamma }|_{A_{\infty ,\gamma }}$ is equal to the Green's function on 
$A_{\infty ,\gamma }$ with pole at $\infty $.
% for each $\gamma \in \Gamma ^{\NN }.$ 
Moreover, $(\gamma ,y)\mapsto G_{\gamma }(y)$ is continuous on $\Gamma ^{\NN }\times \CC .$ 
Let $\mu _{\gamma }:=dd^{c}G_{\gamma }$, where $d^{c}:=\frac{i}{2\pi }(\overline{\partial }-\partial ).$ 
Note that by the argument in \cite{J2}, 
$\mu _{\gamma }$ is a Borel probability measure on $J_{\gamma }$ such that 
$\mbox{supp}\, \mu _{\gamma }=J_{\gamma }.$ 
%Furthermore, for each $\gamma \in \Gamma ^{\NN }$, 
%let $\Omega (\gamma )=\sum _{c} G_{\gamma }(c)$, where $c$ runs over all critical points of 
%$\gamma _{1}$ in $\CC $, counting multiplicities.   
%\end{df}
%\begin{rem}
%\label{r:maxrelent}
%Let $h=(h_{1},\ldots ,h_{m})\in (\Ratp)^{m}$ be an element such that 
%$h_{1},\ldots ,h_{m}$ are mutually distinct. 
%and $\deg (h_{j})\geq 2$ for each $j=1,\ldots ,m.$ 
%Let $\Gamma =\{ h_{1},\ldots ,h_{m}\} $ and 
Let $f:\Gamma ^{\NN }\times \CCI \rightarrow \Gamma ^{\NN }\times \CCI $ 
be the skew product map associated with $\Gamma .$ 
Moreover, let $p=(p_{1},\ldots ,p_{m})\in {\cal W}_{m}$ and 
let $\tau =\sum _{j=1}^{m}p_{j}\delta _{h_{j}}\in {\frak M}_{1}(\Gamma ).$ 
Then, there exists a unique $f$-invariant ergodic Borel probability measure 
$\mu $ on $\Gamma ^{\NN }\times \CCI $ such that $\pi _{\ast }(\mu )=\tilde{\tau }$ and   
$h_{\mu }(f|\sigma )=\max _{\rho \in {\frak E}_{1}(\Gamma ^{\NN }\times \CCI ): 
f_{\ast }(\rho )=\rho, \pi _{\ast }(\rho )=\tilde{\tau} }  h_{\rho }(f|\sigma )=\sum _{j=1}^{m}p_{j}\log (\deg (h_{j}))$, 
where $h_{\rho }(f|\sigma )$ denotes the relative metric entropy 
of $(f,\rho )$ with respect to $(\sigma, \tilde{\tau })$, and 
${\frak E}_{1}(\cdot )$ denotes the space of ergodic measures (see \cite{S3}).  
This $\mu $ is called the {\bf maximal relative entropy measure} for $f$ with respect to 
$(\sigma ,\tilde{\tau }).$   
Note that in \cite[Lemma 5.51]{Splms10} it was shown that 
%if $\G \subset {\cal P}$ then 
for each continuous function $\varphi :\GN \times \CCI \rightarrow \RR $, 
$\int \varphi (\g ,y)d\mu (\g, y)=\int d\tilde{\tau }(\g )\int \varphi (\g ,y)d\mu _{\g }(y).$ 
Thus $(\pi _{\CCI })_{\ast }(\mu )=\int _{\GN }\mu _{\g }d\tilde{\tau }(\g ).$  
%\end{rem}
\end{df} 
\begin{df}
\label{d:phe}
Let $V$ be a non-empty open subset of $\CC .$ Let $\varphi :V \rightarrow \CC $ be a function and 
let $y\in V $ be a point. Suppose that $\varphi $ is bounded around $y.$ 
Then we set \\  
$\mbox{H\"{o}l}(\varphi ,y):= 
\sup \{ \beta \in [0,\infty ) \mid \limsup _{z\rightarrow y,z\neq y}
\frac{|\varphi (z)-\varphi (y)|}{|z-y|^{\beta }}<\infty \}\in [0,\infty ].$
%where $d$ denotes the spherical distance. 
This is called the {\bf pointwise H\"{o}lder exponent of $\varphi $ at $y.$} 
\end{df}
\begin{rem} 
If $\mbox{H\"{o}l}(\varphi ,y)<1$, then 
$\varphi $ is non-differentiable at $y.$ 
If 
 $\mbox{H\"{o}l}(\varphi ,y)>1$, then 
$\varphi $ is differentiable at $y$ and the derivative at $y$ is equal to $0.$ 
See also \cite[Remark 3.39]{Scp}. 
\end{rem}
We now present the results on the pointwise H\"{o}lder exponents and (non-)differentiability 
of $T_{\infty ,\tau }$ at points in $J(G_{\tau }).$ 
\begin{thm}[{\bf Non-differentiability of $T_{\infty ,\tau }$ 
at points in $J(G_{\tau })$}]
\label{t:hnondiffp}
Let $m\in \NN $ with $m\geq 2.$ 
Let $h=(h_{1},\ldots ,h_{m})\in {\cal P}^{m}$ such that $h_{1},\ldots ,h_{m}$ are mutually distinct 
and let  
$\Gamma = \{ h_{1},h_{2},\ldots ,h_{m}\} .$ 
Let $G=\langle h_{1},\ldots ,h_{m}\rangle .$ 
Let $p=(p_{1},\ldots ,p_{m})\in {\cal W}_{m}.$  
Let $f:\Gamma ^{\NN }\times \CCI \rightarrow \Gamma ^{\NN }\times \CCI $ be the  
skew product associated with $\Gamma .$  
Let 
$\tau := \sum _{j=1}^{m}p_{j}\delta _{h_{j}}\in {\frak M}_{1}(\Gamma )
\subset {\frak M}_{1}({\cal P }).$
Let 
$\mu \in {\frak M}_{1}(\Gamma ^{\NN }\times \CCI )$ be the maximal relative entropy measure 
 for $f:\Gamma ^{\NN }\times \CCI \rightarrow \Gamma ^{\NN }\times \CCI $ with respect to 
 $(\sigma ,\tilde{\tau }).$ 
%(Note that the existence and the uniqueness of the maximal relative 
%entropy measure has been shown in \cite{S3}).  
%Moreover, let 
Let $\lambda = (\pi _{\CCI })_{\ast }(\mu )\in {\frak M}_{1}(\CCI ).$ 
Suppose that  
%$\hat{K}(G)\neq \emptyset $, 
$G\in {\cal G}$  and 
$h_{i}^{-1}(J(G))\cap h_{j}^{-1}(J(G))=\emptyset $ for each 
$(i,j)$ with $i\neq j$. 
Then, we have all of the following. 
\begin{enumerate}
\item \label{t:hnondiffp0} $G_{\tau }=G\in {\cal G}_{dis}$, $J_{\ker }(G)=\emptyset $,  
and all statements in Theorem~\ref{randomthm1} hold for $\tau. $  
%\item \label{t:hnondiffp1}
Moreover, $J(G)=\{ z\in \CCI \mid \mbox{ for any neighborhood } U \mbox{ of } z, T_{\infty ,\tau }|_{U} 
\mbox{ is not constant}\} $ and  $\mbox{{\em int}}(J(G))=\emptyset .$
%\item \label{t:hnondiffp2} 
Furthermore, {\em supp} $\lambda =J(G)$ and 
%\item \label{t:hnondiffp3} 
%Moreover, 
for each $z\in J(G)$, $\lambda (\{ z\} )=0.$ 
\item \label{t:hnondiffp4} 
%For almost every $z_{0}\in J(G)$ with respect to $\lambda $, 
There exists a Borel subset $A$ of $J(G)$ with $\lambda (A)=1$ such that 
for each $z_{0}\in A$, 
%and each $\varphi \in (\mbox{{\em LS}}({\cal U}_{f,\tau }(\CCI )))_{nc}$, 
$$
\mbox{{\em H\"{o}l}}(T_{\infty ,\tau } , z_{0})\leq 
u(h,p,\mu )=\frac{-(\sum _{j=1}^{m}p_{j}\log p_{j})}
{\sum _{j=1}^{m}p_{j}\log \deg (h_{j})}< 1.$$ 
\item \label{t:hnondiffp4-1}
%If $h=(h_{1},\ldots ,h_{m})\in {\cal P}^{m}$, then 
%$$
%u(h,p,\mu )=\frac{-(\sum _{j=1}^{m}p_{j}\log p_{j})}
%{\sum _{j=1}^{m}p_{j}\log \deg (h_{j})+\int _{\Gamma ^{\NN }}\Omega (\gamma )\ d\tilde{\tau }(\gamma )}
%$$
%and 
%\begin{align*}
%2 \geq & 
We have that 
$$\dim _{H}(\{ z\in J(G) \mid  \emHol (T_{\infty ,\tau } ,z)\leq u(h,p,\mu )\} )
%\\ 
\geq 
%& 
\frac{\sum _{j=1}^{m}p_{j}\log \deg (h_{j})-\sum _{j=1}^{m}p_{j}\log p_{j}}
{\sum _{j=1}^{m}p_{j}\log \deg (h_{j})}>1, $$
where $\dim _{H}$ denotes the Hausdorff dimension with respect to the Euclidian distance on $\CC .$  
%\end{align*} 
\item \label{t:hnondiffp5} 
%Suppose $h=(h_{1},\ldots ,h_{m})\in {\cal P}^{m}.$ 
%Suppose that at least one of the following {\em (a)}, {\em (b)}, and {\em (c)} holds:
%\begin{itemize}
%\item[{\em (a)}]
% {\em (a)} $\sum _{j=1}^{m}p_{j}\log (p_{j}\deg (h_{j}))>0$.  
%\item[{\em (b)}]
%{\em (b)} $\dim _{H}(J(G))<2.$  
%\item[{\em (c)}]
%{\em (c)} $m=2.$ 
% \end{itemize}
%Then, $u(h,p,\mu )<1$ and  
For each non-empty open subset $U$ of $J(G)$ there exists an uncountable dense subset $A_{U}$ of $U$ such that 
for each $z\in A_{U}$, 
% and each $\varphi \in (\mbox{{\em LS}}({\cal U}_{f,\tau }(\CCI )))_{nc}$, 
$T_{\infty ,\tau } $ is non-differentiable at $z.$ 
\end{enumerate}
 
\end{thm}
In \cite[Theorem 3.82]{Splms10}, it is assumed 
that $G$ is hyperbolic, i.e., $P(G)\subset F(G).$ 
However, in Theorem~\ref{t:hnondiffp}, we do not assume hyperbolicity of $G.$ 
Note that there are many examples of (non-hyperbolic) $G=\langle h_{1},\ldots ,h_{m}\rangle 
\in {\cal G}_{dis}$ for which $\{ h_{i}^{-1}(J(G))\} _{i=1}^{m}$ are mutually disjoint 
(see Theorem~\ref{t:2gengdis}, Propositions~\ref{Constprop}, \ref{semihyposcexprop},  
Examples~\ref{ex:dc1}, \ref{ex:Siegeldc}, \ref{ex:mgengdis}, Remark~\ref{r:sdgdis}).  

%Note that 
%In \cite[Theorem 3.82]{Splms10}, it is assumed 
%that $G$ is hyperbolic, i.e., $P(G)\subset F(G).$ 
%However, in Theorem~\ref{t:hnondiffp}, we do not assume hyperbolicity of $G.$ 
%Note that there are many examples of (non-hyperbolic) $G=\langle h_{1},\ldots ,h_{m}\rangle 
%\in {\cal G}_{dis}$ for which $\{ h_{i}^{-1}(J(G))\} _{i=1}^{m}$ are mutually disjoint 
%(see Proposition~\ref{Constprop}, Remark~\ref{r:sdgdis}, Theorem~\ref{t:2gengdis}).  
%Theorem~\ref{t:hnondiffp} means that even though the chaos of the averaged system 
%disappears in the $C^{0} $ sense, 
%it can remain in the $C^{\alpha }$ sense with some $\alpha \in (0,1)$,  
%where $C^{\alpha }$ denotes the space of 
%$\alpha $-H\"{o}lder continuous functions. 
%In \cite{Scp} it is shown that $T_{\infty ,\tau }$ is H\"{o}lder continuous on $\CCI $ with some exponent.   
%From these, we can say that we have a gradation between chaos and non-chaos. 
%In the proof (section~\ref{Proofs}) of Theorem~\ref{t:hnondiffp},   
%we use Birkhoff's ergodic theorem, potential theory, the Koebe distortion theorem, 
%and some observations 
%(\cite{SdpbpI}) 
%about Con$(J(G))$ and the Julia set of the associated real affine semigroup.    
%from \cite{SdpbpI}. 

We present a result on $2$-generator semigroup $G=\langle h_{1},h_{2}\rangle \in {\cal G}_{dis}$ 
and the associated random dynamics generated by $\tau =\sum _{j=1}^{2}p_{j}\delta _{h_{j}}$ 
where $(p_{1},p_{2})\in {\cal W} _{2}.$ 
\begin{thm}
\label{t:2gengdis}
Let $G=\langle h_{1},h_{2}\rangle \in {\cal G}_{dis}.$ 
Let $(p_{1},p_{2})\in {\cal W}_{2}$ and 
let $\tau =\sum _{j=1}^{2}p_{j}\delta _{h_{j}}.$  
Let $\G =\{ h_{1},h_{2}\} .$ 
Then, we have all of the following.
\begin{enumerate}
\item \label{t:2gengdis1} $h_{1}^{-1}(J(G))\cap h_{2}^{-1}(J(G))=\emptyset $.   
For $((h_{1},h_{2}),(p_{1},p_{2}))$, 
all statements~\ref{t:hnondiffp0}--\ref{t:hnondiffp5} in Theorem~\ref{t:hnondiffp} hold.    
For each $\g \in \GN $, 
$J_{\g }=J_{\g ,\G }=\bigcap _{j=1}^{\infty }\g _{1}^{-1}\cdots \g _{j}^{-1}(J(G)).$ 
The map $\g \mapsto J_{\g }$ is continuous on $\GN $ with respect to the Hausdorff metric 
in the space of all non-empty compact sets in $\CCI .$
\item \label{t:2gengdis2}
For each $J\in \mbox{{\em Con}}(J(G))$, there exists a unique $\g \in \GN $ with 
$J=J_{\g }.$ {\em Con}$(J(G))=\{ J_{\g }\mid \g \in \GN \} .$ 
The map $\g \mapsto J_{\g }$ is a bijection between $\GN $ and $\mbox{{\em Con}}(J(G)).$ 
In particular, 
there exist uncountably many connected components of $J(G).$ 
\item \label{t:2gengdis3}
There exist infinitely many doubly connected components of $F(G).$ 
\item \label{t:2gengdis4}
For each $J\in \mbox{{\em Con}}(J(G))$,  
$T_{\infty ,\tau }|_{J}$ is constant. 
\item \label{t:2gengdis5}
Let $J_{1},J_{2}\in \mbox{{\em Con}}(J(G))$ with $J_{1}\neq J_{2}.$ 
Suppose $T_{\infty ,\tau }|_{J_{1}}=T_{\infty ,\tau }|_{J_{2}}.$ 
Then there exists a doubly connected component $A$ of $F(G)$ such that 
$\partial A\subset J_{1}\cup J_{2}.$ 
\item \label{t:2gengdis6} 
Either $J(h_{1})<_{s}J(h_{2})$ or $J(h_{2})<_{s}J(h_{1}).$ 
Without loss of generality, we may assume that $J(h_{1})<_{s}J(h_{2}).$ 
Then $J_{\min }(G)=J(h_{1})$ and  $J_{\max }(G)=J(h_{2})$.  
Moreover, 
the map 
$\zeta: w=(w_{1},w_{2},\ldots )\in \{ 1,2\} ^{\NN }\mapsto J_{\g (w)}\in \mbox{{\em Con}}(J(G))$, 
where $\g (w) =(h_{w_{1}},h_{w_{2}},\ldots )\in \GN $, 
is a bijection such that $w^{1}<_{l}w^{2}$ implies $\zeta (w^{1})<_{s}\zeta (w^{2})$, 
where $<_{l}$ denotes the lexicographic order in $\{ 1,2\} ^{\NN }$, 
i.e., $(i_{1},\ldots ,i_{n},1,\ldots )<_{l}(i_{1},\ldots ,i_{n},2,\ldots )$.  
\item \label{t:2gengdis7} 
Suppose $J(h_{1})<_{s}J(h_{2}).$ 
Then $T_{\infty ,\tau }^{-1}(\{ 0\} )=K(h_{1})$ and $T_{\infty ,\tau }^{-1}(\{ 1\} )=\overline{F_{\infty }(h_{2})}.$ 
Moreover, for each $t\in (0,1)$, 
exactly one of the following {\em (a)} and {\em (b)} holds.
\begin{itemize}
\item[{\em (a)}]
\label{t:2gengdis7a} 
There exists a unique $w\in \{ 1,2\} ^{\NN } $ such that $T_{\infty ,\tau }^{-1}(\{ t\} )=J_{\g (w)}.$
Moreover, $\sharp \{ n\in \NN \mid w_{n}=1\} =\sharp \{ n\in \NN \mid w_{n}=2\} =\infty .$ 
Moreover, there exists exactly one bounded component $B_{w}$ of $F_{\g (w)}.$ 
Furthermore, 
$\partial B_{w}=\partial A_{\infty ,\g (w)}=J_{\g (w)}.$ 
\item[{\em (b)}]
There exist two elements $\rho ,\mu \in \{ 1,2\} ^{\NN }$ such that 
$\rho <_{l}\mu $, $J_{\g (\rho )}<_{s}J_{\g (\mu )}$, 
and $T_{\infty ,\tau }^{-1}(\{ t\} )=K_{\g (\mu )}\setminus \mbox{{\em int}}(K_{\g (\rho )}).$ 
Moreover, 
either {\em (i)} $\rho =(1,2,2,2,\ldots )$ and $\mu =(2,1,1,1,\ldots )$, or 
{\em (ii)} there exists a finite word $(i_{1},\ldots ,i_{n})\in \{ 1,2\}^{n}$ for some $n\in \NN $ 
such that $\rho =(i_{1},\ldots ,i_{n},1,2,2,2,\ldots )$ and 
$\mu =(i_{1},\ldots ,i_{n},2,1,1,1,\ldots ).$  
 Moreover, 
there exists a doubly connected component $A$ of $F(G)$ such that 
$\partial A\subset J_{\g (\rho )}\cup J_{\g (\mu )}.$ 
Furthermore, $J_{\g (\rho )}$ is a quasicircle. 
\end{itemize}

\end{enumerate}
\end{thm}
\begin{rem}
\label{r:jgstudy}
We remark that in general, $\gamma \in \GN \mapsto J_{\g }$ is not continuous (\cite[Example 1.7]{S7}). 
Under the assumptions of Theorem~\ref{t:2gengdis}, 
for the studies of $\{ J_{\g }\} _{\g \in \GN }$, 
see \cite{SdpbpII,SdpbpIII}. 
%In \cite{SdpbpII}, it was shown that under the assumptions of Theorem~\ref{t:2gengdis}, 
%if $\g \in \GN $ and if int$K_{\g }$ is not connected, then 
%there exists an element $g\in \G $ with $J(g)=J_{\min }(G)$ and 
%an $s\in \NN $ such that for each $n\in \NN $ with $n\geq s$, 
%$\g _{n}=g.$ 
In \cite{SdpbpII}, under the assumptions of Theorem~\ref{t:2gengdis} and 
assuming that $h_{1}$ (with $J(h_{1})<_{s}J(h_{2})$) is hyperbolic and 
$P^{\ast }(\langle h_{2}\rangle )\subset \mbox{int}(K(h_{1}))$ 
(which implies $G$ is hyperbolic, i.e., $P(G)\subset F(G)$, 
see \cite[Theorem 2.36]{SdpbpI}), a classification of the fiberwise Julia sets $J_{\g }$ was given. 
In particular, it was  shown that under the assumptions of Theorem~\ref{t:2gengdis}, 
if the above $h_{1}$ is hyperbolic, $P^{\ast }(\langle h_{2}\rangle )\subset \mbox{int}(K(h_{1}))$ 
and $J(h_{1})$ is not a Jordan curve, then 
for any $\g =(\g _{1},\g_{2},\ldots )\in \GN $ satisfying that 
(a) $\sharp \{ n\in \NN \mid \g _{n}\neq h_{1}\} =\infty $ 
and (b) there exists a strictly increasing sequence $\{ n_{k}\} _{k=1}^{\infty }$ in $\NN $ 
such that $\sigma ^{n_{k}}(\g )\rightarrow (h_{1},h_{1},h_{1},\ldots )$ as $k\rightarrow \infty $, 
the Julia set $J_{\g }$ of $\g $ satisfies that 
(I) $J_{\g }$ is a Jordan curve but not a quasicircle, (II) the unbounded component $A_{\infty ,\g }$ 
of $\CCI \setminus J_{\g }$ is a John domain, and (III) the bounded component of 
$\CCI \setminus J_{\g }$ is not a John domain. Note that the above phenomenon is a new one 
which cannot hold in the usual iteration dynamics of a single polynomial.

\end{rem}
\begin{rem}
\label{r:2genreal}
Under the assumption of Theorem~\ref{t:2gengdis}, suppose that 
$h_{1}$ and $h_{2}$ with $J(h_{1})<_{s}J(h_{2})$ are real polynomials. 
Then for each $\g \in \GN $, $J_{\g }$ is symmetric with respect to the real axis, 
and $T_{\infty ,\tau }$ is symmetric with respect to the real axis. 
If, in addition to the above assumption, $h_{1}$ is hyperbolic, 
$P^{\ast }(\langle h_{2}\rangle )\subset \mbox{int}(K(h_{1}))$  
and \ref{t:2gengdis7}(a) in Theorem~\ref{t:2gengdis} holds, 
then by \cite{SdpbpII,SdpbpIII}, $J_{\g (w)}$ is a Jordan curve and 
$\sharp (J_{\g (w)}\cap \RR )=2.$ 
For the figure of the Julia set of $\langle h_{1},h_{2}\rangle \in {\cal G}_{dis}$ and 
the graph of $T_{\infty ,\tau }$, see \cite{Splms10}.  
%See Figures~\ref{fig:dcjulia}, \ref{fig:dcgraphgrey2}. 

\end{rem}
We now present some results on $3$-generator semigroups in ${\cal G}_{dis}$ and the associated random dynamics.  
\begin{thm}
\label{t:3genmain}
Let $G=\langle h_{1},h_{2},h_{3}\rangle \in {\cal G}_{dis}.$ 
For each $i=1,2,3,$ let 
$J_{i}\in \mbox{{\em Con}}(J(G))$ with $J(h_{i})\subset J_{i}.$ 
Suppose without loss of generality 
(since $(\mbox{{\em Con}}(J(G)),\leq _{s})$ is totally ordered), 
that $J_{1}\leq _{s}J_{2}\leq _{s}J_{3}.$ 
%(Note: 
% without loss of generality we may assume that.) 
Then, we have exactly one of the following {\em (1),(2),(3)}.
\begin{itemize}
\item[{\em (1)}] 
$\{ h_{i}^{-1}(J(G))\} _{i=1,2,3}$ are mutually disjoint, 
$J_{\min }(G)=J(h_{1})$, $J_{\max }(G)=J(h_{3})$, 
$\hat{K}(G)=K(h_{1})$ and $F_{\infty }(G)=F_{\infty }(h_{3}).$ 
\item[{\em (2)}]
$h_{1}^{-1}(J(G))\cap (\bigcup _{i=2,3}h_{i}^{-1}(J(G)))=\emptyset $, 
$h_{2}^{-1}(J(G))\cap h_{3}^{-1}(J(G))\neq \emptyset $, 
$J_{\min }(G)=J(h_{1})$ and $\hat{K}(G)=K(h_{1}).$ 

\item[{\em (3)}] 
$h_{3}^{-1}(J(G))\cap (\bigcup _{i=1,2}h_{i}^{-1}(J(G)))=\emptyset $, 
$h_{1}^{-1}(J(G))\cap h_{2}^{-1}(J(G))\neq \emptyset $, 
$J_{\max }(G)=J(h_{3})$ and $F_{\infty }(G)=F_{\infty }(h_{3}).$  
\end{itemize}
Moreover, we have the following. 
{\em (a)} If $J_{1}=J_{2},$ then {\em (3)} holds. 
{\em (b)} If $J_{2}=J_{3}$, then {\em (2)} holds. 
{\em (c)} If  
$h_{2}^{-1}(J(G))\cap (\bigcup _{i=1,3}h_{i}^{-1}(J(G)))=\emptyset $, then 
{\em (1)} holds and $J_{1}<_{s}J_{2}<_{s}J_{3}$. 
\end{thm}
\begin{cor}
\label{c:3gengdis}
Let $G=\langle h_{1},h_{2},h_{3}\rangle \in {\cal G}_{dis}.$ 
Then there exist infinitely many connected components of $J(G)$  
and there exist infinitely many doubly connected components of $F(G).$ 
More precisely, there exists an $i\in \{ 1,2,3\} $ such that 
{\em (1)} $h_{i}^{-1}(J(G))\cap (\bigcup _{j:j\neq i}h_{j}^{-1}(J(G)))=\emptyset $, 
{\em (2)} either $J(h_{i})=J_{\max }(G)$ or $J(h_{i})=J_{\min}(G)$, and 
{\em (3)} there exists a sequence $\{ J_{n}\} _{n\in \NN }$ of mutually different elements in \mbox{{\em Con}}$(J(G))$ 
and a sequence $\{ A_{n}\} _{n\in \NN }$ of mutually different doubly connected components of $F(G)$ 
such that $J_{n}\rightarrow  J(h_{i})$ and $\overline{A_{n}}\rightarrow J(h_{i})$ as $n\rightarrow \infty $ 
with respect to the Hausdorff metric. 
\end{cor}
\begin{rem}
\label{r:3gendet}
Let $G=\langle h_{1},h_{2},h_{3}\rangle \in {\cal G}_{dis}$, 
$(p_{1},p_{2},p_{3})\in {\cal W}_{3}$ and $\tau =\sum _{i=1}^{3}p_{i}\delta _{h_{i}}.$ 
Then, by Theorem~\ref{randomthm1} and Corollary~\ref{c:3gengdis}, 
the continuous function $T_{\infty ,\tau }$ can detect the boundaries of infinitely many doubly connected 
components of $F(G)$. Moreover, it can detect either $J_{\max }(G)$ or $J_{\min }(G).$  
There are many examples of each of (1), (2), and (3) of Theorem~\ref{t:3genmain} (\cite{SdpbpI}).  
\end{rem}

\begin{rem}
\label{r:gengdisj}
In \cite{SdpbpI}, it was shown that 
there exists a $3$-generator semigroup $G=\langle h_{1},h_{2},h_{3}\rangle \in {\cal G}_{dis}$ 
such that 
$\sharp \mbox{Con}(J(G))=\aleph _{0}.$ 
In \cite{SdpbpI}, it was also shown that for each $n\in \NN $ with $n\geq 2$, 
there exists a $2n$-generator semigroup $G=\langle h_{1},\ldots ,h_{2n}\rangle \in {\cal G}_{dis}$ 
with $\sharp \mbox{Con}(J(G))=n.$ By developing the idea in \cite{SdpbpI}, 
it was shown in \cite{SS} that for each $n\in \NN $ with $n\geq 2$, 
there exists a $4$-generator semigroup $G=\langle h_{1},\ldots ,h_{4}\rangle \in {\cal G}_{dis}$ with  
$\sharp \mbox{Con}(J(G))=n.$ Note that in  \cite{S15}, the author of this paper constructed a new cohomology theory 
for ``backward self-similar systems'' (backward IFSs),  and by using it, 
for a finitely generated semigroup $G=\langle h_{1},\ldots ,h_{m}\rangle \in {\cal G}$, 
we can investigate the 
cardinality of $\mbox{Con}(J(G))$ and $\mbox{Con}(F(G)).$   
More precisely, we investigate the cohomology groups of 
the nerve ${\cal N}_{k}$ of $\{ (h_{i_{1}}\cdots h_{i_{k}})^{-1}(J(G))\mid (i_{1},\ldots ,i_{k})\in \{ 1,\ldots ,m\} ^{k}\} $ 
for each $k\in \NN $ 
and their direct limits as $k\rightarrow \infty .$ In the proofs (section~\ref{Proofs}) of Theorems~\ref{t:2gengdis} and \ref{t:3genmain}, 
we use some results (geometric observations on the nerves ${\cal N}_{k}$ and their inverse limit,  
e.g. Con$(J(G))\cong $Con$(\varprojlim _{k}|{\cal N}_{k}|)$) from \cite{S15} and 
some results on the dynamics of $G\in {\cal G}_{dis}$ from \cite{SdpbpI}.  
\end{rem}
\begin{rem}
\label{r:ncj}
Let $\tau \in {\cal M}_{1}({\cal P}).$  
Suppose $G_{\tau }\in {\cal G}_{dis}$ and $\sharp \mbox{Con}(J(G))\leq \aleph _{0}.$ 
Then, by Theorem~\ref{randomthm1}-\ref{randomthm1c}, 
$T_{\infty ,\tau }:\CCI \rightarrow [0,1]$ is continuous and 
$T_{\infty ,\tau }(J(G_{\tau }))=[0,1].$ 
Thus there exists an element $J\in \mbox{Con}(J(G_{\tau }))$ such that 
$T_{\infty ,\tau }|_{J}$ is not constant. 
This illustrates the difference between $2$-generator semigroups in 
${\cal G}_{dis}$ (see Theorem~\ref{t:2gengdis}) and 
$m$-generator semigroups ($m\geq 3$) in ${\cal G}_{dis}$ (see Remark~\ref{r:gengdisj}). 
\end{rem}
\vspace{-4mm} 
\section{Background and tools}
\label{Background}
In this section, we give the known results and tools to prove the main results. 

{\bf (I)} We first explain the known results on general polynomial semigroups. 
Let $G$ be a polynomial semigroup in ${\cal P}$. Then $F(G)$ is an open subset of $\CCI $, 
$J(G)$ is a compact subset of $\CCI $, and 
for each $g\in G$, $g(F(G))\subset F(G)$ and $g^{-1}(J(G))\subset J(G).$ 
If $H$ is a subsemigroup of $G$, then $F(G)\subset F(H)$ and $J(H)\subset J(G).$ 
We set $E(G):= \{ z\in \CCI \mid \sharp G^{-1}(\{ z\} )<\infty \} .$ 
Then $\sharp E(G)\leq 2$ and for each $z\in \CCI \setminus E(G)$, 
$J(G)\subset \overline{G^{-1}(\{ z\} )}.$ In particular, 
for each $z\in J(G)\setminus E(G)$, $J(G)=\overline{G^{-1}(\{ z\} ) }.$ 
The Julia set $J(G)$ is a perfect set. 
The Julia set $J(G)$ is the unique minimal element in the space of all compact subsets $K$ of $\CCI $ with $\sharp K\geq 3$  
for which $g^{-1}(K)\subset K$ for each $g\in G.$ 
The Julia set $J(G)$ is equal to the closure of the set of repelling fixed points of elements of $G.$ In particular, 
$J(G) =\overline{\cup _{g\in G}J(g)}.$ For the proofs of these results, see \cite{HM}. 
Moreover, if $G=\langle h_{1},\ldots ,h_{m}\rangle $, then 
$J(G)=\cup _{j=1}^{m}h_{j}^{-1}(J(G))$ (see \cite[Lemma 1.1.4]{S1}). 
Moreover, it is easy to see that if $G$ is generated by a compact subset of ${\cal P}$, then 
$\infty \in F(G).$  

%{\bf (II)} We next explain the known results on the dynamics of $G\in {\cal G}_{dis}.$  
%Let $G\in {\cal G}_{dis}.$ Then, $\infty \in F(G)$ and 
% $(\mbox{Con}(J(G)), \leq _{s})$ is totally ordered. 
%Moreover, there exists a unique minimal element $J_{\min }(G) \in  (\mbox{Con}(J(G)), \leq _{s})$ 
%and a unique maximal element $J_{\max}(G)\in (\mbox{Con}(J(G)), \leq _{s}).$ 
%Each connected component of $F(G)$ is either simply or doubly connected. 
%$F_{\infty }(G)$ is simply connected. For the proofs of these results, see \cite{SdpbpI}. 
 
{\bf  (II)} We next explain the known results on the random dynamics of polynomials obtained in \cite{Splms10}.  
Let $\tau \in {\frak M}_{1,c}({\cal P}).$ 
Suppose $J_{\ker }(G_{\tau })=\emptyset $.  
Then there exists a non-empty finite dimensional subspace $U_{\tau }$ of $C(\CCI )$ with 
$M_{\tau }(U_{\tau })=U_{\tau }$ and a bounded operator $\pi _{\tau }: C(\CCI )\rightarrow U_{\tau }$ 
such that for each $\varphi \in C(\CCI )$, 
$M_{\tau }^{n}(\varphi -\pi _{\tau }(\varphi ))\rightarrow 0$ in $C(\CCI )$ as $n\rightarrow \infty .$ 
Moreover, $F_{meas}(\tau )={\frak M}_{1}(\CCI ).$ 
Moreover, there exist at least one and at most finitely many minimal sets of $G_{\tau }$. 
Moreover, for each minimal set $L$ of $G_{\tau }$, the function $T_{L, \tau }:\CCI \rightarrow [0,1]$ 
of probability of tending to $L$ is continuous on $\CCI $ and locally constant on $F(G).$ 
In particular, the function $T_{\infty ,\tau }: \CCI \rightarrow [0,1]$ 
is continuous on $\CCI $ and locally constant on $F(G_{\tau }).$ 
Moreover, denoting by $S_{\tau }$ the union of all minimal sets of $G_{\tau }$, 
we have that for each $z\in \CCI $, there exists a Borel subset ${\cal A}_{z}$ of ${\cal P}^{\NN} $ 
with $\tilde{\tau }({\cal A}_{z})=1$ such that 
for each $\g\in {\cal A}_{z}$, $d(\gamma _{n,1}(z),S_{\tau })\rightarrow  0$ as $n\rightarrow \infty .$ 
For the proofs of these results, see \cite[Theorem 3.15]{Splms10}. 
%Moreover, 
%assuming the hyperbolicity of $G_{\tau }$, under certain conditions some results on the non-differentiability of 
%$T_{\infty ,\tau }$ were obtained in \cite{Splms10}. Note that in the main result (Theorem~\ref{t:ihnondiffp}, 
%Theorem~\ref{t:hnondiffp}) of this paper we do not assume hyperbolicity of $G_{\tau }.$ 

 In the proofs of the main results of this paper, we combine the above results in {\bf (I)(II)} and some new careful observations on 
 the dynamics of $G\in {\cal G}_{dis}$ and associated random dynamics. 
\vspace{-3mm} 
\section{Proofs of the main results} 
\label{Proofs} 
%In this section, we prove the main results. 
\subsection{Proof of Theorem~\ref{randomthm1}} 
\label{Proofsrt1}
In this subsection, we prove Theorem~\ref{randomthm1}.  
We need several lemmas. 

\begin{lem}
\label{l:pbjkerem}
Let $G\in {\cal G}_{dis}$ (possibly generated by a non-compact subset of ${\cal P}$). 
Then, $\infty \in F(G), $ {\em int}$(\hat{K}(G))\neq \emptyset $, 
$F_{\infty }(G)\cup \mbox{{\em int}}(\hat{K}(G))\subset F(G)$, 
and 
for each $z\in \CCI $, there exists an element $g\in G$ with 
$g(z)\in F_{\infty }(G)\cup \mbox{{\em int}}(\hat{K}(G))\subset F(G).$ 
In particular, $J_{\ker }(G)=\emptyset .$ 
\end{lem}
\begin{proof}
By \cite[Theorem 2.20-1,5]{SdpbpI}, 
$\infty \in F(G)$ and int$(\hat{K}(G))\neq \emptyset .$ Moreover, 
by \cite[Proposition 2.19]{SdpbpI}, 
int$(\hat{K}(G))\subset F(G).$ 
%Let $J_{\max}$ be the maximal element in $(\mbox{Con}(J(G)),\leq )$, 
%and let $J_{\min }$ be the minimal element in $(\mbox{Con}(J(G)),\leq )$ 
%(see \cite[Theorem 2.20]{SdpbpI}). 
Let $z\in \CCI $ be a point. 
We consider the following three cases: Case 1: $z\not\in \hat{K}(G).$ 
Case 2: $z\in \mbox{int}(\hat{K}(G)).$ Case 3: $z\in \partial (\hat{K}(G)).$ 
 If we have Case 1, then there exists an element $g\in G$ with  
$g(z)\in F_{\infty }(G).$ 
If we have Case 2, then each element $h\in G$ satisfies 
$h(z)\in \mbox{int}(\hat{K}(G)).$ 
Suppose we have Case 3. Then, by \cite[Theorem 2.20-2]{SdpbpI}, 
$z\in \partial (\hat{K}(G))\subset J_{\min }(G).$ 
%Let $\epsilon \in \RR $ such that 
%$0<\epsilon <\min _{a\in J_{\min },b\in J_{\max}}d(a,b).$ 
By \cite[Theorem 2.1]{SdpbpI}, there exists an element $g\in G$ with 
%$J(g)\subset B(J_{\max},\epsilon ).$ 
%Then 
$J(g)\cap J_{\min }(G)=\emptyset .$ 
By \cite[Theorem 2.20-5(b)]{SdpbpI}, $g(J_{\min }(G))\subset \mbox{int}(\hat{K}(G)).$  
Thus $g(z)\in \mbox{int}(\hat{K}(G)).$ 
Therefore, we obtain that 
for each $z\in \CCI $, there exists an element $g\in G$ with  
$g(z)\in F_{\infty }(G)\cup \mbox{int}(\hat{K}(G))\subset F(G).$ 
Thus, $J_{\ker }(G)=\emptyset .$   
\end{proof}
\begin{lem}
\label{l:fmmmu}
Under the assumptions of Theorem~\ref{randomthm1}, 
statements~\ref{randomthm1c}, \ref{randomthm1-1}, \ref{randomthm1-5-1}--
%, \ref{randomthm1-min}, \ref{randomthm1az}, 
\ref{randomthm1-6} in Theorem~\ref{randomthm1} hold. 
\end{lem}
\begin{proof}
By Lemma~\ref{l:pbjkerem} and \cite[Theorem 3.14]{Splms10}, 
we obtain that $J_{\ker }(G_{\tau })=\emptyset $ and $F_{meas}(\tau )={\frak M}_{1}(\CCI ).$ Thus statement~\ref{randomthm1-5-1} holds. 
By \cite[Lemmas 5.24,  5.27, Theorem 3.31]{Splms10} and \cite[Theorem 1.9]{Scp}, statements \ref{randomthm1-1} and \ref{randomthm1c} in Theorem~\ref{randomthm1} hold. 

We now prove statements~\ref{randomthm1-min} and \ref{randomthm1az} in Theorem~\ref{randomthm1}. 
By \cite[Theorem 2.1]{SdpbpI}, there exists an element $g\in \mbox{supp}\, \tau $ with $J(g)\cap J_{\min }(G_{\tau })=\emptyset .$ 
 By \cite[Theorem 2.20-4,5]{SdpbpI}, int$(K(g))$ is connected and there exists an attracting fixed point $z_{g}$ of $g$ 
in int$(\hat{K}(G_{\tau }))$ such that int$(K(g))$ is the immediate attracting basin of $z_{g}$ for the dynamics of $g$ and 
$\hat{K}(G_{\tau })\subset \mbox{int}(K(g)).$  
Since $G_{\tau }(\hat{K}(G_{\tau }))\subset \hat{K}(G_{\tau })$, 
Zorn's lemma implies that there exists a minimal set $L_{0}$ for $(G_{\tau },\CCI )$ with $L_{0}\subset \hat{K}(G_{\tau }).$ 
Considering the dynamics of $g$ in $\hat{K}(G_{\tau })$, 
%it follows that 
we see that 
there exists a unique minimal set $L_{\tau }$ for $(G_{\tau },\CCI )$ with $L_{\tau }\subset \hat{K}(G_{\tau }).$ 
Therefore $\Min(G_{\tau },\CCI )=\{ \{ \infty \}, L_{\tau }\} .$ Thus statement~\ref{randomthm1-min} holds. 
Statement~\ref{randomthm1az} follows from statements~\ref{randomthm1-5-1}, \ref{randomthm1-min} 
and \cite[Theorem 3.15-5,15]{Splms10}. 

We now prove statement~\ref{randomthm1-6}. 
We again use the element $g$ in the previous paragraph. Since $g^{n}|_{L_{\tau }}\rightarrow z_{g}$ as 
$n\rightarrow \infty $, \cite[Theorem 3.15-12]{Splms10} implies that the number $r_{L_{\tau }}$ in 
\cite[Theorem 3.15-8]{Splms10} is equal to $1.$    By \cite[Theorem 3.15-1,2,9,13,15]{Splms10}, 
it follows that 
there exist two continuous linear functionals $\rho _{1},\rho _{2}:C(\CCI )\rightarrow \CC $ 
such that for each $\varphi \in C(\CCI )$, 
$$M_{\tau }^{n}(\varphi )\rightarrow \rho _{1}(\varphi )\cdot T_{\infty ,\tau }+\rho _{2}(\varphi )\cdot T_{L_{\tau },\tau } 
\mbox{ in } C(\CCI ) \mbox{ as } n\rightarrow \infty ,$$  
and such that supp$\, \rho _{1}=\{ \infty \} $ and supp$\, \rho _{2}=L_{\tau }.$  
From this, it is easy to see that $\rho _{1}=\delta _{\infty }$ and $\rho _{2}$ is a Borel probability measure 
on $\CCI .$ Moreover, by \cite[Theorem 3.15-15]{Splms10}, we obtain that 
$T_{\infty ,\tau }(z)+T_{L_{\tau },\tau }(z)=1$ for each $z\in \CCI .$ 
From these arguments, statement~\ref{randomthm1-6} 
%in Theorem~\ref{randomthm1} 
holds. 
\end{proof} 
\begin{lem}
\label{l:fmcmono}
Let $\tau \in {\frak M}_{1}({\cal P}).$ 
Suppose that $\infty \in F(G_{\tau }).$ 
Let $U$ be a multiply connected component of $F(G_{\tau }).$ 
Let $B$ be a bounded component of $\CC \setminus U.$ 
Let $y\in B$ and let $z\in U.$ 
Then, for any $\gamma \in X_{\tau }$  
with $\gamma _{n,1}(y)\rightarrow \infty $ as 
$n\rightarrow \infty $, we have  
$\gamma _{n,1}(z)\rightarrow \infty $ as $n\rightarrow \infty .$ 
%Then, for each $y\in B$ and each $z\in U$, 
In particular, 
$T_{\infty ,\tau }(y)\leq T_{\infty ,\tau }(z).$ 
\end{lem}
\begin{proof}
Suppose that $\gamma _{n,1}(y)\rightarrow \infty $ as 
$n\rightarrow \infty $. 
Let $\zeta $ be a Jordan curve (i.e. simple closed curve) in $U$ such that 
$y$ belongs to the bounded component of $\CC \setminus \zeta .$ 
By the maximum principle, \cite[Lemma 5.24]{Splms10} and 
forward invariance of $F(G_{\tau })$ under any element of $G_{\tau }$,   
we obtain that $\gamma _{n,1}\rightarrow \infty $ 
as $n\rightarrow \infty $ on $\zeta .$ 
Hence, $\gamma _{n,1}(z)\rightarrow \infty $ as $n\rightarrow \infty .$ 
%Thus, we have completed the proof of our lemma.  
\end{proof}

\begin{prop}
\label{p:fmcnotconst}
Let $\tau \in {\frak M}_{1}({\cal P}).$   
Let $U$ be a multiply connected component of $F(G_{\tau }).$ 
Let $C$ be the boundary of a bounded component of $\CC \setminus  U.$  
Let $V$ be an open subset of $\CCI $ such that $V\cap  C \neq \emptyset .$ 
Then, we have the following.
\begin{enumerate}
\item \label{p:fmcnotconst1}
If $\infty \in F(G_{\tau })$ and {\em int}$(\hat{K}(G_{\tau }))\neq \emptyset $, then 
$T_{\infty ,\tau }|_{V}$ is not constant. 
\item \label{p:fmcnotconst2}
If supp$\, \tau $ is compact, $\sharp \mbox{{\em supp}}\,\tau \leq \aleph _{0} $ and $\hat{K}(G_{\tau })\neq \emptyset $, 
then $T_{\infty ,\tau }|_{V}$ is not constant. 
\end{enumerate}
\end{prop}
\begin{proof}We may assume that $V$ does not meet the unbounded component of $\CCI \setminus U.$ 
We first prove statement~\ref{p:fmcnotconst1}. 
Suppose that $\infty \in F(G_{\tau })$ and int$(\hat{K}(G_{\tau }))\neq \emptyset $. 
Let $y\in V\cap C.$ 
Let $\zeta $ be a Jordan curve in $U$ such that 
$y$ belongs to the bounded component $A$ of $\CC \setminus \zeta .$ 
%We may assume that $V\subset A.$ 
Since $C\subset J(G_{\tau })$, 
\cite[Corollary 3.1]{HM} implies that there exists a $g \in G_{\tau }$ 
such that $J(g)\cap V\cap A\neq \emptyset .$ 
Then, $\zeta \subset F_{\infty }(g).$ For, 
if $\zeta \subset \mbox{int}K(g)$, then the maximum principle implies that 
$A\subset  F(g)$, which is a contradiction. Hence, 
$\zeta \subset F_{\infty }(g).$ Therefore,  
%\begin{equation}
%\label{eq:p:fmcncpf1}
$g^{n} \rightarrow \infty \mbox{ as } n\rightarrow \infty \mbox{ on }U.
$
%\end{equation} 
Since $J(g)\cap V\cap A\neq \emptyset $ and int$(\hat{K}(G_{\tau }))\neq \emptyset $, 
there exists a point $y_{1}\in V\cap A$ and an $l\in \NN $ such that 
$g^{l}(y_{1})\in \mbox{int}(\hat{K}(G_{\tau })).$ 
Let $y_{2}\in U\cap V$ be a point. 
%By (\ref{eq:p:fmcncpf1}), 
We may assume that  
$g^{l}(y_{2})\in F_{\infty }(G_{\tau }).$ 
Let $\{ \g _{j} \} _{j=1}^{m}$ be a finite sequence of elements of 
$\G _{\tau }$ such that 
$g^{l}=\g _{m}\circ \cdots \circ \g _{1}.$ 
Then, there exists a neighborhood $W$ of $(\g _{1},\ldots ,\g _{m})$ in $\G _{\tau }^{m}$ such that 
for each $\alpha =(\alpha _{1},\ldots ,\alpha _{m})\in W$, 
$\alpha _{m}\cdots \alpha _{1}(y_{1})\in \mbox{int}(\hat{K}(G_{\tau }))$ and 
$\alpha _{m}\cdots \alpha _{1}(y_{2})\in F_{\infty }(G_{\tau }).$ 
We set  
$Z:= \{ \rho =(\rho _{1},\rho _{2},\ldots )\in X _{\tau }\mid (\rho _{1},\ldots ,\rho _{m})\in W\} $.  
Then, 
for each $\omega =(\omega _{1},\omega _{2},\ldots )\in Z $, 
$\{ \omega _{r,1}(y_{1})\} _{r\in \NN }$ is bounded in $\CC $ and 
$\omega _{r,1}(y_{2})\rightarrow \infty $ as $r\rightarrow  \infty .$ 
Hence, $y_{1}$ belongs to a bounded component $B$ of $\CC \setminus U.$ 
By Lemma~\ref{l:fmcmono}, 
$\{ \rho \in X_{\tau }\mid \rho _{n,1}(y_{1})\rightarrow \infty \} 
\subset \{ \rho \in X_{\tau }\mid \rho _{n,1}(y_{2})\rightarrow \infty \} .$ 
From these arguments, 
it follows that 
$T_{\infty ,\tau }(y_{1})+ \tilde{\tau }(Z)\leq T_{\infty ,\tau }(y_{2}).$ 
Since $\tilde{\tau }(Z)>0$, 
we obtain that $T_{\infty ,\tau }(y_{1})<T_{\infty ,\tau }(y_{2}).$ 
Therefore, $T_{\infty ,\tau }|_{V}$ is not constant. 
Thus, we have proved statement~\ref{p:fmcnotconst1}. 

% Statement~\ref{p:fmcnotconst2} follows from a similar argument to the above.
We now prove statement~\ref{p:fmcnotconst2}. 
Let $\zeta $ be a Jordan curve in $U$ such that 
$y$ belongs to the bounded component $A$ of $\CC \setminus \zeta .$ 
We now show the following claim 1:\\ 
Claim 1: There exist a $g\in G_{\tau }$, 
an $l\in \NN $, and a point $y_{1}\in V\cap A$ such that 
$J(g)\cap V\cap A\neq \emptyset $ and $g^{l}(y_{1})\in \hat{K}(G_{\tau }).$   

In order to show claim 1, 
we consider the following two cases.  
Case 1. $\sharp \hat{K}(G_{\tau })\geq 2.$  
Case 2. $\sharp \hat{K}(G_{\tau })=1.$ 

 Suppose that we have case 1. 
By \cite[Corollary 3.1]{HM}, there exists a $g\in G_{\tau }$ such that 
$J(g)\cap V\cap A\neq \emptyset .$ Since $\sharp \hat{K}(G_{\tau })\geq 2$ and 
$\cup _{n\in \NN }g^{n}(V\cap A)\subset \CC $, 
Montel's theorem implies that there exists an $l \in \NN $ and 
a point $y_{1}\in V\cap A$ such that 
$g^{l}(y_{1})\in \hat{K}(G_{\tau }).$ 
Hence, the statement of claim 1 holds when we have case 1. 

 Suppose that we have case 2. Let $z_{0}\in \CC $ be such that 
$\hat{K}(G_{\tau })=\{ z_{0}\} .$ 
By \cite[Lemma 5.28]{Splms10}, $h(z_{0})=z_{0}$ for each $h\in \G_{\tau }$ and 
$z_{0}\in J(G_{\tau }).$ 
Since $\G _{ \tau }$ is compact, there exists an element $\beta _{1}\in \G_{\tau }$ such that 
$z_{0}\not\in E(\beta _{1})$, 
where $E(\beta _{1})$ denotes the exceptional set of $\beta _{1}.$ 
 Moreover, 
  \cite[Corollary 3.1]{HM} implies that there exists an element $\beta _{2}\in G_{\tau }$ 
such that $J(\beta _{2})\cap V\cap A\neq \emptyset .$ 
By \cite[Proposition 2.2 (3)]{S7}, 
there exists a $p\in \NN $ such that $J(\beta _{1}\beta _{2}^{p})\cap V\cap A\neq \emptyset .$ 
Let $g:=\beta _{1}\beta _{2}^{p}.$ 
Since $h(z_{0})=z_{0}$ for each $h\in G_{\tau }$ and $z_{0}\not\in E(\beta _{1})$, 
we obtain that $z_{0}\not\in E(g).$ 
Therefore, there exist an $l\in \NN $ and a point $y_{1}\in V\cap A$ such that 
$g^{l}(y_{1})=z_{0}\in \hat{K}(G_{\tau }).$ Thus, we have shown claim 1.  
  
 Let $(g,l,y_{1})$ be as in claim 1. Let $y_{2}\in U\cap V$ be a point. 
Since $J(g)\cap V\cap A\neq \emptyset $, 
the maximum principle implies that $g^{n}\rightarrow \infty $ as $n\rightarrow \infty $ 
on $U.$  Hence, we may assume that 
$g^{l}(y_{2})\in F_{\infty }(G_{\tau }). $ 
Therefore $g^{l}(y_{1})\in \hat{K}(G_{\tau })$,  $g^{l}(y_{2})\in F_{\infty }(G_{\tau })$ and 
$y_{1}$ belongs to a bounded component $B$ of $\CC \setminus U.$ 
Combining this with the method in the proof of statement~\ref{p:fmcnotconst1}, 
we obtain that $T_{\infty ,\tau }(y_{1})<T_{\infty ,\tau }(y_{2}).$ 
Therefore, $T_{\infty ,\tau }|_{V}$ is not constant. 
%Let $\{ \g _{j}\} _{j=1}^{m}$ be a finite sequence of 
%elements of $\G _{\tau }$ such that 
%$g^{l}=\g _{m}\circ \cdots \circ \g _{1}.$ 
%We set $Z':= \{ \rho =(\rho _{1},\rho _{2},\ldots )\in X_{\tau }\mid \rho _{j}=\g _{j} (\forall j=1,\ldots ,m)\} .$ 
%Then, for each $\omega \in Z'$, 
%$\{ \omega _{r}\cdots \omega _{1}(y_{1})\} _{r\in \NN }$ is bounded in $\CC 
%$ and $\omega _{r}\cdots \omega _{1}(y_{2})\rightarrow \infty $ as $r\rightarrow \infty .$ 
%Hence, $y_{1}$ belongs to a bounded component $B$ of $\CC \setminus U.$ 
%By Lemma~\ref{l:fmcmono}, 
%$\{ \rho \in X_{\tau }\mid \rho _{n}\cdots \rho _{1}(y_{1})\rightarrow \infty \} 
%\subset \{ \rho \in X_{\tau }\mid \rho _{n}\cdots \rho _{1}(y_{2})\rightarrow \infty \} .$ 
%From these arguments, 
%it follows that 
%$T_{\infty ,\tau }(y_{1})+ \tilde{\tau }(Z')\leq T_{\infty ,\tau }(y_{2}).$ 
%Since $\tilde{\tau }(Z')>0$, 
%we obtain that $T_{\infty ,\tau }(y_{1})<T_{\infty ,\tau }(y_{2}).$ 
%Therefore, $T_{\infty ,\tau }|_{V}$ is not constant. 
Thus, we have proved statement~\ref{p:fmcnotconst2}. 
%Thus, we have completed the proof of Proposition~\ref{p:fmcnotconst}. 
\end{proof}
\begin{cor}
\label{c:fmcnc}
Let $\tau \in {\frak M}_{1,c}({\cal P}).$ Suppose 
that $\hat{K}(G_{\tau })\neq \emptyset $ and $J_{\ker }(G_{\tau })= \emptyset .$    
Let $U$ be a multiply connected component of $F(G_{\tau }).$ 
Let $C$ be the boundary of a bounded component of $\CC \setminus  U.$
%Let $C$ be a bounded component of $\partial U.$  
Let $V$ be an open subset of $\CCI $ such that $V\cap  C \neq \emptyset .$ 
Then, $T_{\infty ,\tau }:\CCI \rightarrow [0,1]$ is continuous and 
$T_{\infty ,\tau }|_{V}$ is not constant. 
\end{cor}
\begin{proof}
Since supp$\,\tau $ is compact, we have $\infty \in F(G_{\tau }).$ 
By \cite[Theorem 3.31]{Splms10}, int$\hat{K}(G_{\tau })\neq \emptyset .$ 
By Proposition~\ref{p:fmcnotconst}, it follows that $T_{\infty ,\tau }|_{V}$ is not constant. 
Moreover, by \cite[Theorem 3.22]{Splms10}, $T_{\infty ,\tau }:\CCI \rightarrow [0,1]$ is continuous. 
\end{proof}
%\begin{df}
%For any $\gamma =(\gamma _{1},\gamma _{2}\ldots )\in {\cal P}^{\NN }$, 
%we set $K_{\gamma }=\{ z\in \CC \mid \{ \gamma _{n,1}(z)\} _{n=1}^{\infty } \mbox{ is bounded in }\CC \} .$ 
%This is called the {\bf filled-in Julia set of } the sequence $\gamma .$  
%Moreover, we set $A_{\infty ,\gamma }:= \{ z\in \CCI \mid \gamma _{n,1}(z)\rightarrow \infty  
%\mbox{ as }n\rightarrow \infty \} . $  
%\end{df}
\begin{lem}
\label{l:pbkg}
Let $\Gamma $ be a subset of ${\cal P}$ and let $G=\langle \Gamma \rangle .$ 
Suppose $G\in {\cal G}_{dis}.$ 
Then for each $\gamma \in \Gamma ^{\NN }$, $K_{\g } $ is a connected compact subset $\CC $,
 $A_{\infty ,\g }$ is a simply connected domain, and $K_{\g }\cup A_{\infty ,\g }=\CCI .$   
\end{lem}
\begin{proof}
Since $G\in {\cal G}_{dis}$, by \cite[Theorem 2.20]{SdpbpI} we have $\infty \in F(G).$ 
%We take the hyperbolic distance on $F_{\infty }(G).$ 
For each  $r>0$, we denote by  $B_{h}(\infty ,r)$ the ball with center $\infty $ and radius $r$ with respect to 
the hyperbolic distance on $F_{\infty }(G).$ Then 
$g(B_{h}(\infty ,r))\subset B_{h}(\infty ,r)$ for each $g\in G.$ 
Let $r>0$ be small enough such that $B_{h}(\infty ,r)$ is simply connected. Let $B:=B_{h}(\infty ,r).$  
By \cite[Lemma 5.24]{Splms10}, 
for each $\alpha \in \G ^{\NN }$, $\alpha _{n,1}\rightarrow \infty $ uniformly on $B$ as $n\rightarrow \infty .$ 
Therefore for each $\g \in \G ^{\NN }, K_{\g }\cup A_{\infty ,\g }=\CCI $ and   
$A_{\infty, \g }$ is an open neighborhood of $\infty .$ By the maximum principle, 
$A_{\infty ,\g }$ is connected. Moreover, $A_{\infty, \g }=\cup _{n=1}^{\infty }(\g _{n,1})^{-1}(B).$ 
Since $G\in {\cal G}$, each $\g _{n,1}^{-1}(B)$ is a simply connected domain. Thus 
$A_{\infty ,\g }$ is the union of increasing simply connected domains $\gamma _{n,1}^{-1}(B).$   Therefore 
$A_{\infty ,\g }$ is simply connected. Thus $K_{\g }$ is connected.   
\end{proof}

\begin{lem}
\label{l:y1y2}
Let $\tau \in {\frak M}_{1}({\cal P}).$ 
Suppose $G_{\tau }\in {\cal G}_{dis}.$ 
Let $A$ be a doubly connected component of $F(G_{\tau }).$ 
Let $y_{1}\in A$ and let $y_{2}$ be a point in the unbounded component of 
$\CC \setminus A.$ 
Then, we have the following.
\begin{enumerate}
\item
\label{l:y1y2-1}
For any $\gamma \in X_{\tau }$ with $\gamma _{n,1}(y_{1})\rightarrow \infty $ as $n\rightarrow \infty $, 
we have $\gamma _{n,1}(y_{2})\rightarrow \infty $ as $n\rightarrow \infty .$ In particular, 
$T_{\infty ,\tau }(y_{1})\leq T_{\infty ,\tau }(y_{2}).$
\item
\label{l:y1y2-2}
In addition to the assumptions of our lemma, 
suppose $y_{2}\in F(G_{\tau }).$ Let $U$ be the connected component 
of $F(G_{\tau })$ containing $y_{2}.$ Suppose that either $U$ is doubly connected or 
$U=F_{\infty }(G_{\tau }).$ Then $T_{\infty ,\tau }(y_{1})<T_{\infty ,\tau }(y_{2}).$ 
\end{enumerate}
\end{lem}
\begin{proof}
We first prove statement~\ref{l:y1y2-1}. 
Since $G_{\tau }\in {\cal G}_{dis}$, by \cite[Theorem 2.20]{SdpbpI} we have $\infty \in F(G_{\tau }).$ 
Let $\gamma \in X_{\tau }$ and suppose $\gamma _{n,1}(y_{1})\rightarrow \infty $ as $n\rightarrow \infty .$ 
By \cite[Lemma 5.24]{Splms10}, $\gamma _{n,1}\rightarrow \infty $ locally uniformly on $A$ as $n\rightarrow \infty .$ 
Therefore $K_{\g }\subset \CC \setminus A.$ Thus $\partial \hat{K}(G_{\tau })\subset K_{\g }\subset \CC \setminus A.$ 
By \cite[Theorem 2.20-2]{SdpbpI}, $\partial \hat{K}(G_{\tau })\subset J_{\min }(G_{\tau }).$ 
Since $J_{\min }(G_{\tau })$ is included in the bounded component of $\CC \setminus A$, 
and since $K_{\g }$ is connected (see Lemma~\ref{l:pbkg}), 
it follows that $K_{\g }$ is included in the bounded component of $\CC \setminus A.$ 
Therefore $\g _{n,1}(y_{2})\rightarrow \infty $ as $n\rightarrow \infty .$ 
Thus we have proved statement~\ref{l:y1y2-1}. 

 We now prove statement~\ref{l:y1y2-2}. We prove the following claim:\\ 
Claim: There exists a map $g\in G_{\tau }$ such that 
$g(y_{1})\in \mbox{int}(\hat{K}(G_{\tau }))$ and $g(y_{2})\in F_{\infty }(G_{\tau }).$

 To prove this claim, let $B_{1} $ and $B_{2}$ be the two connected components of $\partial A$. 
We may assume $B_{2}<_{s}B_{1}.$ 
For each $i=1,2,$ let $B_{i}'\in \mbox{Con}(J(G_{\tau }))$ with $B_{i}\subset B_{i}'.$ 
Then $B_{2}'<_{s}B_{1}'.$ Therefore $J_{\min }(G_{\tau })\leq _{s}B_{2}'<_{s}B_{1}'.$ 
Let $D$ be a bounded, doubly connected, and open neighborhood of $B_{1}'$ such that 
$J_{\min }(G_{\tau })\cup \{ y_{1}\} <_{s}\overline{D}$ and $y_{2}$ belongs to the unbounded component of $\CC \setminus \overline{D}.$ 
By \cite[Lemma 4.2]{SdpbpI}, there exists an element $h\in G_{\tau }$ with 
$J(h)\subset D.$ Then $J(h)\cap J_{\min }(G_{\tau })=\emptyset .$ 
Moreover, $y_{2}\in F_{\infty }(h).$ 
By \cite[Theorem 2.20-4,5]{SdpbpI}, int$(K(h))$ is connected and is an immediate basin of an attracting 
fixed point $z_{h}$ of $h$, and  $z_{h}\in \mbox{int}(\hat{K}(G_{\tau })).$ 
Since $\partial \hat{K}(G_{\tau })\subset J_{\min }(G_{\tau })$ (\cite[Theorem 2.20-2]{SdpbpI}), 
$\{ z_{h}\} <_{s}J_{\min }(G_{\tau })<_{s}\overline{D}.$ 
Since $z_{h}$ belongs to the bounded  component of $\CC \setminus J(h)$, 
it follows that $y_{1}$ belongs to the bounded component of $\CC \setminus J(h).$ 
Therefore, there exists an $n\in \NN $ such that 
$h^{n}(y_{1})\in \mbox{int}(\hat{K}(G_{\tau }))$ and $h^{n}(y_{2})\in F_{\infty }(G_{\tau }).$ 
Thus, we have proved the claim. 

 Let $g\in G_{\tau }$ be the element in the above claim. Let 
 $h_{1},\ldots ,h_{n}\in \G _{\tau } $ be some elements such that 
 $g=h_{n}\circ \cdots \circ h_{1}.$ 
 Then there exists a neighborhood $W$ of $(h_{1},\ldots ,h_{n})$ in $\G _{\tau }^{n}$ 
such that for each $\omega =(\omega _{1},\ldots ,\omega _{n})\in W$, 
$\omega_{n}\cdots \omega _{1}(y_{1})\in \mbox{int}(\hat{K}(G_{\tau }))$ and 
$\omega _{n}\cdots \omega _{1}(y_{2})\in F_{\infty }(G_{\tau }).$ 
Therefore, for each $\g =(\g _{1},\g _{2},\ldots )\in X_{\tau }$ with $(\g _{1},\ldots,\g _{n})\in W$, 
we have that $\{ \g_{r,1}(y_{1})\} _{r\in \NN }$ is bounded and that 
$\g _{r,1}(y_{2})\rightarrow \infty $ as $r\rightarrow \infty .$ 
 Combining it with statement~\ref{l:y1y2-1}, we get  
 $T_{\infty ,\tau }(y_{1})+\tilde{\tau }(\{ \g \in X_{\tau }\mid (\g _{1},\ldots ,\g _{n})\in W\} )\leq 
T_{\infty , \tau }(y_{2}).$ 
Since $ \tilde{\tau }(\{ \g \in X_{\tau }\mid (\g _{1},\ldots ,\g _{n})\in W\} )>0$, 
we obtain $T_{\infty ,\tau }(y_{1})<T_{\infty ,\tau }(y_{2}).$ 
Therefore we have proved statement~\ref{l:y1y2-2}.
%Thus we have proved Lemma~\ref{l:y1y2}.
\end{proof} 
\begin{lem}
\label{l:j1j2}
Let $\tau \in {\frak M}_{1}({\cal P}).$ 
Suppose $G_{\tau }\in {\cal G}_{dis}.$ Let 
$J_{1},J_{2}\in \mbox{{\em Con}}(J(G_{\tau }))$ with $J_{1}<_{s}J_{2}.$ 
Then $\sup _{z\in J_{1}}T_{\infty ,\tau }(z)\leq \inf _{z\in J_{2}}T_{\infty ,\tau }(z).$ 
\end{lem}
\begin{proof}
By \cite[Theorem 2.20-1]{SdpbpI}, 
$\infty \in F(G_{\tau }).$  
By \cite[Lemma 4.4]{SdpbpI}, there exists a doubly connected component $A$ of $F(G_{\tau })$ with 
$J_{1}<_{s}A<_{s}J_{2}.$ By Lemmas~\ref{l:fmcmono}, \ref{l:y1y2}, 
it follows that $\sup _{z\in J_{1}}T_{\infty ,\tau }(z)\leq \inf _{z\in J_{2}}T_{\infty ,\tau }(z).$
\end{proof}

\begin{lem}
\label{l:y1pos}
Let $\tau \in {\frak M}_{1}({\cal P})$ and suppose $\infty \in F(G_{\tau }).$ 
Let $A\in \mbox{{\em Con}}(F(G_{\tau }))$ be multiply connected and let $y_{1}\in A.$ 
Then $T_{\infty ,\tau }(y_{1})>0.$ 
\end{lem}
\begin{proof}
Let $K$ be a bounded component of $\CC \setminus A$ and 
let $B\in \mbox{Con}(J(G_{\tau }))$ 
%be such that 
such that 
$\partial K \subset B. $ Let $D$ be a bounded neighborhood of $B$ such that 
$y_{1}$ belongs to the unbounded component of $\CC \setminus \overline{D}.$ 
By \cite[Corollary 3.1]{HM}, there exists an element $\alpha \in G_{\tau }$ with 
$J(\alpha )\cap D\neq \emptyset .$ 
%Let $U\in \mbox{Con}(F(G_{\tau }))$ with $y_{1}.$ 
By the maximum principle, 
$A\subset F_{\infty }(\alpha ).$ Therefore, there exists an $m\in \NN $ such that 
$\alpha ^{m}(y_{1})\in F_{\infty }(G_{\tau }).$ 
 Let 
 $h_{1},\ldots ,h_{n}\in \G _{\tau }$ be some elements such that 
 $\alpha ^{m}=h_{n}\circ \cdots \circ h_{1}.$ 
 Then there exists a neighborhood $W$ of $(h_{1},\ldots ,h_{n})$ in $\G _{\tau }^{n}$ 
such that for each $\omega =(\omega _{1},\ldots ,\omega _{n})\in W$, 
%$\omega_{n}\cdots \omega _{1}(y_{1})\in \mbox{int}(\hat{K}(G_{\tau }))$ and 
$\omega _{n}\cdots \omega _{1}(y_{1})\in F_{\infty }(G_{\tau }).$ 
Therefore, for each $\g \in X_{\tau }$ with $(\g _{1},\ldots ,\g _{n})\in W$, 
$\g _{r,1}(y_{1})\rightarrow \infty $ as $r\rightarrow \infty .$ 
Thus $T_{\infty ,\tau }(y_{1})\geq \tilde{\tau }(\{ \g \in X_{\tau }\mid 
(\g _{1},\ldots ,\g _{n})\in W\} )>0.$ 
\end{proof}
\begin{cor}
\label{c:pby1pos}
Let $\tau \in {\frak M}_{1}({\cal P})$ and suppose 
$G_{\tau }\in {\cal G}_{dis}.$ 
Let $A\in \mbox{{\em Con}}(F(G_{\tau }))$ be doubly connected. 
Let $y_{1}\in A.$ Then  $T_{\infty ,\tau }(y_{1})>0.$  
\end{cor}
\begin{proof}
By Lemma~\ref{l:y1pos} and \cite[Theorem 2.20-1]{SdpbpI}, the statement of our lemma holds.
\end{proof}
\begin{lem}
\label{l:outernc}
Let $\tau \in {\frak M}_{1}({\cal P}).$ Suppose $G_{\tau }\in {\cal G}_{dis }.$ 
Let $A\in \mbox{{\em Con}}(F(G_{\tau }))$ be doubly connected. 
Let $Q$ be an open subset of $\CCI $ with $Q\cap \partial A\neq \emptyset .$ 
Then $T_{\infty ,\tau }|_{Q}$ is not constant. 
\end{lem}
\begin{proof}
Let $B_{1}$ and $B_{2}$ be the two connected components of $\partial A.$ 
Let $B_{2}<_{s}B_{1}.$ If $Q\cap B_{2}\neq \emptyset $, 
then by Lemma~\ref{p:fmcnotconst}-\ref{p:fmcnotconst1} and 
\cite[Theorem 2.20-1,5]{SdpbpI}, $T_{\infty ,\tau }|_{Q}$ is not constant. 
Therefore we may assume $Q\cap B_{1}\neq \emptyset .$ We may also assume that 
$Q$ is a disk and 
$Q\cap B_{2}=\emptyset .$ Since $Q\cap J(G_{\tau })\neq \emptyset $, 
by \cite[Corollary 3.1]{HM} there exists an element $g\in G_{\tau }$ such that 
$J(g)\cap Q\neq \emptyset .$ Since $J_{\min }(G_{\tau })\leq _{s}B_{2}$, 
$J(g)\cap J_{\min }(G_{\tau })=\emptyset .$ 
By \cite[Theorem 2.20-4,5]{SdpbpI}, it follows that $J(g)$ is a quasicircle and 
there exists an attracting fixed point $z_{g}\in \mbox{int}(\hat{K}(G_{\tau }))$ 
of $g$. By \cite[Theorem 2.20-2]{SdpbpI},
 $\partial \hat{K}(G_{\tau })\subset J_{\min }(G_{\tau })\leq _{s}B_{2}<_{s}A$. 
Therefore $\{ z_{g}\} <_{s}A.$ 
Since $J(g)\cap A=\emptyset $, it follows that $A\subset \mbox{int}K(g).$ 
Let $y_{1}\in A$ be a point. From the above arguments, we obtain that 
there exists a number $n_{1}\in \NN $ such that for each $n\in \NN $ with $n\geq n_{1}$, 
$g^{n}(y_{1})\in \mbox{int}(\hat{K}(G_{\tau })).$ 
Moreover, since $J(g)\cap Q\neq \emptyset $, there exists a point $y_{2}\in Q$ and a number 
$n_{2}\in \NN $ such that for each $n\in \NN $ with $n\geq n_{2}$, 
$g^{n}(y_{2})\in F_{\infty }(G_{\tau }).$ 
Let $m:=\max \{ n_{1},n_{2}\} $.  
Let $\alpha _{1},\ldots ,\alpha _{p}\in \G_{\tau }$ be some elements  
such that $g^{m}=\alpha _{p}\circ \cdots \circ \alpha _{1}.$ 
 Let $W$ be a neighborhood of $(\alpha _{1},\ldots \alpha _{p})$ 
 in $\G_{\tau }^{p}$ such that 
for each $\omega =(\omega _{1}, \ldots ,\omega _{p})\in W$, 
$\omega _{p}\cdots \omega _{1}(y_{1})\in \mbox{int}(\hat{K}(G_{\tau }))$ and 
$\omega _{p}\cdots \omega _{1}(y_{2})\in F_{\infty }(G_{\tau }).$ 
Therefore, for each $\gamma \in X_{\tau }$ with $(\gamma _{1},\ldots ,\gamma _{p})\in W$, 
$\{ \gamma _{r,1}(y_{1})\} _{r\in \NN }$ is bounded and 
$\g _{r,1}(y_{2})\rightarrow \infty $ as $n\rightarrow \infty .$ 
Combining this with Lemma~\ref{l:y1y2}-\ref{l:y1y2-1}, 
%it follows that 
we see that 
$T_{\infty ,\tau }(y_{1})<T_{\infty ,\tau }(y_{1})+\tilde{\tau }(\{ \g \in X_{\tau }\mid 
(\g _{1},\ldots ,\g _{p})\in W\} )\leq T_{\infty ,\tau }(y_{2}).$   
Therefore, $T_{\infty ,\tau }|_{Q}$ is not constant. 
Thus, we have proved our lemma.  
\end{proof}
\begin{lem}
\label{l:intT1}
Let $\tau \in {\frak M}_{1}({\cal P}).$ 
Suppose $\mbox{{\em int}}(\hat{K}(G_{\tau }))\neq \emptyset .$ 
Then we have the following. 
\begin{enumerate}
\item \label{l:intT1-1}
$\mbox{{\em int}}(T_{\infty, \tau }^{-1}(\{ 1\}))\subset F(G_{\tau }).$ 
\item \label{l:intT1-2}
If, in addition to the assumption of our lemma, $\infty \in F(G_{\tau })$, 
then for each open subset $Q$ of $\CCI $ with $Q\cap \partial F_{\infty }(G_{\tau })\neq \emptyset $, 
$T_{\infty ,\tau }|_{Q}$ is not constant. 
\end{enumerate}
\end{lem}
\begin{proof}
We first prove statement~\ref{l:intT1-1}. 
We prove the following claim.\\ 
Claim. For each $z_{0}\in T_{\infty ,\tau }^{-1}(\{ 1\} )$, 
there exists no $g\in G_{\tau }$ with $g(z_{0})\in \mbox{int}(\hat{K}(G_{\tau })).$

To prove this claim, let $z_{0}\in T_{\infty ,\tau }^{-1}(\{ 1\} )$ and 
suppose there exists an element $g\in G_{\tau }$ with $g(z_{0})\in \mbox{int}(\hat{K}(G_{\tau })).$ 
Let $h_{1},\ldots ,h_{m}\in \G_{\tau }$ be some elements with 
$g=h_{m}\circ \cdots \circ h_{1}.$ 
Then there exists a neighborhood $W$ of $(h_{1},\ldots ,h_{m})$ in $\G_{\tau }^{m}$ 
such that for each $\omega =(\omega _{1},\ldots ,\omega _{m})\in W$, 
$\omega _{m}\cdots \omega _{1}(z_{0})\in \mbox{int}(\hat{K}(G_{\tau })).$ 
Therefore for each $\g \in X_{\tau }$ with $(\g _{1},\ldots ,\g _{m})\in W$, 
$\{ \g _{n,1}(z_{0})\} _{n\in \NN }$ is bounded. 
Thus 
$T_{\infty ,\tau }(z_{0})\leq 1-\tilde{\tau }(\{ \g \in X_{\tau }\mid (\g _{1},\ldots ,\g _{m})\in W\} )<1.$ 
This is a contradiction. Hence we have proved the claim.

 From this claim, $G_{\tau }(\mbox{int}(T_{\infty ,\tau }^{-1}(\{ 1\} )))\subset \CCI \setminus \mbox{int}(\hat{K}(G_{\tau })).$ 
Therefore $ \mbox{int}(T_{\infty ,\tau }^{-1}(\{ 1\} ))\subset F(G_{\tau }).$ 
Thus we have proved statement~\ref{l:intT1-1}. 

 We now prove statement~\ref{l:intT1-2}. 
 Suppose $\infty \in F(G_{\tau }).$ Let $Q$ be an open subset of $\CCI $ with $Q\cap \partial F_{\infty }(G_{\tau })\neq \emptyset .$ 
By \cite[Lemma 5.24]{Splms10}, $T_{\infty ,\tau }|_{F_{\infty }(G_{\tau })}\equiv 1$. 
Combining this with statement~\ref{l:intT1-1}, we obtain that $T_{\infty ,\tau }|_{Q}$ is not constant.  
Thus we have proved statement~\ref{l:intT1-2}. 
% Thus we have proved our lemma. 
\end{proof}
\begin{lem}
\label{l:ksetnc}
Let $\tau \in {\frak M}_{1}({\cal P}).$ Suppose $\infty \in F(G_{\tau }).$ 
Then $\mbox{{\em int}}(T_{\infty ,\tau }^{-1}(\{ 0\})\subset F(G_{\tau })$, and 
for each open subset $Q$ of $\CCI $ with $Q\cap \partial \hat{K}(G_{\tau })\neq \emptyset $, 
$T_{\infty ,\tau }|_{Q}$ is not constant. 
\end{lem}
\begin{proof}
We can prove this lemma in the same way as that in the proof of Lemma~\ref{l:intT1}. 
\end{proof}
\begin{thm}
\label{t:pbgen} 
Let $\tau \in {\frak M}_{1}({\cal P})$ (we do not assume that {\em supp}$\,\tau$ is compact). 
Suppose $G_{\tau }\in {\cal G}_{dis}.$ 
Then statements ~\ref{randomthm1-1}, \ref{randomthm1m}, \ref{randomthm1tca} and \ref{randomthm1q} in Theorem~\ref{randomthm1} 
hold. 
\end{thm}
\begin{proof}
By \cite[Theorem 2.20-1,5]{SdpbpI}, $\infty \in F(G_{\tau })$ and $\mbox{int}(\hat{K}(G_{\tau }))\neq \emptyset .$  
By \cite[Lemma 5.27]{Splms10}, 
%$T_{\infty ,\tau }$ is locally constant on $F(G_{\tau }).$ 
statement~\ref{randomthm1-1} holds. 
Statement~\ref{randomthm1m} follows from Lemmas~\ref{l:y1y2} and \ref{l:j1j2}.  
Statement~\ref{randomthm1tca} follows from Corollary~\ref{c:pby1pos} and Lemma~\ref{l:y1y2}-\ref{l:y1y2-2}. 
Statement~\ref{randomthm1q} follows from Lemmas~\ref{l:outernc}, \ref{l:intT1} and \ref{l:ksetnc}. 
%
%Thus we have proved our theorem.
\end{proof}

We now prove Theorem~\ref{randomthm1}.\\ 
{\bf Proof of Theorem~\ref{randomthm1}:}
Theorem~\ref{randomthm1} follows from Lemma~\ref{l:fmmmu} and Theorem~\ref{t:pbgen}. 
\qed 

\subsection{Proof of Theorem~\ref{t:hnondiffp}}
In this subsection, we prove Theorem~\ref{t:hnondiffp}. 
We need several lemmas.

\begin{lem}
\label{l:hndp1}
Under the assumptions of Theorem~\ref{t:hnondiffp}, statement~\ref{t:hnondiffp0} in  
%\ref{t:hnondiffp1}, 
%\ref{t:hnondiffp2} in 
Theorem~\ref{t:hnondiffp} holds. 
\end{lem}
\begin{proof}
Since $J(G)=\cup _{j=1}^{m}h_{j}^{-1}(J(G))$ (\cite[Lemma 2.4]{S3}), 
we obtain that $J(G)$ is disconnected. Thus $G\in {\cal G}_{dis}.$ 
By Theorem~\ref{randomthm1}, for the $\tau $, all statements in Theorem~\ref{randomthm1} hold. 
%Thus statement~\ref{t:hnondiffp0} in Theorem~\ref{t:hnondiffp} holds. 
The rest of statement~\ref{t:hnondiffp0} follows from \cite[Lemma 3.75]{Splms10} and  
%Statement~\ref{t:hnondiffp2} follows from 
\cite[Theorem 4.3, Lemma 5.1]{S3}. 
\end{proof}
\begin{lem}
\label{l:hndp2}
Under the assumptions of Theorem~\ref{t:hnondiffp}, 
we obtain that 
{\em (1)} $u(h,p,\mu )$ is well-defined and 
$u(h,p,\mu )=\frac{-\sum _{j=1}^{m}p_{j}\log p_{j}}{\sum _{j=1}^{m}p_{j}\log (\deg (h_{j}))}$, 
{\em (2)} for $\lambda $-a.e. $z_{0}\in J(G)$, 
$\emHol(T_{\infty,\tau },z_{0})\leq u(h, p, \mu )$, and {\em (3)}  
 $\pi _{\CCI }: \tilde{J}(f)\rightarrow J(G)$ is a homeomorphism. 
\end{lem}
\begin{proof}Since $\pi _{\ast }(\mu )=\tilde{\tau }$, 
$\int \log \tilde{p}\ d\mu =\sum _{j=1}^{m}p_{j}\log p_{j}.$  
By \cite[Lemma 5.52]{Splms10} and that $G\in {\cal G}$, 
we obtain  that $u(h,p,\mu )$ is well-defined and 
$u(h,p,\mu )=\frac{-\sum _{j=1}p_{j}\log p_{j}}{\sum _{j=1}^{m}p_{j}\log \deg (h_{j})}.$ 

Since $h_{i}^{-1}(J(G))\cap h_{j}^{-1}(J(G))=\emptyset $ for each $(i,j)$ with $i\neq j$, 
we may assume that $J(h_{1})<_{s}\cdots <_{s}J(h_{m}).$ 
Then, by \cite[Proposition 2.24]{SdpbpI}, $J(h_{1})\subset J_{\min }(G).$ 
Since $h_{i}^{-1}(J(G))\cap h_{j}^{-1}(J(G))=\emptyset $ for each $(i,j)$ with $i\neq j$, 
since $J(G)=\bigcup _{j=1}^{m}h_{j}^{-1}(J(G))$ (\cite[Lemma 2.4]{S3}), 
and since $J(h_{j})\subset h_{j}^{-1}(J(G))$, 
it follows that for each $j\geq 2$, 
$J(h_{j})\cap J_{\min }(G)=\emptyset .$ Hence, by \cite[Theorem 2.20-2,5]{SdpbpI}, 
$h_{j}^{-1}(J(G))\cap P(G)=\emptyset $ for each $j\geq 2.$ 
Let $A:=\{ (\g , y)\in J(G)\mid \exists n\in \NN \mbox{ s.t. } \sigma ^{n}(\g )=(1,1,1,\ldots )\} .$ 
Since $\pi _{\ast }(\mu )=\tilde{\tau }$, and since $\tilde{\tau }(\{ (1,1,1,\ldots )\} )=0$, 
it follows that $\mu (A)=0.$

Since $\pi _{\CCI }:\tilde{J}(f)\rightarrow J(G)$ is surjective (\cite[Lemma 4.5]{Splms10}), 
and since $h_{i}^{-1}(J(G))\cap h_{j}^{-1}(J(G))=\emptyset $ for each $(i,j)$ with $i\neq j$, 
we obtain that 
$\pi _{\CCI }:\tilde{J}(f)\rightarrow J(G)$ is a homeomorphism. 
Thus $\lambda (\pi _{\CCI }(A))=0.$ 
Let $\{ t_{n}\} _{n=1}^{\infty }$ be a decreasing sequence of real numbers such that 
$t_{n}>u(h,p,\mu )$ for each $n\in \NN $ and such that $t_{n}\rightarrow u(h,p,\mu )$ as $n\rightarrow \infty .$   
By Birkhoff's ergodic theorem and \cite[Lemma 5.52]{Splms10}, for each $n\in \NN $ there exists a Borel subset $B_{n}$ of 
$\tilde{J}(f)$ with $\mu (B_{n})=1$ such that 
for each $(\gamma , y)\in B_{n}$, 
$\frac{1}{r}\log (\tilde{p}(f^{r}(\gamma ,y))\cdots \tilde{p}(\gamma ,y)
\| Df^{r}_{(\gamma ,y)}\| _{s}^{t_{n}})\rightarrow 
\int _{\G^{\NN }\times \CCI }\log \tilde{p}(\g, y) d\mu (\g, y) 
+\int _{\G^{\NN }\times \CCI }\log \| Df_{(\g,y)}\| _{s}^{t_{n}}d\mu (\g,y) >0$ as $r\rightarrow \infty .$  
Thus for each $(\g ,y)\in B_{n}$, 
\begin{equation}
\label{eq:hndp2pf1}
\tilde{p}(f^{r}(\g ,y))\cdots \tilde{p}(\g,y)\| D(\g _{r,1})_{y}\| _{s}^{t_{n}}\rightarrow \infty 
\mbox{ as }r\rightarrow \infty . 
\end{equation}
Let $C:= (J(G)\setminus \pi _{\CCI }(A))\cap \bigcap _{n=1}^{\infty }\pi _{\CCI }(B_{n}).$ 
Then $\lambda (C)=1.$ 
%For each $n\in \NN $, 
Let $z_{0}\in C.$ Let $\gamma \in \G ^{\NN }$ be the unique element 
$(\g ,z_{0})\in \tilde{J}(f).$ 
Since $z_{0}\in J(G)\setminus \pi _{\CCI }(A)$, 
 there exists a $j\in \{ 2,\ldots ,m\} $ and a 
strictly increasing sequence $\{ n_{k}\} _{k=1}^{\infty }$ of positive integers 
such that $\gamma _{n_{k}+1}=j$ for each $k\in \NN .$ 
Then for each $k\in \NN $, $\gamma _{n_{k},1}(z_{0})\in \gamma _{n_{k}+1}^{-1}(J(G))=h_{j}^{-1}(J(G)).$  
We may assume that there exists a point $z_{1}\in h_{j}^{-1}(J(G)) \subset \CCI \setminus P(G)$ such that 
$\gamma _{n_{k},1}(z_{0})\rightarrow z_{1}$ as $k\rightarrow \infty .$  
By (\ref{eq:hndp2pf1}) and \cite[Lemma 5.48-1]{Splms10}, 
we obtain that for each $n\in \NN $, 
$\limsup _{z\rightarrow z_{0}, z\neq z_{0}}\frac{|T_{\infty ,\tau }(z)-T_{\infty ,\tau }(z_{0})|}{d(z,z_{0})^{t_{n}}}=\infty .$ 
Therefore $\Hol(T_{\infty ,\tau },z_{0})\leq u(h,p,\mu ).$ 
Thus we have proved our lemma. 
\end{proof} 
\begin{df}[\cite{SdpbpI}]
\label{d:Psi}
For a polynomial $g$, we denote by 
$a(g)\in \CC $ the coefficient of 
the highest degree term of $g.$ 
%\item 
We set 
RA $:=\{ ax+b\in \RR [x]\mid a,b\in \RR ,a\neq 0\} $.  
 The space RA is a semigroup 
with the semigroup operation being functional composition.
Any subsemigroup of RA will be called a {\em real affine semigroup}.
We define a map $\Psi :$  ${\cal P} \rightarrow $ RA as follows: 
For a polynomial $g\in {\cal P}$,     
%of the form 
%$g(z)=a(g)z^{\deg (g)}+$ lower degree terms, 
%where $a(g)\in \CC $, 
we set $\Psi (g)(x):= \deg (g)x+\log | a(g)|.$ 
We remark that $\Psi (g\circ h)=\Psi (g)\circ \Psi (h).$ 
For a polynomial semigroup $G$, 
 we set $\Psi (G):= \{ \Psi (g)\mid g\in G\}  $ ($\subset $ RA). 
Thus $\Psi (G)$ is a real affine semigroup. 
We set $\hat{\RR }:= \RR \cup \{ \pm \infty \} $ endowed with 
the topology such that 
$\{ (r,+\infty ]\} _{r\in \RR }$ makes a fundamental neighborhood system of 
$+\infty $, and such that $\{ [-\infty ,r)\} _{r\in \RR }$ makes a 
fundamental neighborhood system of $-\infty .$ 
For a real affine semigroup $H$, we set \\ 
$M(H):= \overline{ \{ x\in \RR \mid \exists h\in H,  h(x)=x, |h'(x)|>1\} } 
\ (\subset \hat{\RR })
,$ 
where the closure is taken in the space $\hat{\RR }.$   
%Moreover, we denote by ${\cal M}_{H}$ the set of all connected components of 
%$M(H).$ 
We denote by $\eta : $ RA $\rightarrow {\cal P}$ the natural embedding 
defined by $\eta (x\mapsto ax+b)=(z\mapsto az+b)$, where 
$x\in \RR $ and $z\in \CC .$ 
\end{df}

\begin{lem}
\label{l:djsum}
Under the assumptions of Theorem~\ref{t:hnondiffp}, 
we get that {\em (1)} $M(\Psi (G))$ is a Cantor set in $\RR $, 
{\em (2)} $M(\Psi (G))=\bigcup _{j=1}^{m}(\Psi (h_{j}))^{-1}(M(\Psi (G)))$, 
{\em (3)} $(\Psi (h_{i}))^{-1}(M(\Psi (G)))\cap (\Psi (h_{j}))^{-1}(M(\Psi (G)))$ 
$=\emptyset $ for each $(i,j)$ with $i\neq j$, and {\em (4)} 
$\sum _{j=1}^{m}\frac{1}{\deg (h_{j})}<1.$ 
\end{lem}
\begin{proof}
We use the arguments in the proof of \cite[Lemma 4.9]{SdpbpI}. 
For each $\gamma \in \G ^{\NN }$, let 
$J(G)_{\g }:= \bigcap _{j=1}^{\infty }\g _{j,1}^{-1}(J(G)).$ 
Since $J(G)=\bigcup _{j=1}^{m}h_{j}^{-1}(J(G))$ (\cite[Lemma 2.4]{S3}) and 
since $h_{i}^{-1}(J(G))\cap h_{j}^{-1}(J(G))$ $=\emptyset $ for each 
$(i,j)$ with $i\neq j$, 
we obtain that 
$J(G)=\amalg _{\g \in \G ^{\NN }}J(G)_{\g }$ (disjoint union). 
By \cite[Corollary 4.19]{S15}, for each $\g \in \G ^{\NN }$, 
$J(G)_{\g }$ is connected. Thus each $J(G)_{\g }$ is a connected component of $J(G).$ 
By \cite[Proposition 2.2(3)]{S7}, \cite[Lemma 4.1]{SdpbpI} and  
that $J_{\g }\subset J(G)_{\g }$ 
for each $\g \in \GN $, 
%$\sup _{z\in J(\g _{n,1})}d(z,J_{\g })\rightarrow 0$ as $n\rightarrow \infty .$ 
%For each $\g \in \GN $, since , 
it follows that for each $\g \in \GN $, 
%\begin{equation}
%\label{eq:djsumpf1}
$\sup _{z\in J(\g _{n,1})}d(z,J(G)_{\g })\rightarrow 0 \mbox{ as } n\rightarrow \infty .$  
%\end{equation} 
 By \cite[Lemma 4.5]{SdpbpI}, $M(\Psi (G))=J(\eta (\Psi (G)))\subset \RR .$ 
Since $J(\eta (\Psi (G)))=\bigcup _{j=1}^{m}(\eta (\Psi (h_{j})))^{-1}(J(\eta (\Psi (G))))$, 
by \cite[Theorem 2.6]{Fa} 
it follows that $M(\Psi (G))$ is the self-similar set constructed by 
contracting similitudes $(\Psi (h_{1}))^{-1},\ldots ,(\Psi (h_{m}))^{-1}$ on $\RR .$ 
Let $b_{\min }:= \min \{ \frac{-1}{\deg (h_{j})-1}\log |a(h_{j})|\mid j=1,\ldots ,m\} $ and 
$b_{\max }:= \max \{ \frac{-1}{\deg (h_{j})-1}\log |a(h_{j})|\mid j=1,\ldots ,m\} .$ 
Note that $\frac{-1}{\deg (g)-1}\log |a(g)|$ is the unique fixed point of $\Psi (g)$ in $\RR .$ 
Let $I=[b_{\min },b_{\max}]$ be the closed interval between $b_{\min }$ and $b_{\max}.$ 
Then we have that $\bigcup _{j=1}^{m}(\Psi (h_{j}))^{-1}(I)\subset I.$ 
It follows that 
$M(\Psi (G))=\bigcup _{\g \in \GN } \bigcap _{n=1}^{\infty }(\Psi (\g _{n,1}))^{-1}(I).$ 
Let $\rho :\GN \rightarrow M(\Psi (G))$ be the map defined by 
$\rho (\g ):= \bigcap _{n=1}^{\infty }(\Psi (\g _{n,1}))^{-1}(I)$ for each $\g .$ 
Then $\rho :\GN \rightarrow M(\Psi (G))$ is continuous. 
For each $\g \in \GN $ and each $n\in \NN $, 
$\frac{-1}{\deg (\g _{n,1})}\log |a(\g _{n,1})|$ is the fixed point of 
$\Psi (\g _{n,1})$ in $I$. Therefore 
$\frac{-1}{\deg (\g _{n,1})-1}\log |a(\g _{n,1})|=\rho (\omega ^{\gamma ,n})$, 
where $\omega ^{\g, n}\in \GN $ is the $n$-periodic point of $\sigma :\GN \rightarrow \GN $ 
with $((\omega ^{\g ,n})_{1},\ldots ,(\omega ^{\g ,n})_{n})=(\g _{1},\ldots ,\g _{n}).$ 
Since $\omega ^{\g ,n}\rightarrow \g $ in $\GN $ as $n\rightarrow \infty $, 
it follows that for each $\g \in \GN $, 
$\lim _{n\rightarrow \infty } \frac{-1}{\deg (\g _{n,1})-1}\log |a(\g _{n,1})|=\rho (\gamma ).$  
For each $\g \in \GN $, let $B_{\g }\in \mbox{Con}(M(\Psi (G)))$ with 
 $\lim _{n\rightarrow \infty } \frac{-1}{\deg (\g _{n,1})-1}\log |a(\g _{n,1})|\in B_{\g }.$ 
Let $\tilde{\Psi }:\mbox{Con}(J(G))\rightarrow \mbox{Con}(M(\Psi (G)))$ be the map 
defined by $\tilde{\Psi }(J(G)_{\g }):=B_{\g }$ for each $\g \in \GN .$ 
By \cite[Claim 2 in the proof of Lemma 4.9]{SdpbpI}, 
$\tilde{\Psi }: \mbox{Con}(J(G))\rightarrow \mbox{Con}(M(\Psi (G)))$ is injective. 
Therefore, it follows that $\rho :\GN \rightarrow M(\Psi (G))$ is 
injective. Thus, $\rho :\GN \rightarrow M(\Psi (G))$ is a homeomorphism. 
In particular, $M(\Psi (G))$ is a Cantor set in $I.$ 
Let $0<\epsilon <\min \{ |a-b|\mid a\in (\Psi (h_{i}))^{-1}(M(\Psi (G))), b\in (\Psi (h_{j}))^{-1}(M(\Psi (G))), i\neq j\} $ and 
let $U$ be the  $\epsilon $-neighborhood of $M(\Psi (G))$ in $\RR .$ 
(Thus $U$ is a finite union of bounded open intervals.) 
Since $\rho $ is a homeomorphism, 
$(\Psi (h_{i}))^{-1}(M(\Psi (G)))\cap (\Psi (h_{j}))^{-1}(M(\Psi (G))=\emptyset 
$ for each $(i,j)$ with $i\neq j$. Hence 
%$\{ (\Psi (h_{j}))^{-1}\} _{j=1}^{m}$ satisfies the 
%open set condition with $U$, i.e., 
%there exists an open set $U$ in $\RR $ such that 
$\bigcup _{j=1}^{m}(\Psi (h_{j}))^{-1}(\overline{U})\subset U$ and 
$(\Psi (h_{i}))^{-1}(\overline{U})\cap (\Psi (h_{j}))^{-1}(\overline{U})=\emptyset $ for each 
$(i,j)$ with $i\neq j.$ 
%In fact, $U$ can be taken as an $\epsilon $-neighborhood of $M(\Psi (G))$ in $\RR $ with small $\epsilon >0.$ 
Thus denoting by $l$ the one-dimensional Lebesgue measure, 
$\sum _{j=1}^{m}\frac{1}{\deg (h_{j})}l(U)=\sum _{j=1}^{m}l((\Psi (h_{j}))^{-1}(U))<l(U).$ 
Hence $\sum _{j=1}^{m}\frac{1}{\deg (h_{j})}<1.$  
%By the well-known fact of fractal geometry (see \cite{Fa,MU}), 
%the Hausdorff dimension $\dim _{H}(M(\Psi (G)))$ of $M(\Psi (G))$ with respect to the Euclidian distance on $\RR $ 
%is equal to the unique solution of 
%the equation $\sum _{j=1}^{m}(\frac{1}{\deg (h_{j})})^{s}=1, s\geq 0$,  
%and setting $t=\dim _{H}(M(\Psi (G)))$, 
%the $t$-dimensional Hausdorff measure $H^{t}(M(\Psi (G)))$ of $M(\Psi (G))$ 
%is positive and finite. Suppose that $t=1.$ (We will deduce a contradiction later.) 
%By the method of proof of \cite[Proposition 4.5.9]{MU}, 
%we can show that if $U\setminus (\bigcup _{j=1}^{m}(\Psi (h_{j}))^{-1}(\overline{U})\neq \emptyset $, 
%then 
%the $1$-dimensional Lebesgue measure of $M(\Psi (G))$ is zero. 
%Since we are assuming $t=1$, it follows that 
%$\overline{U}=\bigcup _{j=1}^{m}(\Psi (h_{j}))^{-1}(\overline{U}).$ By the uniqueness of 
%self-similar set, we obtain that $M(\Psi (G))=\overline{U}.$ 
%However, it contradicts that $M(\Psi (G))$ is a Cantor set. Hence  
%we must have that $t<1.$ Since $t$ is the unique solution of 
%$\sum _{j=1}^{m}(\frac{1}{\deg (h_{j})})^{s}=1, s\geq 0$, 
%and since the function $s\mapsto \sum _{j=1}^{m}(\frac{1}{\deg (h_{j})})^{s}$ is 
%strictly decreasing as $s$ is increasing, 
%we get  
%$\sum _{j=1}^{m}\frac{1}{\deg (h_{j})}<1.$ 
Thus we have proved our lemma.
\end{proof} 
We now prove Theorem~\ref{t:hnondiffp}.\\ 
{\bf Proof of Theorem~\ref{t:hnondiffp}:} 
Statement~\ref{t:hnondiffp0}  
%\ref{t:hnondiffp1},\ref{t:hnondiffp2} 
follows from 
Lemma~\ref{l:hndp1}. 
By Lemma~\ref{l:hndp2}, we have 
$u(h,p,\mu )=\frac{-\sum _{j=1}^{m}p_{j}\log p_{j}}{\sum _{j=1}^{m}p_{j}\log (\deg (h_{j}))}$. 
It is easy to see that 
$\min \{ \sum _{j=1}^{m}p_{j}\log \deg (h_{j})+\sum _{j=1}^{m}p_{j}\log p_{j}\mid (p_{1},\ldots ,p_{m})\in {\cal W}_{m}\} 
=-\log (\sum _{j=1}^{m}\frac{1}{\deg (h_{j})}).$ 
Combining these arguments with Lemmas~\ref{l:djsum} and \ref{l:hndp2}, 
we see that 
statement~\ref{t:hnondiffp4} follows.    
 
By \cite[Theorem 1.3 (f)]{S3}, 
$h_{\mu }(f|\sigma )=\sum _{j=1}^{m}p_{j}\log \deg (h_{j}).$ 
Hence, $h_{\mu }(f)=h_{\mu }(f|\sigma )+h_{\pi _{\ast }(\mu )}(\sigma )=
\sum _{j=1}^{m}p_{j}\log \deg (h_{j})-\sum _{j=1}^{m}p_{j}\log p_{j}$, 
where $h_{\mu }(f)$ denotes the metric entropy of $(f, \mu ).$  
Combining this with \cite[Lemma 7.1]{S3}, \cite[Lemma 5.52]{Splms10},  
that $\pi _{\CCI }:\tilde{J}(f)\rightarrow J(G)$ is a homeomorphism (Lemma~\ref{l:hndp2}),  
and that $G\in {\cal G}$, 
%we obtain that 
we see that 
%\begin{equation}
%\label{eq:dimhmu}
$\dim _{H}(\lambda )=\frac{\sum _{j=1}^{m}p_{j}\log \deg (h_{j})-\sum _{j=1}^{m}p_{j}\log p_{j}}
{\sum _{j=1}^{m}p_{j}\log \deg (h_{j})}>1, $ 
%\end{equation}
where $\dim _{H}(\lambda ):= \inf \{ \dim _{H}(A)\mid A\mbox{ is a Borel subset of } J(G), \lambda (A)=1\} .$ 
Hence, we have proved statement \ref{t:hnondiffp4-1}. 
Statement~\ref{t:hnondiffp5} follows from statements~\ref{t:hnondiffp0} and \ref{t:hnondiffp4}. 
Thus we have proved Theorem~\ref{t:hnondiffp}.
\qed 
\subsection{Proof of Theorem~\ref{t:2gengdis}} 
In this subsection, we prove Theorem~\ref{t:2gengdis}. 
We need several lemmas and propositions. 
\begin{df}
Let $\Gamma $ be a non-empty compact subset of ${\cal P}$ and suppose 
$\langle \G \rangle \in {\cal G}_{dis}.$ 
We set $\G _{\min }:= \{ h\in \G \mid J(h)\subset J_{\min }(\langle \G \rangle )\} .$ 
Note that by \cite[Proposition 2.24]{SdpbpI}, 
$\G _{\min }\neq \emptyset .$ 
\end{df}
\begin{lem}
\label{l:mgdiscg}
Let $m\in \NN $ with $m\geq 2.$ Let 
$\G =\{ h_{1},\ldots ,h_{m}\} \subset {\cal P}.$ 
Let $G= \langle \G \rangle $ and suppose 
that $G\in {\cal G}_{dis}.$ 
Suppose that $\sharp \G _{\min }=1.$ 
Then, we have the following {\em (1)} and {\em (2)}. 
%\begin{enumerate}
%\item 
{\em (1)} For each $\g \in \GN $, 
$J_{\g }=\hat{J}_{\g ,\G }=\bigcap _{j=1}^{\infty }\g _{1}^{-1}\cdots \g _{j}^{-1}(J(G)).$ 
%\item 
{\em (2)} The map $\g \mapsto J_{\g }$ is continuous on $\GN $ with respect to the Hausdorff metric 
in the space of non-empty compact subsets of $\CCI .$ 
%\end{enumerate}
\end{lem}
\begin{proof}
We may assume that 
$\G _{\min }=\{ h_{1}\} .$ 
By \cite[Theorem 2.20-5]{SdpbpI}, 
$\emptyset \neq \mbox{int}(\hat{K}(G))\subset \mbox{int}(K(h_{1})).$ 
By \cite[Theorem 2.20-5]{SdpbpI} again, 
for each $j\geq 2$, 
$h_{j}(J(h_{1}))\subset h_{j}(J_{\min }(G))\subset \mbox{int}(\hat{K}(G))\subset \mbox{int}(K(h_{1})).$ 
Therefore for each $j\geq 2$, 
$h_{j}(\mbox{int}(K(h_{1})))\subset \mbox{int}(K(h_{1})).$ 
Thus $\mbox{int}(K(h_{1}))\subset F(G).$ 
 Let $\g \in \GN .$ 
Suppose that there exists a point $y_{0}\in \hat{J}_{\g ,\G }\setminus J_{\g }.$ 
We now consider the following two cases. 
Case 1: $\sharp \{ n\in \NN \mid \g _{n}\neq h_{1}\} =\infty .$ 
Case 2: $\sharp \{ n\in \NN \mid \g _{n}\neq h_{1}\} <\infty .$ 

 Suppose that we have Case 1. 
Then there exist an open neighborhood $U$ of $y_{0}$ in $\CCI $, 
a strictly increasing sequence $\{ n_{j}\} _{j=1}^{\infty }$ of positive integers, 
a number $i\in \{ 2,\ldots ,m\} $, 
and a map $\varphi : U\rightarrow \CCI $, such that  
$\g _{n_{j}+1}=h_{i}$ for each $j\in \NN $, and such that 
$\g _{n_{j},1}\rightarrow \varphi $ uniformly on $U$ as $j\rightarrow \infty .$ 
Since $\gamma _{n_{j},1}(y_{0})\in J(G)$ for each $j$, 
\cite[Lemma 5.6]{SdpbpII} and \cite[Proposition 2.19]{SdpbpI} imply that 
$\varphi $ is constant. 
By \cite[Lemma 3.13]{SdpbpIII}, it follows that 
$d(\g _{n_{j},1}(y_{0}), P^{\ast }(G))\rightarrow 0$ as $j\rightarrow \infty .$ 
Moreover, since $\g _{n_{j}+1}=h_{i}$, 
we obtain $\g _{n_{j},1}(y_{0})\in h_{i}^{-1}(J(G))$ for each $j.$ 
Furthermore, by \cite[Theorem 2.20-2,5]{SdpbpI}, $h_{i}^{-1}(J(G))\subset \CCI \setminus P^{\ast }(G).$ 
This is a contradiction. Hence, we cannot have Case 1. 

 Suppose we have Case 2. 
Let $r\in \NN $ be a number such that 
for each $s\in \NN $ with $s\geq r$, $\g _{s}=h_{1}.$ 
Then $h_{1}^{n}(\g _{r,1}(y_{0}))\in J(G)$ for each $n\geq 0.$ 
Since $y_{0}\not\in J_{\g }$, 
we have $\g _{r,1}(y_{0})\not\in J(h_{1}).$ 
Moreover, since $\g_{r,1}(y_{0})\in J(G)$ and int$(\hat{K}(h_{1}))\subset F(G)$, 
it follows that $\g _{r,1}(y_{0})$ belongs to $F_{\infty }(h_{1})$. 
It implies that  $h_{1}^{n}(\g _{r,1}(y_{0}))\rightarrow \infty $ as $n\rightarrow \infty $. 
However, this contradicts that $h_{1}^{n}(\g _{r,1}(y_{0}))\in J(G)$ for each $n\geq 0.$ 
Therefore, we cannot have Case 2. 

Thus, for each $\g \in \G $, $J_{\g }=\hat{J}_{\g ,\G }.$
Moreover, by \cite[Lemma 3.5]{SdpbpIII}, 
$\hat{J}_{\g ,\G }=\bigcap _{j=1}^{\infty }\g _{1}^{-1}\cdots \g _{j}^{-1}(J(G))$ for each 
$\g \in \GN .$  
Combining the result ``$J_{\g }=\hat{J}_{\g ,\G }$ for each $\g \in \GN $''  with \cite[Proposition 2.2(3)]{S7}, 
we obtain that the map $\g \mapsto J_{\g }$ is continuous. 
%
%Hence, we have proved our lemma.  
\end{proof}
\begin{prop}
\label{p:mgengdis}
Let $m\geq 2$ and 
let $G=\langle h_{1},\ldots ,h_{m}\rangle \in {\cal G}.$ 
Let $(p_{1},\ldots ,p_{m})\in {\cal W}_{m}$ and 
let $\tau =\sum _{j=1}^{m}p_{j}\delta _{h_{j}}.$  
Let $\G =\{ h_{1},\ldots ,h_{m}\} .$ 
Suppose that $h_{i}^{-1}(J(G))\cap h_{j}^{-1}(J(G))=\emptyset $ for 
each $(i,j)$ with $i\neq j.$ 
Then we have the following.
\begin{enumerate}
\item $G\in {\cal G}_{dis} $ and $\sharp \G _{\min }=1$. For each $\g \in \GN $, 
$J_{\g }=\hat{J}_{\g ,\G }=\bigcap _{j=1}^{\infty }\g _{1}^{-1}\cdots \g _{j}^{-1}(J(G)).$ 
The map $\g \mapsto J_{\g }$ is continuous on $\GN $ with respect to the Hausdorff metric 
in the space of all non-empty compact subsets of $\CCI .$ 
\item \label{p:mgengdis2}
For each $J\in \mbox{{\em Con}}(J(G))$, there exists a unique $\g \in \GN $ with 
$J=J_{\g }.$ {\em Con}$(J(G))=\{ J_{\g }\mid \g \in \GN \} .$ 
The map $\g \mapsto J_{\g }$ is a bijection between $\GN $ and $\mbox{{\em Con}}(J(G)).$ 
In particular, 
there exist uncountably many connected components of $J(G).$ 
\item There exist infinitely many doubly connected components of $F(G).$ 
\item For each $J\in \mbox{{\em Con}}(J(G))$,  
$T_{\infty ,\tau }|_{J}$ is constant. 
\item Let $J_{1},J_{2}\in \mbox{{\em Con}}(J(G))$ with $J_{1}\neq J_{2}.$ 
Suppose $T_{\infty ,\tau }|_{J_{1}}=T_{\infty ,\tau }|_{J_{2}}.$ 
Then there exists a doubly connected component $A$ of $F(G)$ such that 
$\partial A\subset J_{1}\cup J_{2}.$ 
\end{enumerate}
\end{prop}
\begin{proof}
Since $J(G)=\bigcup _{j=1}^{m}h_{j}^{-1}(J(G))$ (\cite[Lemma 2.4]{S3}), 
$G\in {\cal G}_{dis}.$ 
By \cite[Proposition 2.24]{SdpbpI}, 
$\G _{\min }\neq \emptyset .$ Without loss of generality, 
we may assume that 
$h_{1}\in \G _{\min }.$ 
Since $J(G)=\bigcup _{j=1}^{m}h_{j}^{-1}(J(G))$ again, 
for each $j\geq 2$, there exists no $J\in \mbox{Con}(J(G))$ with $J(h_{1})\cup J(h_{j})\subset J.$ 
Therefore, $\G _{\min }=\{ h_{1}\} .$ 
By Lemma~\ref{l:mgdiscg}, 
it follows that  
$J_{\g }=\hat{J}_{\g ,\G }=\bigcap _{j=1}^{\infty }\g _{1}^{-1}\cdots \g _{n}^{-1}(J(G))$ for each 
$\g\in \GN $, and that 
the map $\g \mapsto J_{\g }$ is continuous.  
Since $J(G)=\bigcup _{j=1}^{m}h_{j}^{-1}(J(G))$ and since 
$h_{i}^{-1}(J(G))\cap h_{j}^{-1}(J(G))=\emptyset $ for each $(i,j)$ with $i\neq j$, 
we obtain that 
$J(G)=\amalg  _{\g \in \G ^{\NN }} \bigcap _{n=1}^{\infty }\g _{1}^{-1}\cdots \g _{n}^{-1}(J(G)).$ 
Moreover, by \cite[Lemma 3.6]{SdpbpIII}, 
$J_{\g }$ is connected for each $\g \in \GN .$ Therefore 
$J_{\g }$ is a connected component of $J(G)$ for each $\g \in \GN .$ Moreover, 
the map $\g \in \GN \mapsto J_{\g }\in \mbox{Con}(J(G))$ is a bijection. 
In particular, there exist uncountably many connected components of $J(G).$ 
Combining this with \cite[Theorem 2.7-1, Lemma 4.4]{SdpbpI}, 
we obtain that there are infinitely many doubly connected components of $F(G).$ 

Let $J\in \mbox{Con}(J(G)).$ Then there exists a unique element 
$\alpha \in \GN $ such that 
$J=J_{\alpha }.$ 
Let $z_{0}\in J$ be a point. Let $\g \in \GN $ be an element. 
Suppose $\gamma _{n,1}(z_{0})\rightarrow \infty .$ 
Then $\gamma \neq \alpha .$ 
By the uniqueness of $\alpha $, we obtain 
$J_{\g }\neq 
J_{\alpha }.$ 
By \cite[Theorem 2.7]{SdpbpI} and that $\gamma _{n,1}(z_{0})\rightarrow \infty $, 
it follows that  $J_{\g }<_{s} 
J=J_{\alpha }.$ 
Therefore, for each $z\in J$, 
$\gamma _{n,1}(z)\rightarrow \infty .$ 
Thus, $T_{\infty ,\tau }|_{J}$ is constant.  

 We now let  $J_{1},J_{2}\in \mbox{Con}(J(G))$ with $J_{1}\neq J_{2}$ and  
suppose $T_{\infty ,\tau }|_{J_{1}}=T_{\infty ,\tau }|_{J_{2}}.$ 
Without loss of generality, we may assume $J_{1}<_{s}J_{2}.$ 
By \cite[Lemma 4.4]{SdpbpI}, 
there exists a doubly connected component $A$ of $F(G)$ such that 
$J_{1}<_{s}A<_{s}J_{2}.$ Let $B_{1} $ and $B_{2}$ be two connected components of $\partial A$ with 
$B_{1}<_{s}B_{2}.$ For each $i=1,2$,  
let $J_{i}'\in \mbox{Con}(J(G))$ with $B_{i}\subset J_{i}'.$ 
Then $J_{1}\leq _{s}J_{1}'<_{s}A<_{s}J_{2}'\leq _{s}J_{2}.$ 
Suppose $J_{1}<_{s}J_{1}'.$ Then 
by \cite[Lemma 4.4]{SdpbpI}, 
there exists a doubly connected component  $D_{1}$ of $F(G)$ such that 
$J_{1}<_{s}D_{1}<_{s}J_{1}'.$ Therefore 
$J_{1}<_{s}D_{1}<_{s}A<_{s}J_{2}.$ 
By  Lemma~\ref{l:fmcmono}, Theorem~\ref{randomthm1}-\ref{randomthm1m} and Lemma~\ref{l:y1y2}-\ref{l:y1y2-1}, 
it follows that 
$T_{\infty ,\tau }|_{J_{1}}\leq T_{\infty ,\tau }|_{D_{1}}<T_{\infty ,\tau }|_{A}\leq T_{\infty ,\tau }|_{J_{2}}.$ 
However, this contradicts that $T_{\infty ,\tau }|_{J_{1}}=T_{\infty ,\tau }|_{J_{2}}.$ Therefore, 
$J_{1}=J_{1}'.$ 
Similarly, we obtain $J_{2}=J_{2}'.$ 
Therefore, $\partial A\subset J_{1}\cup J_{2}.$  

Thus we have proved our proposition. 
\end{proof}

We now prove Theorem~\ref{t:2gengdis}.

{\bf Proof of \ref{t:2gengdis}:}
Let $\G :=\{ h_{1},h_{2}\} .$ 
By \cite[Theorems 3.17, 3.2]{S15}, 
$h_{1}^{-1}(J(G))\cap h_{2}^{-1}(J(G))=\emptyset .$ 
Thus all statements \ref{t:2gengdis1}--\ref{t:2gengdis5} in Theorem~\ref{t:2gengdis} follow from Proposition~\ref{p:mgengdis} and Theorem~\ref{t:hnondiffp}.   

 We now prove statement~\ref{t:2gengdis6}. 
 By statement~\ref{t:2gengdis2} and \cite[Theorem 2.7]{SdpbpI}, 
 either $J(h_{1})<_{s}J(h_{2})$ or $J(h_{2})<_{s}J(h_{1}).$ 
 We now assume $J(h_{1})<_{s}J(h_{2}).$ Then, by \cite[Proposition 2.24]{SdpbpI}, 
 $J(h_{1})\subset J_{\min }(G)$ and $J(h_{2})\subset J_{\max }(G).$ 
 By statement~\ref{t:2gengdis2}, it follows that 
 $J(h_{1})=J_{\min }(G)$ and $J(h_{2})=J_{\max}(G).$ Let $A=K(h_{2})\setminus \mbox{int}(K(h_{1})).$ 
 We now prove the following claim. \\ 
Claim 1. $h_{1}^{-1}(A)\cup h_{2}^{-1}(A)\subset A.$ 

 To prove this claim, let $\alpha =(h_{2},h_{1},h_{1},\ldots )\in \GN .$ 
Then $J_{\alpha }=h_{2}^{-1}(J(h_{1})).$ 
Since $J(h_{1})=J_{\min }(G)$, statement~\ref{t:2gengdis2} implies that 
$J(h_{1})<_{s}J_{\alpha }=h_{2}^{-1}(J(h_{1})).$ 
Therefore $h_{2}^{-1}(A)\subset A.$ Similarly, 
letting $\beta =(h_{1},h_{2},h_{2},\ldots )\in \GN $, 
we have $J_{\beta }=h_{1}^{-1}(J(h_{2}))<_{s}J(h_{2})$ and 
$h_{1}^{-1}(A)\subset A.$ Thus we have proved Claim 1. 

 We have that $h_{1}^{-1}(A) $ and $h_{2}^{-1}(A)$ are connected compact sets. 
We prove the following claim.

Claim 2. $J_{\beta }=h_{1}^{-1}(J(h_{2}))<_{s}J_{\alpha } =h_{2}^{-1}(J(h_{1})).$ In particular, 
$h_{1}^{-1}(A)<_{s}h_{2}^{-1}(A).$ 

 To prove this claim, suppose that $J_{\beta }<_{s}J_{\alpha }$ does not hold. 
 Then by \cite[Theorem 2.7]{SdpbpI}, 
 $J_{\alpha }<_{s}J_{\beta }.$ 
 This implies that $A=h_{1}^{-1}(A)\cup h_{2}^{-1}(A).$ 
 By \cite[Corollary 3.2]{HM}, we have $J(G)\subset A.$ 
 Since $J(G)$ is disconnected (assumption) and since $A$ is connected, 
 $F(G)\cap A\neq \emptyset .$ 
 Let $y\in F(G)\cap A.$ Since $A=h_{1}^{-1}(A)\cup h_{2}^{-1}(A)$, 
there exists an element $\g \in \GN $ such that 
for each $n\in \NN $, $\g _{n,1}(y)\in A.$ 
Since $y\in A\cap F(G)$ and $G(F(G))\subset F(G)$, 
  $\g _{n,1}(y)\in F_{\infty }(h_{1})\cap A$ for each $n\in \NN .$ 
Therefore there exists a strictly increasing sequence $\{ n_{j}\} _{j=1}^{\infty }$ in $\NN $ 
such that 
for each $j$, $\g _{n_{j}+1}=h_{2}.$ 
Since $y\in F_{\g }$, we may assume that there exist an open neighborhood $U$ of $y$ in $\CCI $ 
and a holomorphic map $\varphi :U\rightarrow \CCI $ such that 
$\g _{n_{j},1}\rightarrow \varphi $ uniformly on $U$ as $j\rightarrow \infty .$ 
Since $\g _{n_{j},1}(y)\in F_{\infty }(h_{1})\cap A\subset (\CCI \setminus \hat{K}(G))\cap A$ for each $j$, 
%Lemma~\ref{ncintk} 
\cite[Lemma 5.6]{SdpbpII}
implies that 
there exists a constant $c\in \CC $ such that $\varphi =c$ on $U.$ 
By \cite[Lemma 3.13]{SdpbpIII}, 
it follows that $c\in P^{\ast }(G).$ Since $P^{\ast }(G)\subset K(h_{1})$ and 
since $\g _{n_{j},1}(y)\in F_{\infty }(h_{1})$ for each $j$, it follows that 
$d(\g _{n_{j},1}(y),J(h_{1}))\rightarrow 0$ as $j\rightarrow \infty .$ 
Combining this with that $\g _{n_{j}+1}=h_{2}$ for each $j$, we obtain that 
$d(\g _{n_{j},1}(y), h_{2}^{-1}(J(h_{1})))\rightarrow \infty .$ 
Since $J(h_{1})<_{s}h_{2}^{-1}(J(h_{1}))$, 
it follows that $c\in F_{\infty }(h_{1}).$ However, this is a contradiction, 
since $c\in P^{\ast }(G)\subset K(h_{1}).$ 
Therefore, $J_{\beta }<_{s}J_{\alpha }.$ Thus we have proved Claim 2. 

  Let $\theta =(h_{2},\theta _{2},\theta _{3},\ldots )\in \GN  $ and 
$\xi =(h_{1},\xi _{2},\xi _{3},\ldots )\in \GN .$ 
Then $J_{\theta }\subset h_{2}^{-1}(J(G))\subset h_{2}^{-1}(A)$ and 
$J_{\xi }\subset h_{1}^{-1}(J(G))\subset h_{1}^{-1}(A).$   
  By claim 2, statement~\ref{t:2gengdis2} and \cite[Theorem 2.7]{SdpbpI}, 
we obtain that 
$J_{\xi }<_{s}J_{\theta }.$ 
Combining this result with statement~\ref{t:2gengdis2} and \cite[Theorem 2.7-3]{SdpbpI}, 
%it follows that 
we see that 
the map $\zeta :\{ 1,2\} ^{\NN }\rightarrow \mbox{Con}(J(G))$ satisfies that 
if $w^{1},w^{2}\in \{ 1,2\} ^{\NN }$ with $w^{1}<_{l}w^{2}$, then 
$\zeta (w^{1})<_{s}\zeta (w^{2}).$ Moreover, by statement \ref{t:2gengdis2}, 
this map $\zeta : \{ 1,2\} ^{\NN }\rightarrow \mbox{Con}(J(G))$  is a bijection. 
 Thus we have proved statement~\ref{t:2gengdis6}. 

 We now prove statement~\ref{t:2gengdis7}.  
Suppose $J(h_{1})<_{s}J(h_{2}).$ Then $J_{\min }(G)=J(h_{1})$ and 
$J_{\max }(G)=J(h_{2}).$ 
By \cite[Theorem 2.20-5]{SdpbpI}, 
we obtain $h_{2}(J(h_{1})) \subset K(h_{1})$. 
Therefore $\hat{K}(G)=K(h_{1}).$ 
Thus $K(h_{1})\subset T_{\infty ,\tau }^{-1}(\{ 0\} ).$ 
Moreover, for any $y\in F_{\infty }(h_{2})$, 
there exists an element $g\in G$ with $g(y)\in F_{\infty }(G).$ Therefore 
$T_{\infty ,\tau }(y)>0.$ It follows that 
$T_{\infty ,\tau }^{-1}(\{ 0\} )=K(h_{1}).$ 
Since $J_{\max }(G)=J(h_{2})$, 
$F_{\infty }(G)=F_{\infty }(h_{2}).$ 
Since $T_{\infty ,\tau }:\CCI \rightarrow [0,1]$ is continuous 
(see Theorem~\ref{randomthm1}-\ref{randomthm1c}),  
$\overline{F_{\infty }(h_{2})}\subset T_{\infty ,\tau }^{-1}(\{ 1\}).$ 
By \cite[Theorem 2.20-5]{SdpbpI}, 
int$(K(h_{2}))$ is connected, int$(K(h_{2}))$ is the immediate basin of 
an attracting fixed point $a$ of $h_{2}$, and $a\in \mbox{int}(\hat{K}(G)).$ 
Therefore, for any $z\in \mbox{int}(K(h_{2}))$, there exists an element 
$h\in G$ such that $h(z)\in \hat{K}(G).$ Thus $T_{\infty ,\tau }(z)<1$ 
for any $z\in \mbox{int}(K(h_{2})).$  
Hence, $T_{\infty ,\tau }^{-1}(\{ 1\} )=\overline{F_{\infty }(h_{2})}.$ 
We now let $w=(w_{1},w_{2},\ldots )\in \{ 1,2\} ^{\NN }.$ 
We first consider  the case 
\begin{equation}
\label{eq:wn12i}
\sharp \{ n\in \NN \mid w_{n}=1\} =\sharp \{ n\in \NN \mid w_{n}=2\} =\infty .
\end{equation}
The following claim follows from \cite[Theorem 3.11(2)]{SdpbpII}.\\ 
Claim 3. There exists exactly one bounded component $B_{w}$ of $F_{\g (w)}.$ 
Moreover, $\partial B_{w}=\partial A_{\infty ,\gamma (w)}=J_{\gamma (w)}.$

By (\ref{eq:wn12i}), 
there exists a sequence $\{ \lambda ^{n}\} _{n=1}^{\infty }$ in 
$\{ 1,2\} ^{\NN }$ such that 
$\lambda ^{1}<_{l}\lambda ^{2}<_{l}\cdots <_{l}w$ and $\lambda ^{n}\rightarrow 
w $ as $n\rightarrow \infty .$ By statements~\ref{t:2gengdis2}, \ref{t:2gengdis6}, 
it follows that 
$J_{\g (\lambda ^{1})}<_{s}J_{\g (\lambda ^{2})}<_{s}\cdots <_{s}J_{\gamma (w)}$ and 
$J_{\g (\lambda ^{n})}\rightarrow J_{\g (w)}$ as $n\rightarrow \infty $ 
with respect to the Hausdorff metric. 
Combining this with \cite[Lemma 4.4]{SdpbpI}, 
Theorem \ref{randomthm1}-\ref{randomthm1m} and Lemmas~\ref{l:fmcmono}, \ref{l:y1y2},   
we obtain that for each $y$ in the bounded connected component of $\CCI \setminus J_{\g (w)}$, 
$T_{\infty ,\tau }(y)<T_{\infty ,\tau }|_{J_{\g (w)}}.$ 
Similarly, we can obtain that for each $y$ in the unbounded connected component of $\CCI \setminus J_{\g (w)}$, 
$T_{\infty ,\tau }(y)>T_{\infty ,\tau }|_{J_{\g (w)}}.$ 
Therefore letting $t:=T_{\infty ,\tau }|_{J_{\g (w)}}\in (0,1)$, 
$T_{\infty ,\tau }^{-1}(\{ t\} )=J_{\g (w)}.$ 

 We now consider the case 
\begin{equation}
\label{eq:wn122} 
\sharp \{ n\in \NN \mid w_{n}=1\} <\infty , w\neq (2,2,2,\ldots ).
\end{equation} 
Let $r\in \NN $ be the minimum number such that for each $n\geq r$, 
$w_{n}=2.$ Then $r\geq 2$ and $w_{r-1}=1.$ 
Let $\rho =w $ and let $\mu =(w_{1},\ldots w_{r-2},2,1,1,1,1,\ldots )\in \{ 1,2\} ^{\NN }$ 
(if $r=2$, then let $\mu =(2,1,1,1,1,\ldots )$). 
Then there exists no $\lambda \in \{ 1,2\} ^{\NN }$ with $\rho <_{l}\lambda <_{l}\mu .$ 
By statements~\ref{t:2gengdis4}, \ref{t:2gengdis6} and 
Theorem~\ref{randomthm1}-\ref{randomthm1c}, we obtain that 
 there exists a doubly connected component $A$ of $F(G)$ with 
$\partial A\subset J_{\g (\rho )}\cup J_{\g (\mu )}$, 
and that there exists a $t\in (0,1)$ with 
$T_{\infty ,\tau }|_{K_{\g (\mu )}\setminus \mbox{int}(K_{\rho })}=t.$ 
Moreover, since $(h_{w_{r-1}}\cdots h_{w_{1}})^{-1}(J(h_{2}))=J_{\g (\rho )}$,  
since $J(h_{2})$ is a quasicircle (\cite[Theorem 2.20-4]{SdpbpI}), 
and since $P^{\ast }(G)\subset \mbox{int}(K(h_{2}))$, 
we obtain that $J_{\g (\rho )}$ is a quasicircle. 
For the element $\rho $, there exists  a sequence 
$\{ \lambda ^{n}\} _{n=1}^{\infty }$ in 
$\{ 1,2\} ^{\NN }$ such that 
$\lambda ^{1}<_{l}\lambda ^{2}<_{l}\cdots <_{l}\rho $ and $\lambda ^{n}\rightarrow 
\rho  $ as $n\rightarrow \infty .$ By statements~\ref{t:2gengdis2}, \ref{t:2gengdis6}, 
it follows that 
$J_{\g (\lambda ^{1})}<_{s}J_{\g (\lambda ^{2})}<_{s}\cdots $ and 
$J_{\g (\lambda ^{n})}\rightarrow J_{\g (\rho )}$ as $n\rightarrow \infty $ 
with respect to the Hausdorff metric. 
Combining this with \cite[Lemma 4.4]{SdpbpI}, 
Theorem \ref{randomthm1}-\ref{randomthm1m} and Lemmas~\ref{l:fmcmono}, \ref{l:y1y2},   
we obtain that for each $y$ in the bounded connected component of $\CCI \setminus J_{\g (\rho )}$, 
$T_{\infty ,\tau }(y)<T_{\infty ,\tau }|_{J_{\g (\rho )}}.$ 
Similarly, we can obtain that for each $y$ in the unbounded connected component of $\CCI \setminus J_{\g (\mu )}$, 
$T_{\infty ,\tau }(y)>T_{\infty ,\tau }|_{J_{\gamma (\mu )}}=T_{\infty ,\tau }|_{J_{\g (w)}}.$ 
Therefore 
%letting $t:=T_{\infty ,\tau }|_{J_{\g (w)}}\in (0,1)$, 
$T_{\infty ,\tau }^{-1}(\{ t\} )=K_{\g (\mu )}\setminus \mbox{int}(K_{\g (\rho )}).$ 
From these arguments, statement~\ref{t:2gengdis7} follows. 

 Thus, we have proved Theorem~\ref{t:2gengdis}. 
\qed 
\vspace{-2mm} 
\subsection{Proofs of Theorem~\ref{t:3genmain} and Corollary~\ref{c:3gengdis}}
\vspace{-2mm} 
In this subsection, we prove 
Theorem~\ref{t:3genmain} and Corollary~\ref{c:3gengdis}.\\ 
{\bf Proof of Theorem~\ref{t:3genmain}:} 
Since $G\in {\cal G}_{dis}$, 
by \cite[Theorem 1.7, Theorem 1.5]{S15} there exists a number $k\in \{ 1,2,3\} $ 
such that 
\begin{equation}
\label{eq:kisol}
h_{k}^{-1}(J(G))\cap  h_{j}^{-1}(J(G)))=\emptyset \mbox{ for each }j \mbox{ with }j\neq k.
\end{equation}
We set $J_{\min }=J_{\min }(G)$ and $J_{\max }=J_{\max }(G).$ 
By \cite[Proposition 2.24]{SdpbpI}, 
we have $J_{\min }=J_{1}$ and $J_{\max }=J_{3}.$
We show the following claim.\\ 
Claim 1. $h_{1}^{-1}(J(G))\cap h_{3}^{-1}(J(G))=\emptyset .$ 
 
To prove this claim, we  consider the following three cases (i),(ii),(iii). 
(i) $J_{1}=J_{2}$. (ii) $J_{2}=J_{3}$. (iii) $J_{1}<_{s}J_{2}<_{s}J_{3}.$ 

 Suppose we have case (i). Since $J(G)=\bigcup _{j=1}^{3}h_{j}^{-1}(J(G))$ (\cite[Lemma 2.4]{S3}), 
we have $J_{\min }=\bigcup _{j=1}^{3}(J_{\min }\cap h_{j}^{-1}(J(G))).$ 
Since $J(h_{3})\subset J_{\max }\subset \CC \setminus J_{\min }$, 
by \cite[Theorem 2.20-5(b)]{SdpbpI} 
%since $h_{3}^{-1}(J_{\min })$ is connected (\cite[Theorem 2.7-3]{SdpbpI}) 
%and since $\sharp J_{\min }\geq 3$, 
%we get that $h_{3}^{-1}(J_{\min })\cap J_{\min }=\emptyset .$ 
%Combining it with \cite[Theorem 2.7]{SdpbpI}, 
we obtain that $J_{\min }\cap h_{3}^{-1}(J(G))=\emptyset .$ 
Therefore $J_{\min }=\bigcup  _{j=1}^{2}(J_{\min }\cap h_{j}^{-1}(J(G))).$
Moreover, since $J_{1}=J_{2}=J_{\min }$, 
and since $h_{j}^{-1}(J_{\min })$ is connected for each $j=1,2$ (\cite[Theorem 2.7]{SdpbpI}), 
we have 
that $J_{\min }\cap h_{j}^{-1}(J(G))\supset h_{j}^{-1}(J_{\min })\neq \emptyset $ for each $j=1,2.$ 
Since $J_{\min }$ is connected, it follows that 
$\bigcap _{j=1}^{2}(J_{\min }\cap h_{j}^{-1}(J(G)))\neq \emptyset .$ 
In particular $h_{1}^{-1}(J(G))\cap h_{2}^{-1}(J(G))\neq \emptyset .$   
 By (\ref{eq:kisol}), it follows that 
 $h_{3}^{-1}(J(G))\cap (\bigcup _{j=1}^{2}h_{j}^{-1}(J(G)))=\emptyset .$ 
%From these arguments, it follows that if we have case (i), then statement (3) in our theorem holds. 

We now suppose we have case (ii). 
By the arguments similar to those in case (i), 
we obtain that $h_{2}^{-1}(J(G))\cap h_{3}^{-1}(J(G))\neq \emptyset $ and  
 $h_{1}^{-1}(J(G))\cap (\bigcup _{j=2,3}h_{j}^{-1}(J(G)))=\emptyset $. 
%and 
%From these arguments,  it follows that if we have case (ii), 
%then statement (2) in our theorem holds. 

 We now suppose that we have case (iii). 
Then by \cite[Corollary 3.7]{SdpbpIII}, 
$h_{j}^{-1}(J(h_{1}))$ is connected for each $j=2,3.$ 
Moreover, since $J(h_{j})\cap J_{\min } =\emptyset $ for each $j=2,3$ 
and $\sharp J_{\min }\geq 2$ (\cite[Theorem 2.20-5(b)]{SdpbpI}),   
we obtain that $h_{j}^{-1}(J(h_{1}))\cap J(h_{1})=\emptyset $ for each $j=2,3.$ 
By \cite[Lemma 3.9]{SdpbpIII}, 
it follows that $J(h_{1})<_{s}h_{j}^{-1}(J(h_{1}))$ for each $j=2,3.$ 
In particular, $h_{j}(K(h_{1}))\subset \mbox{int}(K(h_{1}))$ for each $j=2,3.$ 
Therefore, $\hat{K}(G)=K(h_{1}).$ 
Similarly, we obtain that 
for each $i=1,2,$ 
 $h_{i}^{-1}(J(h_{3}))$ is connected, 
 $h_{i}^{-1}(J(h_{3}))<_{s}J(h_{3})$ and  
 $h_{i}(F_{\infty }(h_{3}))\subset F_{\infty }(h_{3}).$ 
 Therefore $F_{\infty }(G)=F_{\infty }(h_{3}).$ 
 Let $A:=K(h_{3})\setminus \mbox{int}(K(h_{1})).$ 
 From the above arguments, $\bigcup _{j=1}^{3}h_{j}^{-1}(A)\subset A.$ 
 Therefore by \cite[Corollary 3.2]{HM}, $J(G)\subset A.$ 
Moreover, since $J_{1}\neq J_{3}$, 
 $\langle h_{1},h_{3}\rangle \in {\cal G}_{dis}.$ 
By Claim 2 in the proof of Theorem~\ref{t:2gengdis}, 
$h_{1}^{-1}(A)\cap h_{3}^{-1}(A)=\emptyset .$ 
Hence, it follows that $h_{1}^{-1}(J(G))\cap h_{3}^{-1}(J(G))=\emptyset .$   
 
 Thus we have proved Claim 1. 
 
By Claim 1 and (\ref{eq:kisol}), 
we obtain that exactly one of the following (I), (II), (III) holds. 
(I) $\{ h_{i}^{-1}(J(G))\} _{i=1,2,3}$ are mutually disjoint. 
(II) $h_{1}^{-1}(J(G))\cap (\bigcup _{j=2,3}h_{j}^{-1}(J(G)))=\emptyset $ and 
$h_{2}^{-1}(J(G))\cap h_{3}^{-1}(J(G))\neq \emptyset .$ 
(III)  $h_{3}^{-1}(J(G))\cap (\bigcup _{j=1,2}h_{j}^{-1}(J(G)))=\emptyset $ and 
$h_{1}^{-1}(J(G))\cap h_{2}^{-1}(J(G))\neq \emptyset .$ 
 
 Suppose we have Case (I). Then by Proposition~\ref{p:mgengdis}-\ref{p:mgengdis2}, 
 %Lemma~\ref{l:mgdiscg}, 
 $J_{\min }=J(h_{1})$ and $J_{\max }=J(h_{3}).$ 
 Hence $F_{\infty }(G)=F_{\infty }(h_{3}).$ 
 By \cite[Theorem 2.20-5]{SdpbpI}, 
 $h_{j}(J(h_{1}))\subset \mbox{int}(\hat{K}(G))\subset \mbox{int}(K(h_{1}))$ for each 
 $j=2,3.$ 
 Therefore $\hat{K}(G)=K(h_{1}).$ Thus statement (1) of our theorem holds. 
 
 Suppose we have Case (III). 
Since $h_{3}^{-1}(J(G))\cap (\bigcup _{j=1}^{2}h_{j}^{-1}(J(G)))=\emptyset $,  
by \cite[Lemma 4.13-4]{SdpbpI} and \cite[Lemmas 3.5, 3.6]{SdpbpIII} we obtain that  
 $\bigcap _{n=1}^{\infty }h_{3}^{-n}(J(G))$ is a connected component of $J(G).$ 
 Since $J(h_{3})\cap J_{\min }=\emptyset $, by \cite[Theorem 2.20-4,5]{SdpbpI} 
 int$(K(h_{3}))$ is connected and there exists an attracting fixed point $z_{0}$ of 
$h_{3}$ in int$(\hat{K}(G))$ such that int$(K(h_{3}))$ is the immediate basin of $z_{0} $ 
for the dynamics of $h_{3}.$ Therefore $\bigcap _{n=1}^{\infty }h_{3}^{-n}(J(G))=J(h_{3}). $ 
Since $J(h_{3})\subset J_{\max }$, we obtain that $J(h_{3})=J_{\max}.$ 
Therefore $F_{\infty }(G)=F_{\infty }(h_{3}).$ 
 Thus statement (3) of our theorem holds. 
 
 Suppose we have Case (II). By the arguments similar to those in Case (III), 
we obtain that  $\bigcap _{n=1}^{\infty }h_{1}^{-n}(J(G))$ is a connected component of $J(G).$ 
Since $J(h_{1})\subset J_{\min }\cap \bigcap _{n=1}^{\infty }h_{1}^{-n}(J(G))$, 
it follows that $J_{\min }=\bigcap _{n=1}^{\infty }h_{1}^{-n}(J(G))\subset K(h_{1}).$ 
Moreover, since $(J(h_{2})\cup J(h_{3}))\cap J_{\min }=\emptyset $, 
by \cite[Theorem 2.20-5]{SdpbpI} 
we obtain that $h_{j}(J(h_{1}))\subset \mbox{int}(\hat{K}(G))\subset \mbox{int}(K(h_{1}))$ for each 
$j=2,3.$ 
Hence $K(h_{1})=\hat{K}(G)$ and 
$\mbox{int}(K(h_{1}))\subset \mbox{int}(\hat{K}(G))\subset F(G).$ 
Therefore $J_{\min }=J(h_{1}).$ 
Thus statement (2) of our theorem holds. 
 
 Combining all of the above arguments, we obtain that 
 (a) if $J_{1}=J_{2}$, then statement (3) of our theorem holds, and  
 (b) if $J_{2}=J_{3}$, then statement (2) of our theorem holds. 
 We now suppose $h_{2}^{-1}(J(G))\cap (\bigcup _{j=1,3}h_{j}^{-1}(J(G)))=\emptyset .$ Then 
 by Claim 1, Case (I) holds. Therefore statement (1) of our theorem holds. 
 Thus we have proved Theorem~\ref{t:3genmain}.
\qed

We now prove Corollary~\ref{c:3gengdis}. 

{\bf Proof of Corollary \ref{c:3gengdis}:} 
%By \cite[Theorem 3.17]{S15}, 
%there are infinitely many connected components of $J(G).$ 
%Combining that with \cite[Theorem 2.7-1, Lemma 4.4]{SdpbpI}, 
%we obtain that there are infinitely many doubly connected components of $F(G).$ 
By Theorem~\ref{t:3genmain}, 
there exists a number $i\in \{ 1,2,3\} $ such that 
$h_{i}^{-1}(J(G))\cap (\bigcup _{j:j\neq i}h_{j}^{-1}(J(G)))=\emptyset $ and 
either $J(h_{i})=J_{\max }(G)$ or $J(h_{i})=J_{\min }(G).$ 

Suppose $J(h_{i})=J_{\min }(G).$ Let $j\in \{ 1,2,3\} $ be an element with $j\neq i.$  
By \cite[Proposition 2.2(3)]{S7}, 
for each $z\in J(h_{i})$,  
$d(z,J(h_{j}h_{i}^{k}))\rightarrow 0$ as $k\rightarrow \infty .$ 
%Since $J(h_{i})<_{s}J(h_{j}h_{i}^{k})$ for each $k$, 
%we obtain that $J(h_{j}h_{i}^{k})\rightarrow J(h_{i})$ as $k\rightarrow \infty $ 
%with respect to the Hausdorff metric. 
For each $k$, let $I_{k}\in \mbox{Con}(J(G))$ with $J(h_{j}h_{i}^{k})\subset I_{k}.$ 
Then 
by the compactness of the space of all non-empty connected compact subsets of $\CCI $ with respect to 
the Hausdorff metric, 
we obtain that $I_{k}\rightarrow J(h_{i})$ as $k\rightarrow \infty $ with respect to 
the Hausdorff metric. 
Moreover, for each $k$, we have $I_{k}\neq J_{\min }(G)$ since 
$I_{k}\subset h_{j}^{-1}(J(G))$ and $J_{\min }(G)\subset h_{i}^{-1}(J(G)).$ 
Let $\{ J_{n}\} _{n=1}^{\infty }$ be a subsequence of 
$\{ I_{k}\}$ such that 
$J_{1}>_{s}J_{2}>_{s}\cdots >_{s}J(h_{i})$ and $J_{n}\rightarrow J(h_{i})$ as $n\rightarrow \infty .$ 
By \cite[Lemma 4.4]{SdpbpI}, for each $n$ 
 there exists a doubly connected component $A_{n}$ of $F(G)$ with 
$J_{n}>_{s}A_{n}>_{s}J_{n+1}.$ 
Then $\overline{A_{n}}\rightarrow J(h_{i})$ as $n\rightarrow \infty .$ 

 Suppose $J(h_{i})=J_{\max }.$ 
%Then by \cite[Theorem 2.20-4,5]{SdpbpI}, int$(K(h_{i}))$ is connected and it is the immediate basin of 
%an attracting  fixed point $z_{0} $ of $h_{i}$ with respect to the dynamics of $h_{i}$ and $z_{0}\in %\mbox{int}(\hat{K}(G)).$ 
%Combining it with 
By the arguments similar to those in the previous paragraph, we obtain that 
there exists a sequence $\{ J_{n}\} $ of mutually distinct elements in Con$(J(G))$ and 
a sequence $\{ A_{n}\} $ of mutually distinct doubly connected components of $F(G)$ such that 
$J_{n}\rightarrow J(h_{i})$ and $\overline{A_{n}}\rightarrow J(h_{i})$ as $n\rightarrow \infty $ 
with respect to the Hausdorff metric. 
 Thus we have proved Corollary~\ref{c:3gengdis}.   
\qed 
\vspace{-4mm}
\section{Examples}
\label{Examples} 
\vspace{-3mm} 
In this section we give some examples. 
\begin{df}
Let $G$ be a polynomial semigroup. We say that 
$G$ is semi-hyperbolic if 
there exists an $N\in \NN $ and a $\delta >0$ such that 
for each $z\in J(G)$ and for each $g\in G$, 
$\deg (g:V\rightarrow B(z,\delta ))\leq N$ for each 
$V\in \mbox{Con}(g^{-1}(B(z,\delta ))).$ Here, 
$\deg $ denotes the degree of finite branched covering. 
We say that $G$ is hyperbolic if $P(G)\subset F(G).$ 
\end{df}
\begin{prop}[Proposition 2.40 in \cite{SdpbpI}]
\label{Constprop}
Let $G$ be a 
polynomial semigroup generated by 
a compact subset $\G $ of ${\cal P}.$ 
Suppose that $G\in {\cal G}$ and  
{\em int}$(\hat{K}(G))\neq \emptyset .$ 
Let $b\in $ {\em int}$(\hat{K}(G)).$ 
%be a point. 
Moreover, let $d\in \NN $ be any positive integer such that 
$d\geq 2$, and such that 
$(d, \deg (h))\neq (2,2)$ for each $h\in \G .$ 
Then, there exists a number $c>0$ such that 
for each $a\in \CC $ with $0<|a|<c$, 
there exists a compact neighborhood $V$ of 
$g_{a}(z)=a(z-b)^{d}+b$ in ${\cal P}$ 
satisfying 
that for any non-empty subset $V'$ of $V$,  
the polynomial semigroup 
$\langle \G\cup V'\rangle $ generated by the family $\G \cup V'$ 
belongs to ${\cal G}_{dis}$ and $\hat{K}(\langle \G\cup V'\rangle)=\hat{K}(G)$.  
% and $(\G \cup V')_{\min }\subset \G .$ 
Moreover, in addition to the assumption above, 
if $G$ is semi-hyperbolic (resp. hyperbolic), 
then the above $\langle \G\cup V'\rangle $ is semi-hyperbolic (resp. hyperbolic).   
\end{prop}
\begin{prop}[Proposition 6.1 in \cite{Splms10}] 
\label{semihyposcexprop}
%(See \cite{hiroki5}) 
Let $h_{1}\in {\cal P}.$ 
%be a polynomial 
%with $\deg (f_{1})\geq 2$. 
%such that $J(f_{1})$ is connected. 
%Let $K(f_{1})$ be the filled-in Julia set of $f_{1}$ and 
Suppose that $K(h_{1})$ is connected and {\em int}$(K(h_{1}))$ is not empty. 
Let $b\in \mbox{{\em int}}(K(h_{1}))$ be a point. 
Let $d$ be a positive integer such that 
$d\geq 2.$ Suppose that $(\deg (h_{1}),d)\neq (2,2).$ 
Then, there exists a number $c>0$ such that 
for each $\l \in \{ \l\in \Bbb{C}: 0<|\l |<c\} $, 
setting $h_{\l }=(h_{\l ,1},h_{\l ,2})=
(h_{1},\l (z-b)^{d}+b )$ and $G_{\l }:= \langle h_{1},h_{\l, 2}\rangle $, 
 we have all of the following.
\begin{itemize}
\item[{\em (a)}] 
$G_{\lambda }\in {\cal G}_{dis}.$ 
Moreover, 
$h_{\l }$ satisfies the open set condition with 
an open subset $U_{\l }$ of $\CCI $ (i.e., $h_{\l ,1}^{-1}(U_{\l })\cup h_{\l, 2}^{-1}(U_{\l })\subset U_{\l }$ and 
$h_{\l ,1}^{-1}(U_{\l })\cap h_{\l ,2}^{-1}(U_{\l })=\emptyset $), 
$h_{\l ,1}^{-1}(J(G_{\l }))\cap h_{\l, 2}^{-1}(J(G_{\l }))=\emptyset $, 
{\em int}$(J(G_{\l }))=\emptyset $, 
$J_{\ker }(G_{\l })=\emptyset $, 
$G_{\l }(K(h_{1}))\subset K(h_{1})\subset \mbox{{\em int}}(K(h_{\lambda ,2}))$  
and 
$\emptyset \neq K(h_{1})\subset \hat{K}(G_{\l }).$ 
%\item[{\em (b)}]
%, and $\hat{K}(G_{\lambda })\neq \emptyset $.
\item[{\em (b)}]
If $h_{1}$ is semi-hyperbolic (resp. hyperbolic),
% and $K(f_{1})$ is connected, 
then 
$G_{\l }$ is semi-hyperbolic (resp. hyperbolic), 
$J(G_{\l } )$ is porous (for the definition of porosity, see \cite{S7}),  and 
$\dim _{H}(J(G_{\l }))<2$.  
\end{itemize} 
\end{prop}
For the dynamics of (semi-)hyperbolic rational semigroups,  
see \cite{S1, S4,S6,S7,SdpbpI,SdpbpII,SdpbpIII, SU2,SU4}. For the study of the Hausdorff dimension of the Julia sets 
of (semi-)hyperbolic rational semigroups (with open set condition),  
see \cite{S6, S7, SU2,SU4}.  

 Regarding Proposition~\ref{semihyposcexprop}, 
 we can sometimes give the concrete values $c.$ 
\vspace{-1mm}
\begin{ex}[Devil's coliseum] 
\label{ex:dc1}
Let $h_{1}(z)=z^{2}-1$ and let $\lambda \in \CC $ with $0<|\lambda |\leq 0.01. $ 
Let $h_{2}(z)=\lambda z^{3}.$ 
Let $G=\langle h_{1},h_{2}\rangle $ and $\tau := \sum _{i=1}^{2}\frac{1}{2}\delta _{h_{i}}.$ 
Let $A:= K(h_{2})\setminus B$ 
where $B=D(0,0.4)\cup D(-1, 0.16)$.   
Then we have $\overline{B}\subset \mbox{int}(K(h_{1}))\subset D(0,2)$ and 
$h_{2}(K(h_{1}))\subset h_{2}(D(0,2))\subset \overline{B}\subset \mbox{int}(K(h_{1}))$. 
Therefore  
$P^{\ast }(G)\subset \mbox{int}(K(h_{1}))$. 
Hence $G$ is hyperbolic and $G\in {\cal G}.$  
Moreover, we have  
$h_{1}(B)\subset B$, $h_{2}(B)\subset B$, 
$K(h_{2})=\overline{D(0,|\lambda |^{-1/2})}$,
 $h_{1}^{-1}(K(h_{2}))\subset K(h_{2})$, 
 and $h_{2}^{-1}(K(h_{2}))\subset K(h_{2}).$ Hence  
$h_{1}^{-1}(A)\cup h_{2}^{-1}(A)\subset A$. 
Also, it is easy to see that $h_{1}^{-1}(A)\cap h_{2}^{-1}(A)=\emptyset .$ 
Therefore $J(G)\subset A$, 
$h_{1}^{-1}(J(G))\cap h_{2}^{-1}(J(G))=\emptyset $, $G\in {\cal G}_{dis}$ and $\emptyset \neq K(h_{1})\subset \hat{K}(G).$ 
 %Moreover, by \cite[Example 6.2]{Splms10}, we obtain that $G$ is hyperbolic. 
By Theorems~\ref{randomthm1} and \ref{t:hnondiffp}, 
we obtain that 
$J_{\ker}(G)=\emptyset $,  
%By Theorem~\ref{kerJthm2} and Lemma~\ref{l:lsncnonc}, 
%Proposition~\ref{p:jnonconst1}, 
$T_{\infty ,\tau }$ is \Hol der continuous on $\CCI $, the set of varying points of 
$T_{\infty ,\tau }$ is equal to $J(G),$ 
%Moreover, by Theorem~\ref{t:hnondiff}, 
%$\dim _{H}(J(G))<2$ 
and for each non-empty open subset $U$ of $J(G)$ there exists an uncountable dense 
subset $A_{U}$ of $U$ such that for each $z\in A_{U}$, 
$T_{\infty ,\tau }$ is not differentiable at $z.$ 
By Theorem~\ref{t:hnondiffp} and \cite[Theorem 3.82]{Splms10}, 
there exists a Borel subset $A$ of $J(G)$ with $\dim _{H}(A)\geq 
1+\frac{2\log 2}{\log 2+\log 3}\fallingdotseq 1.7737$ 
such that for each $z\in A$, 
$\Hol(T_{\infty ,\tau },z)= u(h,p,\mu )=\frac{2\log 2}{\log 2+\log 3}
\fallingdotseq 0.7737$ and 
$T_{\infty ,\tau }$ is not differentiable at $z.$  
Moreover, since $G$ is hyperbolic and $h_{1}^{-1}(J(G))\cap h_{2}^{-1}(J(G))=\emptyset $, 
$\dim _{H}(J(G))<2$ (see \cite{S2} or \cite[Theorem 3.82]{Splms10}).  
It is easy to see that $\Min (G_{\tau },\CCI )=\{ \{ \infty \} ,\{ 0\} \} .$ 
Thus regarding statements in Theorem~\ref{randomthm1} for $\tau $, 
$L_{\tau }=\{ 0\} $ and $\mu  _{\tau }=\delta _{0}.$  
%For the figures of $J(G_{\tau })$ and the graph of $T_{\infty ,\tau }$, see 
%\cite{Splms10}. 
%See Figures~\ref{fig:dcjulia}, \ref{fig:dcgraphgrey2}, 
%and \ref{fig:dcgraphudgrey2}.  
$T_{\infty ,\tau }$ is called a devil's coliseum. 
It is a complex analogue of the devil's staircase. 
\end{ex}
\begin{rem}
\label{r:sdgdis} 
By Proposition~\ref{Constprop}, 
there exists a $2$-generator polynomial semigroup 
$G=\langle h_{1},h_{2}\rangle $ in ${\cal G}_{dis}$ such that 
$h_{1}$ has a Siegel disk.  
In fact, Proposition~\ref{Constprop} implies that 
for each $h_{1}\in {\cal P}$ with $\langle h_{1}\rangle \in {\cal G}$ 
which has a Siegel disk, 
there exists an element $h_{2}\in {\cal P}$ such that 
$G=\langle h_{1},h_{2}\rangle $ belongs to ${\cal G}_{dis}.$ 
Note that for such a $G$, we can apply 
 Theorems~\ref{randomthm1}, \ref{t:hnondiffp}, \ref{t:2gengdis}, 
 even though $G$ is not semi-hyperbolic. 
\end{rem}

\begin{ex}
\label{ex:Siegeldc} 
Let $\theta \in \RR $ be a Brjuno number (\cite{Br, Y}) and let 
$f_{1}(z)=e^{2\pi i \theta }z+ z^{2}$ Then $f_{1}$ has a Siegel disk with center $0$ (\cite{Br,Y}). 
Applying Proposition~\ref{Constprop} or Proposition~\ref{semihyposcexprop} (with $b=0$), 
we obtain that there exists a number $c>0$ such that 
for each $\lambda \in \CC $ with $0<|\lambda |<c$, 
setting $f_{2}(z)=\lambda z^{3}$, we have that 
$G:=\langle f_{1},f_{2}\rangle \in G_{dis }.$ 
Since $f_{1}$ is not semi-hyperbolic, $G$ is not semi-hyperbolic. 
We can apply Theorems~\ref{randomthm1}, \ref{t:hnondiffp}, \ref{t:2gengdis}, 
to this $G$, 
even though $G$ is not semi-hyperbolic. 
\end{ex}
\begin{ex}
\label{ex:mgengdis}
Let $f_{1},f_{2}\in {\cal P}$ be two elements such that $G=\langle f_{1},f_{2}\rangle $ belongs 
to ${\cal G}_{dis}.$ Then by \cite[Theorems 1.5, 1.7]{S15}, 
$f_{1}^{-1}(J(G))\cap f_{2}^{-1}(J(G))=\emptyset .$ 
Let $n\geq 2$ and let $A$ be a non-empty subset of 
$\Lambda _{n}:= \{ f_{i_{1}}\circ \cdots \circ f_{i_{n}}\mid i_{1},\ldots, i_{n}\in \{ 1,2\}\} $ 
with $\sharp A\geq 2.$  
Let $G_{A}$ be the polynomial semigroup generated by $A.$ 
Then $J(G_{\Lambda _{n}})=J(G)$ (\cite[Theorem 2.4]{HM}). 
Thus for each $(g_{1},g_{2})\in A^{2}$ with $g_{1}\neq g_{2}$, we have 
$g_{1}^{-1}(J(G_{A}))\cap g_{2}^{-1}(J(G_{A}))=\emptyset .$ Moreover, 
$G_{A}\in {\cal G}_{dis}.$ 
For the semigroup $G_{A}$, we can apply  Theorems~\ref{randomthm1} and \ref{t:hnondiffp}. 
If $f_{1}$ is not semi-hyperbolic and $f_{1}^{n}\in A$, then $G_{A}$ is not semi-hyperbolic.   
\end{ex}
\begin{ex}
\label{ex:3gen}
Let $f_{1},f_{2}\in {\cal P}$ be two elements such that 
$\langle f_{1},f_{2}\rangle \in {\cal G}$, $J(f_{1})\cap J(f_{2})\neq \emptyset $ and 
$0\in \mbox{int}(\hat{K}(\langle f_{1},f_{2}\rangle ))$ (e.g., $f_{1}(z)=z^{2}-1$ and $f_{2}$ is a small 
perturbation of $f_{1}$). Let $d\in \NN  $ with $d\geq 3.$ 
Then by Proposition~\ref{Constprop}, 
there exists a number $c>0$ such that for each $a\in \CC $ with $0<|a|<c$, 
setting $f_{3}(z)=az^{d}$, we have that 
$G:=\langle f_{1},f_{2},f_{3}\rangle $ belongs to ${\cal G}_{dis}.$ 
For this $G$, we can apply Theorems~\ref{randomthm1}, \ref{t:3genmain} and Corollary~\ref{c:3gengdis}. 
Since $J(f_{1})\cap J(f_{2})\neq \emptyset $, we have 
$f_{1}^{-1}(J(G))\cap f_{2}^{-1}(J(G))\neq \emptyset .$ Thus Theorem~\ref{t:3genmain} implies 
that statement (3) of Theorem~\ref{t:3genmain} holds.  
\end{ex} 

\end{document}